\documentclass[11pt,a4paper]{articleLH}

\newtheorem{theorem}{Theorem}[section]
\newtheorem{proposition}[theorem]{Proposition}
\newtheorem{corollary}[theorem]{Corollary}
\newtheorem{lemma}[theorem]{Lemma}

\theoremstyle{definition}
\newtheorem{definition}[theorem]{Definition}

\theoremstyle{remark}
\newtheorem{remark}[theorem]{Remark}

\numberwithin{equation}{section}

\newlist{enumlemma}{enumerate}{1}
\setlist[enumlemma]{label={(\alph*)}, ref=\thelemma(\alph*)}
\crefalias{enumlemmai}{lemma}

\newlist{enumlemmaresume}{enumerate}{1}
\setlist[enumlemmaresume]{label={(\alph*)}, ref=\thelemma(\alph*), start={\value{enumlemmai}+1}}
\crefalias{enumlemmaresumei}{lemma}

\newlist{enumproposition}{enumerate}{1}
\setlist[enumproposition]{label={(\alph*)}, ref=\theproposition(\alph*)}
\crefalias{enumpropositioni}{proposition}

\newlist{enumcorollary}{enumerate}{1}
\setlist[enumcorollary]{label={(\alph*)}, ref=\thecorollary(\alph*)}
\crefalias{enumcorollaryi}{corollary}

\newlist{enumcorollaryresume}{enumerate}{1}
\setlist[enumcorollaryresume]{label={(\alph*)}, ref=\thecorollary(\alph*), start={\value{enumcorollaryi}+1}}
\crefalias{enumcorollaryresumei}{corollary}

\providecommand{\N}{{\mathbb N}}
\providecommand{\R}{{\mathbb R}}
\providecommand{\id}{I}
\providecommand{\rv}{\sigma}
\providecommand{\sourcecondf}{\norm{f}_{\mu} \le R}
\providecommand{\sourcecondg}{\norm{g}_{\mu +1/2} \le R}
\providecommand{\baloracle}{\tau_{\mathfrak{b}}}
\providecommand{\predrate}{\mathcal{R}^{*, \mathrm{pred}}_{\mu,p,R}(\delta)}
\providecommand{\recrate}{\mathcal{R}^{*, \mathrm{rec}}_{\mu,p,R}(\delta)}
\providecommand{\predoracle}{t_{\mathrm{e}}^\mathrm{pred}}
\providecommand{\recoracle}{t_{\mathrm{e}}^\mathrm{rec}}
\providecommand{\eqrefAssPSDlower}{\textup{(\hyperref[AssPSD]{\textbf{PSD}($p$,$c_A$,$\infty$)})}}
\providecommand{\eqrefAssPSDupper}{\textup{(\hyperref[AssPSD]{\textbf{PSD}($p$,$0$,$C_A$)})}}

\renewcommand{\le}{\leqslant}
\renewcommand{\ge}{\geqslant}

\DeclareMathOperator{\E}{{\mathbb E}}
\DeclareMathOperator{\PP}{{\mathbb P}}
\DeclareMathOperator*{\argmin}{argmin}
\DeclareMathOperator{\Pol}{Pol}
\DeclareMathOperator{\rk}{rank}
\DeclareMathOperator{\tr}{tr}
\DeclareMathOperator{\Var}{Var}
\DeclareMathOperator{\Exp}{Exp}

\DeclarePairedDelimiter{\abs}{\lvert}{\rvert}
\newcommand{\bigabs}[1]{\abs[\big]{#1}}
\DeclarePairedDelimiter{\norm}{\lVert}{\rVert}
\newcommand{\bnorm}[1]{\norm[\Big]{#1}}
\DeclarePairedDelimiterX{\scapro}[2]{\langle}{\rangle}{{#1},{#2}}
\newcommand{\bscapro}[2]{\scapro[\Big]{#1}{#2}}
\DeclarePairedDelimiter{\floor}{\lfloor}{\rfloor}
\DeclarePairedDelimiter{\ceil}{\lceil}{\rceil}

\begin{document}

\title{Early stopping for conjugate gradients in\\ statistical inverse problems}

\author{\parbox[t]{13cm}{\centering Laura Hucker\textsuperscript{$*$} and Markus Rei\ss\textsuperscript{$\dagger$}\\[\baselineskip]
{\small Humboldt-Universit{\"a}t zu Berlin, Germany}\\
{\small \mbox{} \textsuperscript{$*$}huckerla@math.hu-berlin.de} \quad
{\small\textsuperscript{$\dagger$}mreiss@math.hu-berlin.de}}}

\date{}

\maketitle

\begin{abstract}
We consider estimators obtained by iterates of the conjugate gradient (CG) algorithm applied to the normal equation of prototypical statistical inverse problems. Stopping the CG algorithm early induces regularisation, and optimal convergence rates of prediction and reconstruction error are established in wide generality for an ideal oracle stopping time. Based on this insight, a fully data-driven early stopping rule $\tau$ is constructed, which also attains optimal rates, provided the error in estimating the noise level is not dominant.

The error analysis of CG under statistical noise is subtle due to its nonlinear dependence on the observations. We provide an explicit error decomposition and identify two terms in the prediction  error, which share important properties of classical bias and variance terms. Together with a continuous interpolation between CG iterates, this paves the way for a comprehensive error analysis of early stopping. In particular, a general oracle-type inequality is proved for the prediction error at $\tau$. For bounding the reconstruction error, a more refined probabilistic analysis, based on concentration of self-normalised Gaussian processes, is developed. The methodology also provides some new insights into early stopping for CG in deterministic inverse problems. A numerical study for standard examples shows good results in practice for early stopping at $\tau$.%

\addvspace{0.5\baselineskip}\noindent%

\addvspace{0.5\baselineskip}\noindent%
\end{abstract}


\section{Introduction} \label{SecIntroduction}

The conjugate gradient (CG) algorithm is arguably one of the computationally most efficient off-the-shelf methods for solving systems of linear equations. Like for standard gradient descent methods, stopping the algorithm \emph{early}, that is before terminating at the solution, induces a regularisation. In the context of deterministic ill-posed inverse problems the seminal work by \citet{Nem1986} and \citet{Han1995} has shown that a stopping criterion based on the discrepancy principle can lead to optimal convergence rates under bounded noise perturbations. Despite its importance in applications, the case of statistical noise, that is random perturbations, has been understood only partially for CG so far. Major reasons are the nonlinearity of the CG algorithm and that canonical Gaussian white noise $\xi$ on an infinite-dimensional Hilbert space has infinite norm $\norm{\xi}=\infty$ almost surely. Therefore, the discrepancy principle cannot be well-defined.

One way out is to modify the stopping criterion so that the noise in the residual is smoothed out. \citet{BlaMat2012} build on this idea and use Gaussian concentration together with the deterministic CG analysis to derive convergence rates, which are, however, suboptimal in terms of a-posteriori error control. In the more involved kernel learning setting a similar approach is taken by \citet{BlaKra2016}, defining the CG algorithm in the underlying energy norm, which regularises the noise. The main point there is to establish good convergence properties of the best iterate along the entire iteration path and then leave the choice of the iterate to a-posteriori model selection on all iterates. An interesting extension to time series data is provided by \citet{SinKriMun2017}. In the statistics and machine learning community CG-type methods are very popular under the name \emph{partial least squares}, see the excellent survey by \citet{RosKra2006} and the recent paper by \citet{FinKri2023}.

Our approach to conjugate gradients for statistical inverse problems is intrinsically different. We consider the standard CGNE (\emph{conjugate gradients for the normal equation}) algorithm, which when implemented is necessarily finite-dimensional, and analyse it carefully, keeping explicitly track of the underlying dimension $D$. This is in line with the approach by \citet{BlaHofRei2018EJS,BlaHofRei2018SIAM} for linear spectral methods, which has been compared to smoothed residual approaches by \citet{Sta2020}. Due to the nonlinear dependence on the noise, however, the analysis must be undertaken for every noise realisation and often requires more sophisticated arguments. This precise nonasymptotic analysis provides also for unspecified perturbations $\xi$ in the data new structural insights  for the CGNE algorithm applied to inverse problems. Throughout the text a deterministic noise and a Gaussian white noise model are considered as concrete examples. The deterministic noise results are essentially known in the literature but obtained differently and thus serve as an easy proof of concept. The  main contributions of this work, when specified to the Gaussian model, are the following:
\begin{itemize}
\item Introducing interpolation between CG iterates, we prevent overshooting and avoid the analysis of update polynomials.
\item We identify two random quantities, called \emph{stochastic error} and \emph{approximation error}, which share important properties of variance and bias and allow for a precise nonasymptotic CG error control.
\item Equilibrating stochastic and approximation error, an oracle stopping time achieves optimal prediction error control under minimal assumptions on the unknown signal and on the noise. In particular, under mild singular value conditions, the best CG iterate will be as good as the best iterate in the prototypical \emph{Showalter method}, that is standard gradient descent flow.
\item We construct a data-driven stopping rule that satisfies an oracle-type inequality for the prediction error and achieves optimal convergence rates whenever the error in estimating the noise level $\norm{\xi}^2$ is not dominant.
\item The convergence rates transfer to the reconstruction error, thus establishing minimax optimality of the best iterate along the CG iteration path and -- under restrictions on the dimension $D$ -- of the iterate selected by our early stopping criterion.
\end{itemize}

In the following \cref{SecModel}, the setting is introduced. \cref{SecOverviewMainResults} explains the main ideas and results of the paper. Afterwards, we progress in detail. Results for the interpolated CG iterates are presented in  \cref{SecInterpolConjGradientMethod}, and the prediction error is analysed in \cref{SecAnalysisPredError}. \cref{SecEarlyStopWeakNorm} derives our stopping criterion and proves an oracle-type inequality for its prediction error. The results are transferred to the reconstruction error in \cref{SecTransferReconstructionError}. The practical performance for some standard examples in \cref{SecNumericalIllustration} is in good agreement with the theory. Technical results and proofs are delegated to \cref{SecAppendix}.


\section{The setting} \label{SecModel}

We consider the classical prototypical model for statistical inverse problems, see e.g.\ \citet{Cav2011} for a survey. We observe
\begin{equation}
Y = g+\xi = Af+\xi, \label{EqModel}
\end{equation}
where $A: H_1 \to H_2$ is a known bounded linear operator between real Hilbert spaces $H_1$ and $H_2$, $f \in H_1$ is the unknown signal of interest and $\xi$ is an error term. In practice, the problem is always discretised such that we assume $H_1 = \R^P$, $H_2 = \R^D$ with $D \ge 3$ and $\xi$ to be $\R^D$-valued. Suppose that $A: \R^P \to \R^D$ is surjective and thus $\rk(A)=D \le P$, which can always be achieved by restricting the observations $Y$ to the range of $A$.  We wish to recover the minimum-norm least-squares solution $f^{\dagger} \coloneqq A^\dagger g = A^\top (AA^\top)^{-1} g$ of $g=Af$ from the observation $Y$ by applying the CG iterates as defined in \cref{SecOverviewMainResults}.

For the analysis (not the algorithm), the model \eqref{EqModel} can be transformed using the singular value decomposition (SVD) of $A$: Let $\lambda_1 \ge \dots \ge \lambda_D > 0$ be the nonzero singular values of $A$. Denote by $u_1,\dots,u_D$ and $v_1,\dots,v_P$ the corresponding left- and right-singular vectors, respectively. Then we obtain the equivalent sequence space model \cite[Eq.~(1.5)]{Cav2011}
\begin{equation*}
Y_i = g_i + \xi_i = \lambda_i f_i + \xi_i, \quad i=1,\dots,D,
\end{equation*}
with $Y_i \coloneqq \scapro{Y}{u_i}$, $g_i \coloneqq \scapro{g}{u_i}$, $f_i \coloneqq \scapro{f}{v_i}$ and $\xi_i \coloneqq \scapro{\xi}{u_i}$. Sometimes we have to consider also the $d\le D$ \emph{distinct} singular values $\tilde\lambda_1>\dots>\tilde\lambda_d>0$. Here and later $\norm{\cdot}$ always denotes the Euclidean norm  and  $\scapro{\cdot}{\cdot}$ the Euclidean inner product.

When specifying convergence rates, we will consider signals $f$ satisfying the Sobolev-type source condition
\begin{equation}
\norm{f}_{\mu} \coloneqq \Big(\sum_{i=1}^D \lambda_i^{-4\mu} f_i^2\Big)^{1/2} \le R \label{EqDefSourceCond}
\end{equation}
for $\mu>0$, $R\ge 1$. Note that $f^{\dagger}_i \coloneqq \scapro{f^{\dagger}}{v_i}=f_i$ for $i=1,\dots,D$ such that $\norm{f^{\dagger}}_{\mu} = \norm{f}_{\mu}$ and the respective source conditions agree. The assumption $\sourcecondf$ on $f$ corresponds to $\sourcecondg$ on $g=Af$.

For some results we will suppose that the inverse problem is mildly ill-posed with singular values $(\lambda_i)$ satisfying the \emph{polynomial spectral decay} assumption
\begin{equation*}
c_A i^{-p} \le \lambda_i \le C_A i^{-p}, \quad i=1,\dots,D, \tag{\textbf{PSD}($p$,$c_A$,$C_A$)} \label{AssPSD}
\end{equation*}
where $p \ge 0$ is fixed and $0 \le c_A \le C_A \le \infty$. If $C_A=\infty$ or $c_A=0$, we explicitly write \eqrefAssPSDlower{} or \eqrefAssPSDupper{}, respectively.
Throughout the paper, \mbox{(in-)}equalities up to absolute constants are denoted by `$\cong$', `$\lesssim$', `$\gtrsim$'. Further dependencies on the constants will be highlighted by additional indices.
Bounds of the type $B_1\lesssim_{c_A} B_2$, $B_3\lesssim_{C_A} B_4$ indicate that they hold true already under a one-sided condition, i.e.\ for $C_A=\infty$ and $c_A=0$, respectively.

First results will be obtained for any error $\xi$ or  under minor assumptions, see e.g.\ the minimax bounds of  \cref{ThmWeakMinimax} in conjunction with \cref{RemNoise2Moments}.
To analyse early stopping under random noise later, we often specify to the \emph{Gaussian noise model}
\begin{equation*}
\xi \coloneqq \delta Z \quad \text{with} \quad Z \sim \mathcal{N}(0,\id_D) \quad \text{and} \quad \delta \in(0,1), \tag{\textbf{GN}} \label{AssGN}
\end{equation*}
where we denote by $\id_D$ the identity matrix in $\R^{D \times D}$.
In contrast, in the classical setting of deterministic inverse problems, we study noise $\xi$ with
\begin{equation*}
\norm{Y-Af} = \norm{\xi} \le \delta \quad \text{and} \quad \delta \in(0,1). \tag{\textbf{DN}} \label{AssDN}
\end{equation*}
Alternatively, if $\xi \sim \mathcal{N}(0,\Sigma)$ with $\tr(\Sigma)\le \delta^2$, then $\E[\norm{\xi}^2] \le \delta^2$ and the results for \eqref{AssDN} will hold in a similar manner by taking expectations.

Let us comment on the relationship between our a-priori discrete theory and the potentially underlying continuous inverse problem. The CG algorithm is always applied to a discretised problem, and our analysis will provide nonasymptotic results, which are directly interpretable. In order to assess the size of different terms, we then also adopt a classical asymptotic view, where the noise level $\delta$ tends to zero and simultaneously the dimensions $D=D(\delta), P=P(\delta)$ tend to infinity. As $D(\delta), P(\delta) \to\infty$, we may approach a continuous inverse problem in terms of an operator $A$ and a function $f$ such that the source condition \eqref{EqDefSourceCond} and the spectral decay assumption \ref{AssPSD} may a-priori vary with the discretisation $D$, and we refer to \cref{RemMinimaxTruncationTimeWeak} below for concrete consequences.

To avoid special cases, we will  suppose throughout the paper that
\begin{equation*}
\sum_{j=1,\dots,D: \, \lambda_j=\lambda_i} Y_j^2 > 0, \quad i=1,\dots,D. \tag{\textbf{Y}} \label{AssY}
\end{equation*}
Under \eqref{AssGN}, this is almost surely the case. See \cref{RemSumYnonZero} for a discussion of this assumption.

Let us introduce some further notation. Applying standard functional calculus for symmetric matrices, we
define
\begin{equation*}
\norm{\mathsf{p}}_{Y} \coloneqq \norm{p( AA^\top)Y},\quad \scapro{\mathsf{p}}{\mathsf{q}}_Y \coloneqq \scapro{p(AA^\top)Y}{q(AA^\top)Y}
\end{equation*}
for polynomials $p,q$ in the set $\Pol_k$ of real polynomials with maximal degree $k$. By Assumption~\eqref{AssY}, $\norm{\mathsf{p}}_{Y}$ is indeed a norm and $\scapro{\mathsf{p}}{\mathsf{q}}_Y$ an inner product on $\Pol_k$ for $k<d$. For $k=d$ both lack positive definiteness.
Consider the linear and affine subspaces
\begin{equation*}
\Pol_{k,0} \coloneqq \big\{ p \in \Pol_k \,\big|\, p(0)=0 \big\}, \quad
\Pol_{k,1} \coloneqq \big\{ p \in \Pol_k \,\big|\, p(0)=1 \big\}.
\end{equation*}
We often abbreviate $h(AA^\top)v$ for functions $h$ and $v \in \R^D$ by $\mathsf{h}v$. By abuse of notation, we sometimes write the value of a function $h$ evaluated at $x$ instead of the function itself. For example, $\norm{\mathsf{x^{1/2}}v} = \norm{\mathsf{h}v} = \norm{(AA^\top)^{1/2}v}$ with $h(x)=x^{1/2}$, $x\ge 0$.

Denote by $\delta_{l}$ the Dirac measure in $l \in \R$ and put $\delta_{lk} \coloneqq \delta_l(\{ k \})$ for $l,k \in \R$. We set $a\wedge b=\min(a,b)$, $a\vee b=\max(a,b)$ and $a_+=a\vee 0$ for $a,b\in\R$.


\section{Overview of main results} \label{SecOverviewMainResults}

The original CG method \cite{HesSti1952} is used for numerically solving systems of linear equations, where the coefficient matrix is symmetric and positive (semi-)definite. This is not necessarily the case in our model. Following the standard technique \cite[Section~2.3]{Han1995}, we apply the iterative CG algorithm to the normal equation $A^\top A \widehat f = A^\top Y$ \cite[Algorithm~7.4.1]{Bjo1996} and choose the initial value $\widehat f_0=0$. The iterates $\widehat f_k$, for $k=0,\dots,d$ or up to some stopping criterion, yield our estimators of $f$.
For the theoretical analysis in this paper, the estimators are defined via minimising \emph{residual polynomials}, which emerge from the theory of Krylov subspace methods \cite[Section~2]{Han1995}.

\begin{definition} \label{DefEstimators}
At iteration $k = 0,\dots,d$, the \emph{prediction estimator} $\widehat{g}_k$ of $g = Af$ and the \emph{CG estimator} $\widehat{f}_k$ of the minimum-norm solution $f^\dagger$ of $g=Af$  are given by
\begin{equation*}
\widehat{g}_k \coloneqq (\id_D -r_k(AA^\top))Y \quad \text{and} \quad \widehat{f}_k \coloneqq A^{\dagger}\widehat{g}_k,
\end{equation*}
where the \emph{residual polynomial} $r_k$ is obtained as
\begin{equation*}
r_k \coloneqq \argmin_{p_k \in \Pol_{k,1}} \norm{p_k(AA^\top)Y}^2 = \argmin_{p_k \in \Pol_{k,1}} \norm{\mathsf{p_k}}_Y^2.
\end{equation*}
\end{definition}

In the case $k=d$, by \cref{LemPropInterpolResidualPolynomials:simpleZeros:noninterpol}, $r_d$ maps all squared nonzero singular values of $A$ to zero implying $\widehat{f}_d = A^\dagger Y$. Thus, the CG estimator at iteration~$d$ is a minimum-norm solution of \eqref{EqModel} and  $d$ is the maximal iteration index.

One new idea is to allow also for noninteger iteration indices $t \in [0,d]$ by interpolating linearly between the residual polynomials and thus also between the estimators (see \cref{DefInterpolEstimators}). The time-continuous residual polynomials have useful properties, which are discussed in \cref{SecInterpolConjGradientMethod} in detail.

We start by considering the prediction error  $\norm{A(\widehat{f}_t-f)}^2 = \norm{\widehat{g}_t-g}^2$, i.e.\ the error of the prediction estimator $\widehat{g}_t=A\widehat{f}_t$. The prediction error is the intrinsic metric for the minimisation property of CG as well as for a-posteriori error control based on the residuals. Classically, interpolation  often allows to transfer prediction results to the reconstruction (or estimation) error $\norm{\widehat{f}_t-f^{\dagger}}^2$ \cite[Section~4.3]{EngHanNeu1996}. While in  deterministic situations the prediction error is sometimes also analysed for its own sake \cite[Section~4.5]{Han2010}, optimal regularisation is essential for prediction in statistical inverse problems and there is  a strong genuine interest in the analysis of the prediction error.

Since the estimator $\widehat{g}_t$ depends nonlinearly on the observation $Y$, we cannot calculate a standard bias-variance decomposition, but need to analyse the prediction error for any realisation of $\xi$. In \cref{PropWeakLoss}, we are able to decompose the error nonasymptotically and obtain for $t \in [0,d]$ the bound
\begin{equation}
\norm{A(\widehat{f}_t-f)}^2 \le 2(A_{t,\lambda}+S_{t,\lambda}), \label{EqWeakErrorBoundOverview}
\end{equation}
where $S_{t,\lambda}$ and $A_{t,\lambda}$ are defined in \cref{DefErrorTerms}. The subscript $\lambda$ indicates that the terms correspond to the prediction error, which involves the singular values of $A$, in line with the notation in \cite{BlaHofRei2018EJS,BlaHofRei2018SIAM}. As shown in \cref{LemPropErrorTerms}, $S_{t,\lambda}$ and $A_{t,\lambda}$ share important properties of stochastic and approximation errors. In particular, the \emph{stochastic error term} $S_{t,\lambda}$ is continuously increasing from $S_{0,\lambda} = 0$ to $S_{d,\lambda} = \norm{\xi}^2$. The \emph{approximation error term} $A_{t,\lambda}$ is also continuous and has a relatively tight upper bound, which is  decreasing from $A_{0,\lambda}=\norm{Af}^2$ to $A_{d,\lambda}=0$. These properties are the key for our analysis.

Given the monotonicity properties of $S_{t,\lambda}$ and of the upper bound of $A_{t,\lambda}$, it is natural to consider the following sequential \emph{oracle} index.

\begin{definition} \label{DefBalancedOracle}
The (data-dependent) {\em balanced oracle} index is given by
\begin{equation*}
\baloracle \coloneqq \inf\{t\in[0,d]\,|\,A_{t,\lambda}\le S_{t,\lambda}\}.
\end{equation*}
\end{definition}

Due to the continuity of the interpolations $A_{t,\lambda}$ and $S_{t,\lambda}$ in $t$ and $A_{0,\lambda} \ge S_{0,\lambda}$, $A_{d,\lambda} \le S_{d,\lambda}$, the index $\baloracle \in [0,d]$ exists and the errors are indeed \emph{balanced} at $\baloracle$, meaning that $A_{\baloracle,\lambda}=S_{\baloracle,\lambda}$. 

The choice is motivated by the fact that the bound $2(A_{\baloracle,\lambda} + S_{\baloracle,\lambda})$ of the prediction error at the balanced oracle is by the monotonicity of $t \mapsto S_{t,\lambda}$ at most twice as large as the minimal prediction error bound:
\begin{equation*}
A_{\baloracle,\lambda}+S_{\baloracle,\lambda} = 2 S_{\baloracle,\lambda} \le 2 \inf_{t \in [0,d]} (A_{t,\lambda} \vee S_{t,\lambda}) \le 2\inf_{t \in [0,d]} (A_{t,\lambda}+S_{t,\lambda}).
\end{equation*}
The balanced oracle $\baloracle$ serves as a theoretical benchmark, not accessible to the user, that we shall try to mimic in a data-driven way, see \citet[Section~1.8]{Tsy2009} for a discussion of the oracle concept.

Assume \eqref{AssPSD} with $p>1/2$ and Gaussian noise \eqref{AssGN}. In \cref{ThmWeakMinimax}, we show that the expected stochastic error $\E[S_{\baloracle,\lambda}]$ and the prediction estimator $\widehat{g}_{\baloracle}$ at the balanced oracle attain the risk (mean squared error) rate
\begin{equation}
\predrate \coloneqq R^{2/(4\mu p +2p + 1)} \delta^{(8\mu p + 4p)/(4 \mu p +2p+1)} \label{EqDefPredRate}
\end{equation}
over Sobolev-type ellipsoids, i.e.
\begin{equation}
\sup_{f: \, \sourcecondf} \E\big[\norm{A(\widehat f_{\baloracle}-f)}^2\big]=\sup_{g: \, \sourcecondg} \E\big[\norm{\widehat g_{\baloracle}-g}^2\big]
\lesssim_{\mu,p,C_A} \predrate. \label{EqBalOraclePredErrorMinimaxOverview}
\end{equation}
From an asymptotic point of view, where the dimension $D=D(\delta) \to \infty$ as $\delta \to 0$, this rate is \emph{minimax optimal}
if $D(\delta)$ grows faster than $(R^2 \delta^{-2})^{1/(4 \mu p +2p+1)}$ as shown in \cref{PropWeakErrorLowerBound}. The proof of the upper bound is very general and allows extensions to other source conditions and noise specifications, see \cref{RemNoise2Moments}. It exhibits close parallels to Showalter's method or standard gradient flow, see \cref{RemShowalter}. This is the first time that the CGNE iteration path is shown to achieve the optimal prediction error rate under \eqref{AssGN} at some iterate.

For deterministic noise of level $\delta$, the prediction error at the least-squares solution (i.e.\ at the last iteration $\widehat f_d$) is $\delta^2$. Under \eqref{AssGN}, however, this becomes $\E[\norm{A(\widehat f_{d}-f)}^2]=\delta^2D$ and regularisation by early stopping is essential for larger dimensions $D$ exactly as in nonparametric regression. Still, in case of small $D \ll (R^2 \delta^{-2})^{1/(4 \mu p +2p+1)}$, the estimation problem becomes in a statistical sense parametric and no regularisation is needed. The rate $\delta^2 D$, attained by $\widehat{f}_d$, is then minimax optimal for the prediction error, see \cref{RemMinimaxTruncationTimeWeak}.

To define a data-driven stopping rule, we monitor the squared \emph{residual norm} (also called \emph{discrepancy}, \emph{data misfit} or \emph{training error})
\begin{equation*}
R_t^2 \coloneqq \norm{A\widehat{f}_t-Y}^2 = \norm{r_t(AA^\top)Y}^2
\end{equation*}
at each iteration step, which is greedily minimised in the CG construction. Similarly to \citet{BlaHofRei2018EJS,BlaHofRei2018SIAM}, but for a random balanced oracle $\baloracle$, we  mimic the balanced oracle by a discrepancy-type principle.

\begin{definition} \label{DefEarlyStopRule}
For a suitable critical value $\kappa > 0$, introduce the early stopping rule
\begin{equation*}
\tau \coloneqq \inf\big\{t \in [0,d] \,\big|\,R_t^2\le \kappa\big\}.
\end{equation*}
\end{definition}
Under \eqref{AssDN}, setting $\kappa \coloneqq c^2 \delta^2$ for an a-priori fixed constant $c>1$, this coincides with the well-known discrepancy principle \citep[Sections~4.3 and~7.3]{EngHanNeu1996}. Note that by the linear interpolation, $R_t^2$ is continuous in $t$ with $R_0^2 = \norm{Y}^2$ and $R_d^2=0$. Hence, $\tau \in [0,d]$ exists and we have the identity $R_{\tau}^2=\kappa$, provided $\kappa \le \norm{Y}^2$. The balanced oracle can also be expressed in terms of the residual norm, and in \cref{RemDiscrepancyPrinciple}, the relation between the stopping indices $\tau$ and $\baloracle$ is clarified.

\cref{ThmWeakOracleIneqEarlyStop} yields for any signal $f$ an oracle-type inequality under Gaussian noise \eqref{AssGN}, relating the prediction error at $\tau$ to the expected minimal prediction error bound:
\begin{equation}
\E\big[\norm{A(\widehat f_\tau-f)}^2\big] \lesssim \inf_{t \in [0,d]} \E[A_{t,\lambda}+S_{t,\lambda}] +\delta^2\sqrt D + \abs{\kappa - \delta^2 D}. \label{EqBalOracleIneqOverview}
\end{equation}
Note that we do not have a classical oracle inequality in terms of the oracle error $\inf_{t \in [0,d]} \E[\norm{A(\widehat{f}_t-f)}^2]$, since the decomposition into the two error terms in \eqref{EqWeakErrorBoundOverview} gives only an upper bound. The dimension-dependent error term $\delta^2 \sqrt{D}$ is due to the variability of the Gaussian noise $\xi$ and unavoidable in the linear case of truncated SVD estimation \cite{BlaHofRei2018EJS}. Consequently, the early stopped estimator $\widehat{g}_{\tau}$ has the same error rate as $\widehat{g}_{\baloracle}$ if the remainder term is small enough. In this case of a-posteriori error control, we say that we can \emph{adapt} to the balanced oracle. In practice, the noise level $\delta$ needs to be estimated, and the threshold $\kappa$ is allowed to deviate from $\E[\norm{\xi}^2]=\delta^2 D$.

The nonasymptotic oracle inequality \eqref{EqBalOracleIneqOverview} does not require to fix any source conditions on $f$, yet allows  a straight-forward derivation of adaptive asymptotic error bounds, see \citet[Section~1.3]{Cav2011} for more details on this approach. Assume again \eqref{AssPSD} with $p>1/2$ and \eqref{AssGN}. If the remainder term $\delta^2 \sqrt{D} + \abs{\kappa - \delta^2 D}$ is at most of the order of the optimal rate, the bound \eqref{EqBalOraclePredErrorMinimaxOverview} directly carries over such that
\begin{equation*}
\sup_{f: \, \sourcecondf} \E\big[\norm{A(\widehat f_{\tau}-f)}^2\big]=\sup_{g: \, \sourcecondg} \E\big[\norm{\widehat{g}_{\tau}-g}^2\big] \lesssim_{\mu,p,C_A} \predrate,
\end{equation*}
where $\predrate$ is given in \eqref{EqDefPredRate}, and the optimal rate of convergence is attained. The early stopping rule $\tau$ is \emph{minimax adaptive} for the prediction error over Sobolev-type ellipsoids for all regularities $\mu>0$ with $\sqrt{D} \lesssim (R^2 \delta^{-2})^{1/(4\mu p + 2p +1)} \lesssim_{\mu,p,c_A} D$ (see \cref{CorEarlyStopWeakMinimax}).

In \cref{SecTransferReconstructionError}, we transfer the results for the prediction error to the reconstruction error, i.e.\ the error of the CG estimator. Similarly to the prediction error, the analysis of the reconstruction error relies on a nonasymptotic upper bound given in \cref{PropRoughBoundStrongError}, which depends on the error terms and the residual polynomial.

Under \eqref{AssPSD} with $p>1/2$ and \eqref{AssGN}, combining \cref{PropRoughBoundStrongError} with an interpolation inequality and $\E[S_{\baloracle,\lambda}] \lesssim_{\mu,p,C_A} \predrate$ yields
\begin{equation*}
\sup_{f: \, \sourcecondf} \E\big[\norm{\widehat{f}_{\baloracle}-f^\dagger}^2\big] \lesssim_{\mu,p,C_A} \recrate
\end{equation*}
with
\begin{equation}
\recrate \coloneqq R^{(4p+2)/(4\mu p +2p+1)} \delta^{8 \mu p / (4\mu p +2p+1)}. \label{EqDefRecRate}
\end{equation}
Under the same growth condition $D \gtrsim_{\mu,p,c_A} (R^2 \delta^{-2})^{1/(4\mu p +2p +1)}$ as before, this rate is minimax optimal such that $\widehat{f}_{\baloracle}$ simultaneously attains the optimal rate of convergence for the prediction as well as for the reconstruction error (see \cref{ThmStrongMinimax}).

For the reconstruction error at $\tau$, we need to perform a more intricate high probability analysis of the \emph{intrinsic regularisation parameter} $\abs{r'_{\tau}(0)}$ and the error terms at $\tau$ occurring in the upper bound on the reconstruction error. In \cref{ThmEarlyStopStrongMinimax}, under the same assumptions as for the prediction error, we show that the early stopping rule $\tau$ is also minimax adaptive for the reconstruction error over signals $f$ satisfying $\sourcecondf$ for all regularities $\mu>0$ with $\sqrt{D} \lesssim (R^2 \delta^{-2})^{1/(4\mu p + 2p +1)} \lesssim_{\mu,p,c_A} D$.
As a byproduct, our analysis yields also the classical convergence rates for CG under deterministic noise stopped by the discrepancy principle, see \cref{PropDiscrepancyPrincipleOrderOpt}, where some arguments may be more transparent than previous ones.

In \cref{SecNumericalIllustration}, we consider some numerical examples from the literature. For mildly ill-posed problems, the practical performance of the data-driven stopping rule $\tau$ is quite remarkable, usually very close to the best iterate along the CG path. Even for a severely ill-posed test problem, which is not covered by our theory, the early stopping criterion provides reasonable results.


\section{The interpolated conjugate gradient method} \label{SecInterpolConjGradientMethod}

We avoid discretisation errors by interpolating in \cref{DefInterpolEstimators} linearly between the residual polynomials from \cref{DefEstimators}. From a computational perspective, the interpolation does not impose additional computational cost. In practice, we calculate the estimators at integer iteration indices,
and as soon as $R_m^2 \le \kappa$ for some $m \in \{0,\dots,d\}$ with $\kappa$ from \cref{DefEarlyStopRule}, we stop the computation. The realised value of $\tau\in(m-1,m]$ can then be calculated from the last two residuals. Since CG iterates usually converge very quickly, stopping at an interpolated iteration helps to prevent overshooting.
For a discussion how the interpolation allows to retrieve an order-optimality result from the theory of deterministic inverse problems see \cref{RemOrderOptRegMethod}.

We analyse in \cref{SecClassicalConjugateGradientMethod} the classical noninterpolated CG method, which forms the basis for the following results. The proofs for this section are given in \cref{SecProofsResPol}.

\begin{definition} \label{DefInterpolEstimators}
For $t=k+\alpha$, $k=0,\dots,d-1$, $\alpha \in (0,1]$, the \emph{interpolated residual polynomial} is defined as
\begin{equation*}
r_t \coloneqq (1-\alpha)r_k + \alpha r_{k+1}.
\end{equation*}
The \emph{interpolated prediction} and \emph{CG estimators} are given by
\begin{equation*}
\widehat{g}_t \coloneqq (\id_D -r_t(AA^\top))Y \quad \text{and} \quad \widehat{f}_t \coloneqq A^{\dagger} \widehat{g}_t,
\end{equation*}
respectively, i.e.
\begin{equation*}
\widehat{g}_t = (1-\alpha) \widehat{g}_k + \alpha \widehat{g}_{k+1} \quad \text{and} \quad \widehat{f}_t = (1-\alpha) \widehat{f}_k + \alpha \widehat{f}_{k+1}.
\end{equation*}
\end{definition}

\begin{remark}
Although we will not use it, the interpolation follows also naturally by defining $r_t$ via penalised minimisation. Let
\begin{equation}
r_{k+1,\lambda}\coloneqq \argmin_{p_{k+1}\in \Pol_{k+1,1}}\Big(\norm{\mathsf{p_{k+1}}}_Y^2+\lambda a_{k+1}(p_{k+1})^2\Big), \label{EqDefrkpen}
\end{equation}
where $a_{k+1}(p)\in\R$ denotes the coefficient in front of $x^{k+1}$ for a polynomial $p(x)$ and $\lambda>0$ is a penalty parameter. At the end of \cref{SecProofsResPol}, we prove
\begin{equation}
r_{k+1,\lambda}= r_k+\alpha_\lambda(r_{k+1}-r_k) \quad \text{with} \quad \alpha_\lambda =\frac{R_k^2-R_{k+1}^2} {R_k^2-R_{k+1}^2+\lambda a_{k+1}(r_{k+1})^2}\in(0,1]. \label{Eqrkpen}
\end{equation}
Since $\alpha_\lambda\downarrow 0$ for $\lambda\uparrow\infty$ and $\alpha_\lambda=1$ for $\lambda=0$, we recover $r_t$ for $t=k+\alpha$, $k=0,\dots,d-1$, $\alpha\in(0,1)$, as $r_{k+1,\lambda}$ for a suitable penalisation parameter $\lambda=\lambda(t)>0$.
\end{remark}

We have extended \cref{DefEstimators} to $t \in [0,d]$, keeping the starting values $\widehat{g}_0 = \widehat{f}_0 = 0$, $r_0 \equiv 1$ and the relation $A \widehat{f}_t = \widehat{g}_t$ for $t \in [0,d]$. Note that $r_t$ for $t \in (k,k+1]$, $k=0,\dots,d-1$, is an element of $\Pol_{k+1,1}$.
\pagebreak

\begin{lemma}[Properties of the squared residual norm] \leavevmode \label{LemPropSquaredResNorm}
\begin{enumlemma}
\item For $t=k+\alpha$, $k=0,\dots,d-1$, $\alpha \in [0,1]$, the squared residual norm interpolates nonlinearly: \label{LemPropSquaredResNorm:NonlinInterpol}
\begin{equation*}
R_t^2 = \norm{r_t(AA^\top)Y}^2 = (1-\alpha)^2 R_k^2 + (1-(1-\alpha)^2)R_{k+1}^2.
\end{equation*}
\item The squared residual norm $R_t^2$ is continuous and decreasing in $t \in [0,d]$. \label{LemPropSquaredResNorm:CtsDecreasing}
\item For $t,s \in [0,d]$, we have \label{LemPropSquaredResNorm:LemResNormMaxts}
\begin{equation*}
R_{t \vee s}^2 \le \scapro{\mathsf{r_t}}{\mathsf{r_s}}_Y
\quad \text{and} \quad
\norm{A(\widehat{f}_t-\widehat{f}_s)}^2 \le \abs{R_t^2 - R_s^2}.
\end{equation*}
Both statements hold with equality for integer $t$ and $s$.
\end{enumlemma}
\end{lemma}

\begin{definition} \label{DefRespolSmallestZero}
Denote by $x_{1,t}$ the smallest zero on $[0,\infty)$ of the $t$-th residual polynomial $r_t$, $t \in (0,d]$. For $x\ge 0$, introduce the  functions
\begin{equation*}
r_{t,<}(x) \coloneqq r_t(x){\bf 1}(x<x_{1,t}), \quad
r_{t,>}(x) \coloneqq r_t(x){\bf 1}(x>x_{1,t})
\end{equation*}
and set $r_{0,<} \coloneqq r_0 \equiv 1$, $r_{0,>} \equiv 0$.
If $x_{1,t}$ is clear from the context, we also write $f_<(x) \coloneqq f(x){\bf 1}(x<x_{1,t})$, $\norm{\mathsf f}_{Y,<} \coloneqq \norm{\mathsf f_<}_Y$ and $\scapro{\mathsf{f}}{\mathsf{g}}_{Y,<}\coloneqq\scapro{\mathsf{f_<}}{\mathsf{g_<}}_Y$ for any functions $f$ and $g$, analogously for `$>$'. In particular, we set $\norm{v}_< \coloneqq \norm{1_<(AA^\top)v}$ and $\norm{w}_< \coloneqq \norm{1_<(A^\top A)w}$ for any vectors $v \in \R^D$ and $w \in \R^P$, where $1_<(x) \coloneqq {\bf 1}(x < x_{1,t})$, analogously for `$>$' and `$\ge$'.
\end{definition}

For integer iteration indices, the residual polynomials possess several well-known properties, which are collected in \cref{LemPropResidualPolynomials}. These carry over to the interpolated analogues. We state in the following lemma the most important properties, denoting by $\norm{\cdot}_{\mathrm{spec}}$ the spectral norm of a matrix.

\begin{lemma}[Properties of the residual polynomial]  \label{LemPropInterpolResidualPolynomials}
Let $t=k+\alpha$ with $k=0,\dots,d-1$ and $\alpha \in (0,1]$.
\begin{enumlemma}
\item For $k=1,\dots,d-1$, $r_k$ has $k$ real simple zeros $0<x_{1,k}< \dots < x_{k,k} < \norm{AA^\top}_{\mathrm{spec}}$. For $k=d$, the zeros are the $d$ distinct squared nonzero singular values $0 < x_{1,d}=\tilde\lambda_d^2 < \dots < x_{d,d}=\tilde\lambda_1^2 = \norm{AA^\top}_{\mathrm{spec}}$. \label{LemPropInterpolResidualPolynomials:simpleZeros:noninterpol}
\item For $\alpha \in (0,1)$, $r_t$ has $k+1$ real simple zeros $x_{i,t}$ satisfying $0 < x_{i,k+1} < x_{i,t} < x_{i,k}$, $i=1,\dots,k+1$, setting $x_{1,0}, x_{k+1,k} \coloneqq \infty$. \label{LemPropInterpolResidualPolynomials:simpleZeros}
\item $r_t$ can be written as \label{LemPropInterpolResidualPolynomials:Formula}
\begin{equation*}
r_t(x)=\prod_{i=1}^{\ceil{t}}\Big(1-\frac{x}{x_{i,t}}\Big), \ x \in [0,\norm{AA^\top}_{\mathrm{spec}}], \quad \text{with} \quad \abs{r_t'(0)}=\sum_{i=1}^{\ceil{t}} x_{i,t}^{-1}.
\end{equation*}
\item $r_t$ is nonnegative, decreasing and convex on $[0, x_{1,t}]$ and log-concave on $[0, x_{1,t})$. \label{LemPropInterpolResidualPolynomials:Properties}
\item For $t_j = k+\alpha_j$, $j=1,2$, with $k =0,\dots,d-1$ and $0 < \alpha_1 < \alpha_2 < 1$, we have \label{LemPropInterpolResidualPolynomials:interlacedZeros}
\begin{equation*}
x_{1,t_2} < x_{1,t_1} < x_{2,t_2} < \dots < x_{k,t_1} < x_{k+1,t_2} < x_{k+1, t_1}.
\end{equation*}
\item For fixed $x \in [0,\norm{AA^\top}_{\mathrm{spec}}]$, $t \mapsto r_{t,<}(x)$ is continuous and nonincreasing on $[0,d]$. \label{LemPropInterpolResidualPolynomials:rtNonIncreasingCts}
\item For $k=1,\dots,d$, $i=1,\dots,k$, we have \label{LemPropInterpolResidualPolynomials:ZerosEigenvalues:noninterpol}
\begin{equation*}
x_{k+1-i,k} \le \lambda_i^2.
\end{equation*}
\end{enumlemma}
\end{lemma}

The following key proposition without interpolation is originally due to \citet[(3.14)]{Nem1986}, see \cref{LemRkbound} for details. We extend it to real iteration indices. It allows the restriction to squared singular values in the interval $[0,x_{1,t}]$ for the error analysis.

\begin{proposition}\label{PropRtbound}
The residual polynomial satisfies $\norm{\mathsf{r_t}}_Y^2\le \norm{\mathsf{r_{t,<}^{1/2}}}_Y^2$ for all $t \in [0,d]$.
\end{proposition}

On the interval $[0,x_{1,t}]$, the residual polynomial can be controlled in terms of the value of its first derivative $r_t'(0)$ in zero. This quantity plays an essential role in the CG analysis, as discussed in \cref{RemDerivativeRespolZero:tradeoff}.

The first inequality in the lemma below is well-known, dating back to \citet[Lemma~4.3.15]{Lou1989}, and follows from the convexity of the residual polynomial $r_t$ on $[0,x_{1,t}]$. The second one is new and results from the log-concavity of $r_t$ on the same interval.

\begin{lemma} \label{LemrtBound}
For all $t \in [0,d]$, we have
\begin{equation*}
(1-\abs{r_t'(0)}x)_+\le
r_t(x)\le e^{-\abs{r_t'(0)}x},\quad x\in[0,x_{1,t}].
\end{equation*}
\end{lemma}

Using the upper bound, we get directly the following result. It will allow us to control the approximation error $A_{t,\lambda}$ from \cref{DefErrorTerms}.

\begin{corollary} \label{CorUpperBoundApproxErrorSourceCond}
Under the source condition $\sourcecondg$, for all $t \in (0,d]$, we have
\begin{equation*}
\norm{\mathsf{r_{t,<}^{1/2}}g}^2 \le R^2(\mu+1/2)^{2\mu +1} \abs{r_t'(0)}^{-2\mu -1}.
\end{equation*}
\end{corollary}


\section{Analysis of the prediction error} \label{SecAnalysisPredError}

\subsection{Decomposition of the prediction error}

Due to the nonlinearity of the CG estimator, we cannot completely disentangle data and approximation errors \cite[Section~7.3]{EngHanNeu1996}. We are, however, able to identify terms that share important properties of these errors, in particular regarding monotonicity along the iteration path.

\begin{definition} \label{DefErrorTerms}
For $t \in [0,d]$, the \emph{stochastic error term} is defined as
\begin{equation*}
S_{t,\lambda} \coloneqq \norm{\mathsf{(1-r_{t,<})^{1/2}}\xi}^2
\end{equation*}
and the \emph{approximation error term} as
\begin{equation*}
 A_{t,\lambda} \coloneqq \norm{\mathsf{r_{t,<}^{1/2}}g}^2+R_t^2-\norm{\mathsf{r_{t,<}^{1/2}}Y}^2.
\end{equation*}
\end{definition}

\begin{proposition}\label{PropWeakLoss}
For $t \in [0,d]$, we have
\begin{equation*}
\norm{A(\widehat f_t-f)}^2=A_{t,\lambda}+S_{t,\lambda}-2\scapro{\xi}{\mathsf{r_{t,>}}Y}\le 2(A_{t,\lambda}+S_{t,\lambda}).
\end{equation*}
\end{proposition}

\begin{proof}
For $t=0$, the claim is clear, since $r_{0,<} \equiv 1$, $r_{0,>} \equiv 0$ and $\widehat{f}_0 = 0$. Let $t>0$. We deduce by direct calculation
\begin{align*}
\norm{A(\widehat f_t-f)}^2
&=\norm{(\mathsf{1-r_t})Y-g}^2\\
&=\norm{\mathsf{r_t}Y}^2+\norm{\xi}^2-2\scapro{\xi}{\mathsf{r_t}Y}\\
&=R_t^2+\norm{\xi}^2-2 \big( \scapro{\xi}{\mathsf{r_{t,>}}Y} + \scapro{\xi}{\mathsf{r_{t,<}}Y} \big) \\
&= R_t^2+\norm{\xi}^2-2 \scapro{\xi}{\mathsf{r_{t,>}}Y} + \norm{\mathsf{r_{t,<}^{1/2}}(Y-\xi)}^2 - \norm{\mathsf{r_{t,<}^{1/2}}Y}^2 - \norm{\mathsf{r_{t,<}^{1/2}}\xi}^2 \\
&= \norm{\mathsf{r_{t,<}^{1/2}}g}^2+R_t^2- \norm{\mathsf{r_{t,<}^{1/2}}Y}^2+\norm{\mathsf{(1-r_{t,<})^{1/2}}\xi}^2-2\scapro{\xi}{\mathsf{r_{t,>}}Y} \\
&= A_{t,\lambda}+S_{t,\lambda}-2\scapro{\xi}{\mathsf{r_{t,>}}Y}.
\end{align*}
When restricting to singular values $\lambda_i^2 < x_{1,t}$ the same calculation gives
\begin{equation*} \norm{A(\widehat f_t-f)}_<^2=\norm{\mathsf{r_{t,<}^{1/2}}g}^2+\norm{\mathsf{r_{t}}Y}_<^2- \norm{\mathsf{r_{t,<}^{1/2}}Y}^2+\norm{\mathsf{(1-r_{t})_<^{1/2}}\xi}^2.
\end{equation*}
On the other hand, we have $\norm{A(\widehat f_t-f)}_{\ge}^2= \norm{(\mathsf{1-r_t})Y-g}_{\ge}^2\le 2(\norm{\xi}_{\ge}^2+\norm{\mathsf{r_t}Y}_>^2)$ and thus
\begin{align*}
\tfrac12\norm{A(\widehat f_t-f)}^2 &\le \norm{A(\widehat f_t-f)}_<^2+\tfrac12\norm{A(\widehat f_t-f)}_{\ge}^2\\
 &\le \norm{\mathsf{r_{t,<}^{1/2}}g}^2+\norm{\mathsf{r_{t}}Y}_<^2- \norm{\mathsf{r_{t,<}^{1/2}}Y}^2+\norm{\mathsf{(1-r_{t})_<^{1/2}}\xi}^2+\norm{\xi}_{\ge}^2+\norm{\mathsf{r_t}Y}_>^2\\
&= \norm{\mathsf{r_{t,<}^{1/2}}g}^2+R_t^2- \norm{\mathsf{r_{t,<}^{1/2}}Y}^2+\norm{\mathsf{(1-r_{t,<})^{1/2}}\xi}^2,
\end{align*}
which equals $A_{t,\lambda}+S_{t,\lambda}$ and gives the upper bound.
\end{proof}

\begin{lemma} \label{LemPropErrorTerms}
The stochastic error term $S_{t,\lambda}$, $t \in [0,d]$, satisfies
\begin{enumlemma}
\item $S_{t,\lambda}=0$ for $\xi=0$, \label{LemPropErrorTerms:StochErrorXiZero}
\item $S_{0,\lambda}=0$ and $S_{d,\lambda}=\norm{\xi}^2$, \label{LemPropErrorTerms:StochErrorZeroD}
\item $S_{t,\lambda} \le \norm{((\mathsf{\abs{r_t'(0)}x)^{1/2}\wedge 1})\xi}^2$, \label{LemPropErrorTerms:BoundStochError}
\item $S_{t,\lambda}$ is nondecreasing in $t$, \label{LemPropErrorTerms:StochErrorNonDec}
\item $S_{t,\lambda}$ is continuous in $t$. \label{LemPropErrorTerms:StochErrorCts}
\end{enumlemma}
The approximation error term $A_{t,\lambda}$, $t \in [0,d]$, satisfies
\begin{enumlemmaresume}
\item $A_{t,\lambda}\le 0$ if $Af=0$, \label{LemPropErrorTerms:BoundApproxErrorZeroSignal}
\item $A_{0,\lambda}=\norm{Af}^2$ and $A_{d,\lambda}=0$, \label{LemPropErrorTerms:ApproxErrorZeroD}
\item $A_{t,\lambda}\le \norm{\mathsf{r_{t,<}^{1/2}}g}^2 \le \norm{\mathsf{\exp(-\abs{r_t'(0)}x/2)}g}^2$, \label{LemPropErrorTerms:BoundApproxError}
\item the upper bound $\norm{\mathsf{r_{t,<}^{1/2}}g}^2$ is nonincreasing in $t$, \label{LemPropErrorTerms:ApproxErrorBoundNonInc}
\item $A_{t,\lambda}$ is continuous in $t$. \label{LemPropErrorTerms:ApproxErrorCts}
\end{enumlemmaresume}
\end{lemma}

\begin{proof}
Property~\subcref{LemPropErrorTerms:StochErrorXiZero} follows from the definition of $S_{t,\lambda}$, while \subcref{LemPropErrorTerms:StochErrorZeroD} and \subcref{LemPropErrorTerms:ApproxErrorZeroD} are due to $r_{0,<} \equiv 1$ and $r_d(\lambda_i^2)=0$, $i=1,\dots,D$, by \cref{LemPropInterpolResidualPolynomials:simpleZeros:noninterpol}.
Since $t \mapsto r_{t,<}(x)$ is nonincreasing for fixed $x \in [0,\norm{AA^\top}_{\mathrm{spec}}]$ by \cref{LemPropInterpolResidualPolynomials:rtNonIncreasingCts}, we obtain \subcref{LemPropErrorTerms:StochErrorNonDec} and \subcref{LemPropErrorTerms:ApproxErrorBoundNonInc}. Using $1-r_t(x)\le (\abs{r_t'(0)}x)\wedge 1$ on $[0,x_{1,t}]$ by \cref{LemrtBound} and $r_t\ge 0$ on that interval, we have \subcref{LemPropErrorTerms:BoundStochError}.
Properties~\subcref{LemPropErrorTerms:BoundApproxErrorZeroSignal} and \subcref{LemPropErrorTerms:BoundApproxError} follow from $R_t^2 \le \norm{\mathsf{r_{t,<}^{1/2}}Y}^2$ by \cref{PropRtbound} and \cref{LemrtBound}. Note that $R_t^2$ is continuous in $t$ by \cref{LemPropSquaredResNorm:CtsDecreasing}. Together with \cref{LemPropInterpolResidualPolynomials:rtNonIncreasingCts}, this yields the remaining properties \subcref{LemPropErrorTerms:StochErrorCts} and \subcref{LemPropErrorTerms:ApproxErrorCts}.
\end{proof}

\begin{remark} \label{RemSumYnonZero}
For the proof of $S_{d,\lambda}=\norm{\xi}^2$ and $A_{d,\lambda}=0$, it was crucial that $r_d$ maps all squared nonzero singular values of $A$ to zero. This is guaranteed by Assumption~\eqref{AssY}. Assuming only $Y \neq 0$ and defining the maximal iteration index as in \citet[Section~5.2]{BlaKra2016} would not suffice.
\end{remark}

For later purposes we also need the following relatively rough bound, whose proof can be found in \cref{SecProofsAnalysisPredError}.

\begin{lemma}\label{Lemrkg}
We have $\norm{\mathsf{r_{t,<}}g}^2\le 6S_{t,\lambda}+2A_{t,\lambda}$ for all $t\in [0,d]$.
\end{lemma}

\subsection{An oracle minimax bound for the prediction error} \label{SecOracleMinimaxBoundPredError}

In this section, we analyse the prediction error of the balanced oracle $\baloracle$ from \cref{DefBalancedOracle}.

\begin{lemma} The balanced oracle satisfies without further assumptions
\begin{equation*}
\norm{A(\widehat f_{\baloracle}-f)}^2\le \tfrac{4e}{e-1}\inf_{\rho>0} \big( \norm{\mathsf{\exp(-\rho x/2)}g}^2 \vee \norm{\mathsf{(1-\exp(-\rho x))^{1/2}}\xi}^2\big).
\end{equation*}
\end{lemma}

\begin{proof}
In the case $Af=0$, we have $\baloracle = 0$ such that $\norm{A(\widehat f_{\baloracle}-f)}^2 = 0$ and the claim is clear. Suppose $Af \neq 0$. Then $\baloracle > 0$ and also $\abs{r_{\baloracle}'(0)} > 0$. \cref{LemPropErrorTerms:BoundStochError} and \subcref{LemPropErrorTerms:BoundApproxError} allow us to bound the error terms.
Since the upper bound for $A_{t,\lambda}$ is decreasing in $\abs{r_t'(0)}$ and the upper bound for $S_{t,\lambda}$ is increasing in $\abs{r_t'(0)}$ and $A_{\baloracle,\lambda}=S_{\baloracle,\lambda}$, we deduce by \cref{PropWeakLoss}
\begin{align}
\norm{A(\widehat f_{\baloracle}-f)}^2
&\le 2(A_{\baloracle,\lambda}+S_{\baloracle,\lambda}) \notag\\
&= 4\big(A_{\baloracle,\lambda} \wedge S_{\baloracle,\lambda}\big) \notag\\
&\le 4\big( \norm{\mathsf{\exp(-\abs{r_{\baloracle}'(0)}x/2)}g}^2 \wedge \norm{(\mathsf{(\abs{r_{\baloracle}'(0)}x)^{1/2}\wedge 1})\xi}^2\big) \label{EqForRemWeakErrorDN}\\
&\le 4\inf_{\rho>0}\big( \norm{\mathsf{\exp(-\rho x/ 2)}g}^2 \vee \norm{(\mathsf{(\rho x)^{1/2}\wedge 1})\xi}^2\big). \notag
\end{align}
It remains to apply $y\wedge 1\le \frac e{e-1}(1-e^{-y})$ for $y=\rho x\ge 0$.
\end{proof}

\begin{remark} \label{RemWeakErrorDN}
Equation~\eqref{EqForRemWeakErrorDN} yields
\begin{equation*}
\norm{A(\widehat f_{\baloracle}-f)}^2 \le 4 \norm{ (\mathsf{(\abs{r_{\baloracle}'(0)}x)^{1/2} \wedge \mathsf 1})\xi }^2 \le 4\norm{\xi}^2.
\end{equation*}
In particular, stopping at the balanced oracle does not incur a significantly larger error than iterating until the terminal iteration index $d$, where the data are perfectly reproduced.
Under \eqref{AssDN}, the deterministic upper bound $4\delta^2$ follows.
\end{remark}

\begin{theorem} \label{ThmWeakMinimax}
Suppose $f$ satisfies the source condition $\sourcecondf$. For all $t \in (0,d]$, we have
\begin{equation}
\norm{A(\widehat f_t-f)}^2 \le 2R^2(\mu+1/2)^{2\mu+1} \abs{r_t'(0)}^{-2\mu-1}+2\norm{(\mathsf{(\abs{r_t'(0)}x)^{1/2}\wedge 1})\xi}^2. \label{EqThmWeakMinimaxBoundPredError}
\end{equation}
Under \eqrefAssPSDupper{} with $p>1/2$ and \eqref{AssGN}, the balanced oracle satisfies
\begin{equation}
\E\big[\norm{A(\widehat f_{\baloracle}-f)}^2\big]\le 4 \E[S_{\baloracle,\lambda}] \le C_{\mu,p,C_A} \predrate \label{EqThmWeakMinimaxRatePredErrorBalOracle}
\end{equation}
with $\predrate$ given in \eqref{EqDefPredRate} and a constant $C_{\mu,p,C_A}>0$ depending only on $\mu$, $p$ and $C_A$.
\end{theorem}

\begin{proof}
By \cref{LemPropErrorTerms:BoundApproxError,CorUpperBoundApproxErrorSourceCond}, the approximation error satisfies
\begin{equation}
A_{t,\lambda} \le R^2 (\mu+1/2)^{2\mu+1} \abs{r_t'(0)}^{-2\mu-1}. \label{EqUpperBoundApproxErrorSource}
\end{equation}
Combining the bound in \cref{LemPropErrorTerms:BoundStochError} and \eqref{EqUpperBoundApproxErrorSource} with \cref{PropWeakLoss} yields the first claim.

Assume \eqref{AssGN} and \eqrefAssPSDupper{} with $p>1/2$. If $Af=0$, the second claim is clear. Suppose $Af \neq 0$ such that $\baloracle > 0$ and $\abs{r_{\baloracle}'(0)} > 0$. Since the upper bound for $A_{t,\lambda}$ is decreasing in $\abs{r_t'(0)}$ and the upper bound for $S_{t,\lambda}$ is increasing in $\abs{r_t'(0)}$ and $A_{\baloracle,\lambda}=S_{\baloracle,\lambda}$, we obtain by \cref{PropWeakLoss}
\begin{align}
\E\big[\norm{A(\widehat f_{\baloracle}-f)}^2\big]
&\le 2\E[A_{\baloracle,\lambda}+S_{\baloracle,\lambda}] \label{EqRateStochError} \\
&= 4\E[A_{\baloracle,\lambda} \wedge S_{\baloracle,\lambda}] \notag\\
&\le 4\E\big[R^2(\mu+1/2)^{2\mu+1} \abs{r_{\baloracle}'(0)}^{-2\mu-1}\wedge \norm{(\mathsf{(\abs{r_{\baloracle}'(0)}x)^{1/2}\wedge 1})\xi}^2\big]  \notag\\
&\le 4\E\Big[\inf_{\rho>0}\big( R^2(\mu+1/2)^{2\mu+1} \rho^{-2\mu-1} \vee \norm{(\mathsf{(\rho x)^{1/2}\wedge 1})\xi}^2\big)\Big]  \notag\\
&\le 4\inf_{\rho>0}\E\big[ R^2(\mu+1/2)^{2\mu+1} \rho^{-2\mu-1}+\norm{(\mathsf{(\rho x)^{1/2}\wedge 1})\xi}^2\big]  \notag\\
&\le 4\inf_{\rho>0}\Big( R^2(\mu+1/2)^{2\mu+1} \rho^{-2\mu-1} + (C_A \vee 1)^2 \delta^2 \sum_{i=1}^D \big((\rho i^{-2p})\wedge 1\big)\Big). \notag
\end{align}
Note that $\rho x^{-2p} \le 1$ if and only if $x \ge \rho^{1/(2p)}$. A Riemann sum approximation yields
\begin{equation}
\sum_{i=1}^D \big((\rho i^{-2p})\wedge 1\big) \le \int_0^{D} \big((\rho x^{-2p})\wedge 1\big) dx
\le \tfrac{2p}{2p-1} \rho^{1/(2p)}. \label{EqBoundExpWeakStochError}
\end{equation}
Hence,
\begin{equation*}
\E\big[\norm{A(\widehat f_{\baloracle}-f)}^2\big] \le 4 \inf_{\rho>0} \big( R^2(\mu+1/2)^{2\mu+1}\rho^{-2\mu-1} + (C_A \vee 1)^2\delta^2\tfrac{2p}{2p-1}\rho^{1/(2p)} \big).
\end{equation*}
The sum is minimal for
\begin{equation*}
\rho = \big( 2(C_A \vee 1)^{-2}(2p-1)(\mu+1/2)^{2\mu+2}R^2 \delta^{-2} \big)^{2p/(4\mu p+2p+1)}>0,
\end{equation*}
which yields the rate in \eqref{EqThmWeakMinimaxRatePredErrorBalOracle}.
Inserting $A_{\baloracle,\lambda}=S_{\baloracle,\lambda}$ in \eqref{EqRateStochError}, the second claim follows.
\end{proof}

\begin{remark} \label{RemDerivativeRespolZero:tradeoff}
As we see in the bound \eqref{EqThmWeakMinimaxBoundPredError} on the prediction error, the quantity $\abs{r_t'(0)}$ controls the size of the approximation and the stochastic error term. The choice of the iteration index $t$ therefore reflects a tradeoff between both errors that can be compared to the bias-variance tradeoff.
For deterministic inverse problems, \citet{Han1995} calls $\abs{r_k'(0)}$, $k \in \N$, the \emph{modulus of convergence}.
\end{remark}

\begin{remark}\label{RemNoise2Moments}
Note that the distribution of the noise $\xi$ only needs to satisfy $\E[\xi_i^2]=\delta^2$, $i=1,\dots,D$, to prove \eqref{EqThmWeakMinimaxRatePredErrorBalOracle}. So, this bound holds for any centred error $\xi$ with covariance matrix $\delta^2 \id_D$. In particular, no Gaussian-type high probability bound is needed to establish the error rate,
as required in the statistical CG literature before.

Let us also emphasise that the bounds of \cref{ThmWeakMinimax} are \emph{dimension-free} because we assumed \eqrefAssPSDupper{} with $p>1/2$ so that $AA^\top$ is trace class and \eqref{EqBoundExpWeakStochError} holds. This leads to the conjecture that the same result also holds
in the infinite-dimensional Hilbert space setting.
\end{remark}

We have shown in \cref{ThmWeakMinimax} that the prediction estimator $\widehat{g}_{\baloracle} = A\widehat{f}_{\baloracle}$ at the balanced oracle attains the error rate $\predrate$. The following proposition, which is proved in \cref{SecProofsAnalysisPredError}, gives the corresponding lower bound.

\begin{proposition} \label{PropWeakErrorLowerBound}
Assume \eqrefAssPSDlower{} with $p>0$ and \eqref{AssGN}. We have
\begin{equation*}
\inf_{\widehat{f}} \sup_{f: \, \sourcecondf} \E_f\big[\norm{A( \widehat{f} - f)}^2\big] \ge c_{\mu,p,c_A} \predrate
\end{equation*}
for $\delta/R \le \tilde{c}_{\mu,p,c_A}$, provided $(R^2 \delta^{-2})^{1/(4\mu p+2p+1)} \lesssim_{\mu,p,c_A} D$, where the infimum is taken over all $\R^P$-valued estimators $\widehat{f}$ and $\predrate$ is given in \eqref{EqDefPredRate}. The positive constants $c_{\mu,p,c_A}$, $\tilde{c}_{\mu,p,c_A}$ depend only on $\mu$, $p$ and $c_A$.
\end{proposition}

\begin{remark}[Role of the dimension] \label{RemMinimaxTruncationTimeWeak}
In the proof of \cref{PropWeakErrorLowerBound}, we need to consider alternatives in the first
\begin{equation*}
t_{\mu,p,R}^{\mathrm{pred}}(\delta) \cong_{\mu,p,c_A} (R^2 \delta^{-2})^{1/(4\mu p+2p+1)}
\end{equation*}
SVD entries of the signal such that we  assume that $D$ is at least of this order. In this case, the error rate $\predrate$ is minimax.
In the case $D \ll t_{\mu,p,R}^{\mathrm{pred}}(\delta)$, we have $\E[\norm{A(\widehat f_{d}-f)}^2]=\E[\norm{\xi}^2]=\delta^2 D \ll \predrate$ and thus by \cref{RemWeakErrorDN} also $\E[\norm{A(\widehat f_{\baloracle}-f)}^2]\lesssim\delta^2 D \ll \predrate$.
In fact, the dimension is so small that regularisation is not needed and iterating until the final index $d$ is optimal. The corresponding lower bound for the prediction error is indeed of order $\delta^2 D$, which can be shown by the generalised Fano method.

If we view \eqref{EqModel} as a discretisation of an underlying continuous inverse problem, which we approach as $D(\delta),P(\delta)\to\infty$ for $\delta\to 0$,  we might rather want to consider sequences $f^{(\delta)}\in \R^{P(\delta)}$, $A^{(\delta)}\in\R^{D(\delta) \times P(\delta)}$ of discretised functions and operators satisfying the source condition \eqref{EqDefSourceCond} and the spectral decay condition \textup{(\hyperref[AssPSD]{\textbf{PSD}($p_{\delta}$,$c_{A,\delta}$,$C_{A,\delta}$)})} with possibly varying parameters $\mu_\delta, R_\delta, p_\delta, c_{A,\delta},C_{A,\delta}$. If these parameters remain bounded away from zero and infinity and $p_{\delta}>1/2+\varepsilon$ for some fixed $\varepsilon>0$, then the same proofs, but with much heavier notation, yield in Theorem \ref{ThmWeakMinimax}
\begin{equation*}
\limsup_{\delta\to 0}\mathcal{R}^{*, \mathrm{pred}}_{\mu_{\delta},p_{\delta},R_{\delta}}(\delta)^{-1}\E\big[\norm{A^{(\delta)}(\widehat f_{\baloracle}-f^{(\delta)})}^2\big] <\infty  
\end{equation*}
and in Proposition \ref{PropWeakErrorLowerBound}
\begin{equation*}
\liminf_{\delta\to 0}\mathcal{R}^{*, \mathrm{pred}}_{\mu_{\delta},p_{\delta},R_{\delta}}(\delta)^{-1}\inf_{\widehat{f}} \sup_{f: \, \norm{f}_{\mu_{\delta}} \le R_{\delta}} \E_f\big[\norm{A^{(\delta)}( \widehat{f} - f)}^2\big] >0,
\end{equation*}
provided $\lim_{\delta \to 0} D(\delta) \delta^{2/(4\mu_{\delta}p_{\delta}+2p_{\delta}+1)} = \infty$. 
Similar extensions also hold for all subsequent results.
\end{remark}

\begin{remark}[Comparison with Showalter's method] \label{RemShowalter}
The CG algorithm can be compared to {Showalter's method} (or {asymptotic regularisation}), which is just the classical {gradient flow} and appears as continuous-time limit of standard iterative linear regularisation schemes. The Showalter estimator of $f$ at $s \ge 0$ is given by
$\widehat{f}_s^{\mathrm{SW}} = A^\dagger (\mathsf{1-\exp(-s x)}) Y$  \cite[Example~4.7]{EngHanNeu1996}.
Its prediction error for deterministic $s \ge 0$ can be decomposed as
\begin{equation}
\E\big[\norm{A(\widehat{f}_s^{\mathrm{SW}} - f)}^2\big] = \norm{\mathsf{e^{-s x}} Af}^2 + \E\big[ \norm{(\mathsf{1-e^{-s x}})\xi}^2 \big] \eqqcolon B_{s,\lambda}^2(f) + V_{s,\lambda}. \label{EqBiasVarianceShowalter}
\end{equation}
Inserting $s \triangleq \tfrac{1}{2} \abs{r_t'(0)}$, the then random term $\norm{\mathsf{e^{-s x}} Af}^2$ coincides with the upper bound on $A_{t,\lambda}$ given in \cref{LemPropErrorTerms:BoundApproxError}. Using that $1-e^{-y} \le y \wedge 1 \le \tfrac{e}{e-1}(1-e^{-y})$ for $y \ge 0$, the stochastic error term $\norm{(\mathsf{1-e^{-s x}})\xi}^2$ of $\widehat{f}_s^{\mathrm{SW}}$ is up to constants equal to $\norm{(\mathsf{s x \wedge 1}) \xi}^2$.
For $s \triangleq \frac12\abs{r_t'(0)}$, this is smaller than the upper bound $\norm{(\mathsf{(\abs{r_t'(0)}x)^{1/2}\wedge 1})\xi}^2$ on $S_{t,\lambda}$ from \cref{LemPropErrorTerms:BoundStochError}.

We achieve, however, the same order of the oracle prediction error under the following mild condition on the singular value decay, where $C_{\lambda} > 0$ is a constant:
\begin{equation}
\sum_{i=j}^D {\lambda}_i^2 \le C_{\lambda} j\lambda_j^2 , \quad j=2,\dots,D.\label{AssShowalter}
\end{equation}
The condition \eqref{AssShowalter} is easily verified under polynomial singular value decay \eqref{AssPSD} with $p>1/2$, but is also true in the case of exponential or Gaussian-type decay.

Assuming for simplicity distinct singular values $(\lambda_j)$, we obtain for deterministic $\rho^{-1}\in({\lambda}_j^2,{\lambda}_{j-1}^2]$
\begin{align}
 &\E\big[\norm{(\mathsf{(\rho x)^{1/2} \wedge 1}) \xi}^2\big] = \delta^2 \big(\#\{i=1,\dots,D\,|\,\lambda_i^2\ge \rho^{-1}\}+{\textstyle\sum_{i:\, \lambda_i^2<\rho^{-1}}} \rho\lambda_i^2\big)  \notag \\
 &
\le  \delta^2 \big( (j-1)+ {\lambda}_j^{-2} \textstyle{\sum_{i=j}^D} {\lambda}_i^2 \big)\le (2C_{\lambda}+1) \delta^2 (j-1) \notag \\
&=(2C_{\lambda}+1)\E\big[\norm{{\bf 1}(\mathsf{\rho x\ge 1}) \xi}^2\big]\le (2C_{\lambda}+1)\E\big[\norm{(\mathsf{(\rho x) \wedge 1}) \xi}^2\big]. \label{RemShowalterIneq}
\end{align}
Combining all possible intervals for $\rho^{-1}$ and considering the case $\rho^{-1} \le {\lambda}_D^2$ separately, the inequality \eqref{RemShowalterIneq} holds for all $\rho\ge \lambda_1^{-2}$. We then obtain by the monotonicity argument with respect to $\rho$ in the proof of \cref{ThmWeakMinimax} and the bias-variance decomposition \eqref{EqBiasVarianceShowalter} that
\begin{align*}
\E\big[\norm{A(\widehat f_{\baloracle}-f)}^2\big] &\le 2\E[A_{\baloracle,\lambda}+S_{\baloracle,\lambda}] \\
&\lesssim \inf_{\rho \ge \lambda_1^{-2}} \big( \norm{\mathsf{e^{-\rho x/2}}Af}^2 + \E\big[\norm{(\mathsf{(\rho x)^{1/2}\wedge 1})\xi}^2\big] \big)\\
&\lesssim_{C_{\lambda}} \inf_{s\ge \lambda_1^{-2}}\big(B_{s,\lambda}^2(f)+V_{s,\lambda}\big) =\inf_{s\ge \lambda_1^{-2}} \E\big[\norm{A(\widehat f_s^{\mathrm{SW}}-f)}^2\big].
\end{align*}
The infimum on the right-hand side can be taken over $s \ge 0$ for sufficiently small $\delta > 0$ whenever $f \neq 0$.

The interesting fact is that the constant involved does not depend on the signal $f$ so that, in particular, all uniform convergence rates over classes of signals for Showalter's method under \eqref{AssGN} directly transfer to CG, e.g.\ under generalised source conditions~\cite{MatHof2008}.

For exponential decay $\lambda_i\thicksim e^{-pi}$, $p>0$, and polynomial source condition, we obtain similarly to the proof of \eqref{EqThmWeakMinimaxRatePredErrorBalOracle} the prediction error rate
\begin{equation*}
\sup_{f: \, \sourcecondf}\E\big[\norm{A(\widehat{f}_{\baloracle} - f)}^2\big] \lesssim_{\mu,p} \log(e + R \delta^{-1})\delta^2.
\end{equation*}
This rate is minimax optimal, the lower bound follows exactly as \cref{PropWeakErrorLowerBound} by replacing the condition on the singular values accordingly, compare also
\citet[Theorem~1]{Tsy2000}.
\end{remark}

\begin{proposition} \label{PropWeakMinimaxDN}
Suppose $f$ satisfies the source condition $\sourcecondf$. Under \eqref{AssDN}, for all $t \in (0,d]$, we can bound
\begin{equation*}
\norm{A(\widehat f_t-f)}^2 \le 2R^2(\mu+1/2)^{2\mu+1} \abs{r_t'(0)}^{-2\mu-1}+2\delta^2
\end{equation*}
and
\begin{equation}
R_t \le R(\mu + 1/2)^{\mu + 1/2} \abs{r_t'(0)}^{-\mu-1/2} + \delta. \label{EqDNBoundResidual}
\end{equation}
\end{proposition}

\begin{proof}
The first part immediately follows from \eqref{EqThmWeakMinimaxBoundPredError}. \cref{PropRtbound} implies
\begin{equation*}
R_t \le \norm{\mathsf{r_{t,<}^{1/2}}Y} \le \norm{\mathsf{r_{t,<}^{1/2}}g} + \norm{\xi}.
\end{equation*}
\cref{CorUpperBoundApproxErrorSourceCond} and Assumption~\eqref{AssDN} yield the second claim.
\end{proof}

\begin{remark} \label{RemDerivativeRespolZero:comparisonDN}
In \eqref{EqDNBoundResidual}, we obtain the same upper bound on the residual norm as \citet[Lemma~7.10]{EngHanNeu1996} with a slightly smaller constant on the first summand.
\end{remark}


\section{Early stopping and its prediction error} \label{SecEarlyStopWeakNorm}

In \cref{SecAnalysisPredError}, we showed that the balanced oracle achieves the optimal MSE rate. Similarly as  \citet{BlaHofRei2018EJS,BlaHofRei2018SIAM}, where the balanced oracle stopping index was deterministic though, we mimic $\baloracle$ by the fully data-driven stopping rule $\tau$ from \cref{DefEarlyStopRule}.

\begin{lemma}\label{LemTauwRes}
The balanced oracle $\baloracle$ can be equally expressed as
\begin{equation*}
\baloracle=\inf\big\{t \in [0,d] \,|\, R_t^2\le \norm{\xi}^2  +2\scapro{\xi}{\mathsf{r_{t,<}}g}\big\}.
\end{equation*}
\end{lemma}

\begin{proof}
By the definitions of $A_{t,\lambda}$ and $S_{t,\lambda}$, for any $t \in [0,d]$, we deduce
\begin{align}
R_t^2-\norm{\xi}^2 -2\scapro{\xi}{\mathsf{r_{t,<}}g} \notag
&=A_{t,\lambda}+\norm{\mathsf{r_{t,<}^{1/2}}Y}^2-\norm{\mathsf{r_{t,<}^{1/2}}g}^2-\norm{\xi}^2 -2\scapro{\xi}{\mathsf{r_{t,<}}g} \notag\\
&=A_{t,\lambda}+\norm{\mathsf{r_{t,<}^{1/2}}\xi}^2-\norm{\xi}^2 \notag\\
&= A_{t,\lambda}- S_{t,\lambda}. \label{EqDiffWeakErrors}
\end{align}
We find back the original definition of $\baloracle$.
\end{proof}

This alternative expression of $\baloracle$ motivates the choice of the residual-based stopping rule $\tau$, as explained in \cref{RemDiscrepancyPrinciple}. Note that $R_{\baloracle}^2 = \norm{\xi}^2 + 2 \scapro{\xi}{\mathsf{r_{\baloracle,<}} g}$ holds due to $A_{\baloracle, \lambda} = S_{\baloracle, \lambda}$.

\begin{remark} \label{RemDiscrepancyPrinciple}
The integer-valued discrepancy principle $\ceil{\tau}$ with $\kappa \coloneqq c^2 \delta^2$,  $c>1$, is widely used in  deterministic inverse problems, see \citet[Section~4.3]{EngHanNeu1996}.

In \cref{LemTauwRes}, the cross term $\scapro{\xi}{\mathsf{r_{\baloracle,<}}g}$ will be of negligible order, and we would like to set a threshold $\kappa$ close to $\norm{\xi}^2$ for the stopping rule $\tau$. In the Gaussian model \eqref{AssGN} we have $\E[\norm{\xi}^2]=\delta^2 D$ and  $\kappa \coloneqq \delta^2 D$ is a natural choice. In practice, the noise level $\delta^2$ needs to be estimated. As pursued in \citet{BlaHofRei2018EJS,BlaHofRei2018SIAM}, we analyse $\tau$ for general $\kappa >0$ and will later derive error bounds involving the deviation $\abs{\kappa - \delta^2 D}$.
\end{remark}

First, let us determine how much the prediction estimators at the stopping rule $\tau$ and at the balanced oracle $\baloracle$ differ.

\begin{proposition} \label{PropBoundDistanceEarlyStopWeakBal}
Under \eqref{AssGN}, we have
\begin{equation*}
\E\big[\norm{A(\widehat f_\tau-\widehat f_{\baloracle})}^2\big]\le 2\E[\abs{\scapro{\xi}{\mathsf{r_{\baloracle,<}}g}}] + \delta^2\sqrt{2D} + \abs{\kappa - \delta^2 D}.
\end{equation*}
\end{proposition}

\begin{proof}
In the case $\kappa \le \norm{Y}^2$, we have $R_{\tau}^2=\kappa$. Otherwise, if $\kappa > \norm{Y}^2$, we stop immediately at $\tau=0$ such that $0 \le R_{\tau}^2 - R_{\baloracle}^2 = \norm{Y}^2 - R_{\baloracle}^2 < \kappa - R_{\baloracle}^2$. Hence, we have in both cases the inequality
\begin{equation*}
\abs{R_{\tau}^2 - R_{\baloracle}^2} \le \abs{\kappa - R_{\baloracle}^2}.
\end{equation*}
\cref{LemPropSquaredResNorm:LemResNormMaxts} implies
\begin{align*}
\E\big[\norm{A(\widehat f_\tau-\widehat f_{\baloracle})}^2\big] &\le \E \big[ \abs{R_{\tau}^2 - R_{\baloracle}^2} \big] \notag \\
&\le \E\big[\abs{\kappa - \norm{\xi}^2 - 2\scapro{\xi}{\mathsf{r_{\baloracle,<}}g}}\big] \notag \\
&\le 2\E[\abs{\scapro{\xi}{\mathsf{r_{\baloracle,<}}g}}] +  \E\big[\abs{\norm{\xi}^2-\delta^2D}\big] + \abs{\kappa - \delta^2 D} \notag \\
&\le 2\E[\abs{\scapro{\xi}{\mathsf{r_{\baloracle,<}}g}}] + \delta^2\sqrt{2D} + \abs{\kappa - \delta^2 D},
\end{align*}
where we used $\E[\abs{\norm{\xi}^2-\delta^2D}] \le \Var(\norm{\xi}^2)^{1/2} = \delta^2\sqrt{2D}$ under \eqref{AssGN}.
\end{proof}

A first idea could be to bound the cross term via Cauchy--Schwarz inequality:
\begin{equation}\label{EqCrossTermCS}
\E[\abs{\scapro{\xi}{\mathsf{r_{\baloracle,<}}g}}]\le \E\big[\norm{\xi}^2\big]^{1/2}\E\big[\norm{\mathsf{r_{\baloracle,<}}g}^2\big]^{1/2}.
\end{equation}
Due to $\E[\norm{\xi}^2]^{1/2}=\delta\sqrt{D}$, this is unfortunately highly suboptimal under \eqref{AssGN}. If $r_{\baloracle,<}$ were independent of $\xi$ or even deterministic as for the cross term in linear methods \cite{BlaHofRei2018EJS,BlaHofRei2018SIAM}, then the inner product would have expectation zero and standard deviation $\delta \E[\norm{\mathsf{r_{\baloracle,<}}g}^2]^{1/2}$, a dimensionless and much tighter bound than \eqref{EqCrossTermCS}. 

In our case of intricate randomness of $r_{\baloracle,<}$, we are going to show a slightly larger bound that grows like $\sqrt{\log d}$ in the dimension, but still establishes a sufficiently low order of the cross term, see \cref{PropCrossTerm}. The probabilistic proof  uses concentration of self-normalised Gaussian processes by analysing the potential complexity of $\mathsf{r_{\baloracle,<}}g$ in $\R^D$. Before, a high probability upper bound on $\abs{r_t'(0)}$ will be established, giving in \cref{CorIterStlambda} an at least linear growth of $t\mapsto S_{t,\lambda}$.

In the sequel, we sometimes write $\lambda(i)^{-2}$ instead of $\lambda_{i}^{-2}$ for notational convenience, where $i \in \{1,\dots,D\}$.

\begin{lemma}\label{LemRhoBound}
Under \eqref{AssGN}, we have for all $z>0$ on an event of probability at least $1-e^{-z}$ the following two implications for all $t \in [0,d]$:
\begin{align*}
&\abs{r_t'(0)} < \lambda_D^{-2} \; \Rightarrow \; \abs{r_t'(0)} < \lambda\Big(\big(\floor{ \tfrac{2e}{e-1}\delta^{-2} S_{t,\lambda}+4(z+\log d)}+1\big) \wedge D\Big)^{-2}, \\
&\abs{r_t'(0)} \ge \lambda_D^{-2} \; \Rightarrow \; D \le \floor{ \tfrac{2e}{e-1}\delta^{-2} S_{t,\lambda}+4(z+\log d)}.
\end{align*}
\end{lemma}

\begin{proof}
For $t=0$, $\abs{r_t'(0)}=0 $ holds. Let $t>0$ and fix $z>0$.
By definition of $S_{t,\lambda}$ and \cref{LemrtBound}, the stochastic error can be lower bounded as
\begin{equation}
S_{t,\lambda} =\norm{\mathsf{(1- r_{t,<})^{1/2}}\xi}^2
\ge \norm{\mathsf{(1-\exp(-\abs{r_t'(0)}x))^{1/2}}\xi}^2
\ge (1-e^{-1})\sum_{i:\, \lambda_i^{-2}\le \abs{r_t'(0)}}\xi_i^2. \label{EqStochErrorLowerBound}
\end{equation}
Let $I_\rho \coloneqq \{i=1,\dots,D \,|\, \lambda_i^{-2}\le\rho\}$, $N(\rho) \coloneqq \#I_\rho$, and introduce the events
\begin{equation*}
\Omega(\rho) \coloneqq \Big\{\sum_{i\in I_\rho}(\delta^{-2}\xi_i^2-1) \ge -2\sqrt{N(\rho)(z+\log d)}\Big\}, \quad \rho>0.
\end{equation*}
By the $\chi^2$-concentration result of \citet[Lemma~1]{LauMas2000}, for fixed $\rho \ge \lambda_1^{-2}$, we obtain
\begin{equation*}
\forall\, \zeta>0:\;\PP \Big( \sum_{i\in I_\rho}(\delta^{-2}\xi_i^2-1) \le -2\sqrt{N(\rho)\zeta} \Big) \le e^{-\zeta}.
\end{equation*}
Since $\Omega(\rho)$ only changes at $\rho=\tilde{\lambda}_i^{-2}$, $i=1,\dots,d$, Bonferroni's inequality yields
\begin{equation*}
\PP \Big( \bigcap_{\rho > 0} \Omega(\rho) \Big) = \PP \Big( \bigcap_{i=1}^d \Omega(\tilde{\lambda}_i^{-2}) \Big) \ge 1- \sum_{i=1}^d \PP \big(\Omega(\tilde{\lambda}_i^{-2})^c\big)
\ge 1 - d e^{-(z+\log d)} = 1-e^{-z}.
\end{equation*}
On $\bigcap_{\rho > 0} \Omega(\rho)\subset \Omega(\abs{r_t'(0)})$, the definition of $\Omega(\abs{r_t'(0)})$ and \eqref{EqStochErrorLowerBound} give
\begin{equation*}
\delta^{-2}(1-e^{-1})^{-1}S_{t,\lambda} \ge N(\abs{r_t'(0)})-2\sqrt{(z+\log d)N(\abs{r_t'(0)})}.
\end{equation*}
The inequality
\begin{equation*}
N(\abs{r_t'(0)})
\le 2\Big(N(\abs{r_t'(0)})-2\sqrt{(z+\log d)N(\abs{r_t'(0)})}\Big)+4(z+\log d),
\end{equation*}
derived from the binomial formula, yields further
\begin{equation}
N(\abs{r_t'(0)}) \le 2\delta^{-2}(1-e^{-1})^{-1}S_{t,\lambda}+4(z+\log d). \label{EqStochErrorLowerBound2}
\end{equation}
In case $\abs{r_t'(0)} < \lambda_D^{-2}$, we additionally have $N(\abs{r_t'(0)}) \le D-1$ such that by definition $\abs{r_t'(0)} <\lambda(N(\abs{r_t'(0)})+1)^{-2}$. By monotonicity of $i \mapsto \lambda(i)^{-2}$ and \eqref{EqStochErrorLowerBound2}, we obtain
\begin{equation*}
\abs{r_t'(0)} < \lambda\Big(\big(\floor{2\delta^{-2}(1-e^{-1})^{-1}S_{t,\lambda}+4(z+\log d)}+1\big) \wedge D\Big)^{-2}.
\end{equation*}
In case $\abs{r_t'(0)} \ge \lambda_D^{-2}$, we have $N(\abs{r_t'(0)})=D$ such that \eqref{EqStochErrorLowerBound2} yields
\begin{equation*}
D \le \floor{2\delta^{-2}(1-e^{-1})^{-1}S_{t,\lambda}+4(z+\log d)}.
\end{equation*}
Therefore, both asserted implications hold for all $t \in [0,d]$ on the same event $\bigcap_{\rho > 0} \Omega(\rho)$, which has probability at least $1-e^{-z}$.
\end{proof}

\begin{corollary} \label{CorIterStlambda}
Under \eqref{AssGN}, we have for all $z>0$ with probability at least $1-e^{-z}$
\begin{equation*}
k <  \floor{\tfrac{2e}{e-1}\delta^{-2} S_{k,\lambda}+4(z+\log d)}+1,\quad k=0,\dots,d.
\end{equation*}
Assuming $\eqref{AssPSD}$ with $p>0$ and $\abs{r_k'(0)} < \lambda_D^{-2}$ in addition, we have on the same event
\begin{equation*}
k \lesssim_{p,c_A,C_A} (\delta^{-2}S_{k,\lambda})^{2p/(2p+1)} + (\log d)^{2p/(2p+1)} + z^{2p/(2p+1)}, \quad k =0,\dots,d.
\end{equation*}
\end{corollary}

\begin{proof}
Fix $z>0$. By \cref{LemPropInterpolResidualPolynomials:Formula} and \subcref{LemPropInterpolResidualPolynomials:ZerosEigenvalues:noninterpol}, we deduce
\begin{equation*}
\abs{r_k'(0)}=\sum_{i=1}^kx_{i,k}^{-1}\ge \sum_{i=1}^k\lambda_i^{-2}\ge\lambda_k^{-2}, \quad k=1,\dots,d.
\end{equation*}
Consider the event of \cref{LemRhoBound}, which holds with probability at least $1-e^{-z}$. For $k=0$, there is nothing to prove. Let $k=1,\dots,d$. In case $\abs{r_k'(0)} < \lambda_D^{-2}$, we then have, by the first implication of \cref{LemRhoBound}, the strict inequality
\begin{equation*}
\lambda_k^{-2} \le \abs{r_k'(0)} < \lambda\Big(\big(\floor{ \tfrac{2e}{e-1}\delta^{-2} S_{k,\lambda}+4(z+\log d)}+1\big) \wedge D\Big)^{-2}.
\end{equation*}
By the monotonicity of $i \mapsto \lambda(i)^{-2}$, we obtain the first claim. In case $\abs{r_k'(0)} \ge \lambda_D^{-2}$, noting that $k \le D$, the asserted bound is directly given by the second implication of \cref{LemRhoBound}.

Assuming additionally \eqref{AssPSD} with $p>0$, we have
$\abs{r_k'(0)} \ge \sum_{i=1}^k \lambda_i^{-2} \gtrsim_{p,C_A} k^{2p+1}$, $k=1,\dots,d$,
such that we can conclude the second claim again by the first implication of \cref{LemRhoBound} on the same event.
\end{proof}

We obtain the following high probability result for the inner product at general random `times' $\rv$. Its second part will be used for controlling the reconstruction error.

\begin{theorem} \label{ThmCrossTermBoundHighProb}
Consider the Gaussian noise model \eqref{AssGN}. Let $\rv$ be any random variable taking values in $[0,d]$ (possibly depending on $\xi$). Without any assumption on $(\lambda_i)$, we have for all $z \ge 1$ with probability at least $1-2e^{-z}$
\begin{equation}
\abs{\scapro{\xi}{\mathsf{r_{\rv,<}}g}} \le C \delta \norm{\mathsf{r_{\rv,<}}g} \Big(\sqrt{\delta^{-2}S_{\rv,\lambda}}+\sqrt{\log d}+\sqrt{z}\Big) \label{EqCrossTermBoundHighProb}
\end{equation}
with a numerical constant $C>0$. Assuming \eqref{AssPSD} with $p>0$ and $\abs{r_{\rv}'(0)}<\lambda_D^{-2}$ in addition, we have on the same event
\begin{equation}
\abs{\scapro{\xi}{\mathsf{r_{\rv,<}}g}} \le C_{p,c_A,C_A} \delta \norm{\mathsf{r_{\rv,<}}g} \big( (\delta^{-2}S_{\rv,\lambda})^{p/(2p+1)} + \sqrt{\log d} + \sqrt{z} \big) \label{EqCrossTermBoundHighProbPSD}
\end{equation}
with a constant $C_{p,c_A,C_A}>0$ depending only on $p$, $c_A$ and $C_A$.
\end{theorem}

\begin{proof}
Fix $z \ge 1$.  Let us consider the space
\begin{equation*}
Q_{g,k,i} \coloneqq \bigg\{ \sum_{j=i+1}^d p_k(\tilde{\lambda}_j^2) \sum_{l=1,\dots,D: \, \lambda_l=\tilde{\lambda}_j} \scapro{g}{u_l} u_l \,\bigg|\, p_k \in \Pol_k \bigg\}\subset \R^D
\end{equation*}
for $k,i=0,\dots,d$. Note that each $Q_{g,k,i}$ is an at most $(k+1)$-dimensional subspace, since $\Pol_k$ has dimension $k+1$. Recall that we set $x_{1,0} \coloneqq \infty$ and define
\begin{equation*}
i_{\rv} \coloneqq \sup\{i=1,\dots,d\,|\, \tilde{\lambda}_i^2\ge x_{1,\rv}\} \vee 0.
\end{equation*}
The definition of $r_{\rv,<}$ implies
\begin{equation*}
\mathsf{r_{\rv,<}}g
= \sum_{j=i_{\rv}+1}^d r_{\rv}(\tilde{\lambda}_j^2) \sum_{l=1,\dots,D: \, \lambda_l=\tilde{\lambda}_j} \scapro{g}{u_l}u_l \in Q_{g, \ceil{\rv}, i_{\rv}}.
\end{equation*}
Denote by $\Pi_{g,k,i}$ the orthogonal projection of $\R^D$ onto $Q_{g,k,i}$. By writing
\begin{equation}
\abs{\scapro{\xi}{\mathsf{r_{\rv,<}}g}} = \norm{\mathsf{r_{\rv,<}}g} \frac{\abs{\scapro{\xi}{\mathsf{r_{\rv,<}}g}}}{\norm{\mathsf{r_{\rv,<}}g}} \le \norm{\mathsf{r_{\rv,<}}g} \norm{\Pi_{g, \ceil{\rv}, i_{\rv}} \xi} \label{EqDecompCrossTerm}
\end{equation}
for $\mathsf{r_{\rv,<}}g \neq 0$, we can reduce the problem to bounding the norm of the projection $\norm{\Pi_{g, \ceil{\rv}, i_{\rv}} \xi} = \delta \norm{\Pi_{g, \ceil{\rv}, i_{\rv}} Z}$. If $\mathsf{r_{\rv,<}}g = 0$, the bound \eqref{EqDecompCrossTerm} on $\abs{\scapro{\xi}{\mathsf{r_{\rv,<}}g}}=0$ continues to hold. In case $\dim(Q_{g,k,i})>0$ for deterministic $k$ and $i$, note that $\E[\norm{\Pi_{g,k,i}Z}] \le \E[\norm{\Pi_{g,k,i}Z}^2]^{1/2} \le \sqrt{k+1}$ because $\norm{\Pi_{g,k,i}Z}^2$ is $\chi^2$-distributed with $\dim(Q_{g,k,i})$ degrees of freedom. Otherwise, $\norm{\Pi_{g,k,i}Z}^2=0$. Moreover, the function $x \mapsto \norm{\Pi_{g,k,i}x}$, $x \in \R^D$, is $1$-Lipschitz such that we obtain by Gaussian concentration \cite[Theorem~3.4]{Mas2007} and Bonferroni's inequality
\begin{align*}
&\PP \Big( \max_{0 \le k,i \le d} \big( \norm{\Pi_{g,k,i} \xi} - \delta \sqrt{k+1} \big) \ge \delta\big(2z+4\log (d+1)\big)^{1/2} \Big) \\
&\le \sum_{k,i=0}^d \PP \Big( \norm{\Pi_{g,k,i}Z}-\E[\norm{\Pi_{g,k,i}Z}] \ge \big(2z+4\log (d+1)\big)^{1/2} \Big) \\
&\le (d+1)^2 e^{-(2z+4\log (d+1))/2}= e^{-z}.
\end{align*}
Hence, with probability at least $1-e^{-z}$
\begin{align}
\norm{\Pi_{g,\ceil{\rv},i_{\rv}}\xi} &\le \delta \Big( \sqrt{\ceil{\rv}+1} + \big(2z+4\log (d+1)\big)^{1/2} \Big) \notag \\
&\lesssim \delta\big(\sqrt{\rv} + \sqrt{\log d} + \sqrt{z}\big) \label{EqUpperBoundNormProjection}
\end{align}
holds. Without any assumption on $(\lambda_i)$, we have by \cref{CorIterStlambda}
\begin{equation}
\rv \le \floor{\rv}+1 \lesssim \delta^{-2}S_{\rv,\lambda} + \log d + z \label{EqUpperBoundTau}
\end{equation}
also with probability at least $1-e^{-z}$, where we used that $t \mapsto S_{t,\lambda}$ is nondecreasing by \cref{LemPropErrorTerms:StochErrorNonDec}. Combining \eqref{EqDecompCrossTerm}--\eqref{EqUpperBoundTau} yields the first claim. For the second assertion, we use $\abs{r_{\floor{\rv}}'(0)} \le \abs{r_{\rv}'(0)} < \lambda_D^{-2}$ by assumption and obtain by \cref{CorIterStlambda} under \eqref{AssPSD}
\begin{equation}
\rv \le \floor{\rv}+1 \lesssim_{p,c_A,C_A} (\delta^{-2}S_{\rv,\lambda})^{2p/(2p+1)} +\log d+z \label{EqUpperBoundTauPolDecay}
\end{equation}
on the same event where \eqref{EqUpperBoundTau} holds. By \eqref{EqDecompCrossTerm}, \eqref{EqUpperBoundNormProjection} and \eqref{EqUpperBoundTauPolDecay}, the claim is deduced.
\end{proof}

We conclude the required bound in expectation for the  cross term.

\begin{proposition} \label{PropCrossTerm}
For the Gaussian noise model \eqref{AssGN}, without any assumption on $(\lambda_i)$, the bound
\begin{equation*}
\E[\abs{\scapro{\xi}{\mathsf{r_{\baloracle,<}}g}}] \le C \delta^2\big(\E[\delta^{-2}S_{\baloracle,\lambda}]+ \E[\delta^{-2}S_{\baloracle,\lambda}]^{1/2} \sqrt{\log d}\big)
\end{equation*}
holds true with a numerical constant $C>0$.
\end{proposition}

\begin{proof}
From \cref{Lemrkg} and the identity $A_{\baloracle,\lambda}=S_{\baloracle,\lambda}$, we deduce
\begin{equation*}
\norm{\mathsf{r_{\baloracle,<}}g}^2\le 8S_{\baloracle,\lambda}.
\end{equation*}
We obtain by \eqref{EqCrossTermBoundHighProb} and \cref{LemHighProbExpBound}
\begin{equation*}
\E[\abs{\scapro{\xi}{\mathsf{r_{\baloracle,<}}g}}] \lesssim \delta \E \bigg[ \norm{\mathsf{r_{\baloracle,<}}g} \Big(\sqrt{\delta^{-2}S_{\baloracle,\lambda}} + \sqrt{\log d} \Big) \bigg].
\end{equation*}
By the Cauchy--Schwarz inequality, we conclude
\begin{align*}
\E[\abs{\scapro{\xi}{\mathsf{r_{\baloracle,<}}g}}] &\lesssim \delta\E\big[\norm{\mathsf{r_{\baloracle,<}}g}^2\big]^{1/2} \E [\delta^{-2}S_{\baloracle,\lambda} + \log d]^{1/2} \\
&\lesssim \delta^2\big(\E[\delta^{-2}S_{\baloracle,\lambda}]+ \E[\delta^{-2}S_{\baloracle,\lambda}]^{1/2} \sqrt{\log d}\big). \qedhere
\end{align*}
\end{proof}

Finally, we can derive the oracle bound for the prediction error.

\begin{theorem} \label{ThmWeakOracleIneqEarlyStop}
For the Gaussian noise model \eqref{AssGN}, without any assumption on $(\lambda_i)$, the prediction error at the stopping rule $\tau$ satisfies the balanced oracle inequality
\begin{equation*}
\E\big[\norm{A(\widehat f_\tau-f)}^2\big]\le C \big(\E[S_{\baloracle,\lambda}]+\delta^2\sqrt D + \abs{\kappa - \delta^2 D}\big)
\end{equation*}
with a numerical constant $C>0$.
\end{theorem}

\begin{proof}
\cref{PropBoundDistanceEarlyStopWeakBal,PropCrossTerm} give
\begin{equation*}
\E \big[ \norm{A(\widehat f_\tau-\widehat{f}_{\baloracle})}^2 \big] \lesssim \E[S_{\baloracle,\lambda}]+\delta^2 \sqrt{D}+ \abs{\kappa - \delta^2 D}.
\end{equation*}
By \cref{PropWeakLoss} and the identity $A_{\baloracle,\lambda} = S_{\baloracle,\lambda}$, we have
\begin{equation*}
\E \big[ \norm{A(\widehat f_{\baloracle}-f)}^2 \big] \le 4 \E[S_{\baloracle,\lambda}].
\end{equation*}
Combining both upper bounds with the inequality
\begin{equation*}
\E \big[ \norm{A(\widehat f_{\tau}-f)}^2 \big] \le 2\E \big[ \norm{A(\widehat f_\tau-\widehat{f}_{\baloracle})}^2 \big] + 2 \E \big[ \norm{A(\widehat f_{\baloracle}-f)}^2 \big]
\end{equation*}
yields the claim.
\end{proof}

If $\abs{\kappa - \delta^2 D}$ is at most of the order of $\delta^2 \sqrt{D}$, the balanced oracle inequality immediately implies minimax adaptivity of the early stopping rule for the prediction error over Sobolev-type ellipsoids for a range of regularity parameters $\mu$.

\begin{corollary} \label{CorEarlyStopWeakMinimax}
Assume \eqrefAssPSDupper{} with $p>1/2$, \eqref{AssGN} and $\abs{\kappa-\delta^2 D} \le C_{\kappa}\delta^2 \sqrt{D}$ for some $C_{\kappa} \ge 0$. Under the source condition $\sourcecondf$, we have
\begin{equation*}
\E\big[\norm{A(\widehat f_\tau-f)}^2\big]\le C_{\mu,p,C_A,C_{\kappa}} \big(\predrate+\delta^2\sqrt D\big)
\end{equation*}
with $\predrate$ given in \eqref{EqDefPredRate} and a constant $C_{\mu,p,C_A,C_{\kappa}}>0$ depending only on $\mu$, $p$, $C_A$ and $C_{\kappa}$. Under \eqref{AssPSD}, this gives the  minimax optimal rate for all regularities ${\mu>0}$ such that $\sqrt{D} \lesssim (R^2 \delta^{-2})^{1/(4\mu p + 2p +1)} \lesssim_{\mu,p,c_A} D$.
\end{corollary}

\begin{proof}
Combining the bound on $\E[S_{\baloracle,\lambda}]$ from \cref{ThmWeakMinimax} with \cref{ThmWeakOracleIneqEarlyStop} yields the upper bound. If $\sqrt{D} \lesssim (R^2 \delta^{-2})^{1/(4\mu p + 2p +1)}$, then $\delta^2 \sqrt{D} \lesssim \predrate$, and we obtain minimax adaptivity for the given range of regularity parameters $\mu$ by \cref{PropWeakErrorLowerBound}.
\end{proof}

\begin{proposition}
Under \eqref{AssDN}, the discrepancy principle $\tau$ with $\kappa \coloneqq c^2 \delta^2$, $c\ge 0$, satisfies
\begin{equation*}
\norm{A(\widehat{f}_{\tau}-f)}^2 \le (c+1)^2 \delta^2.
\end{equation*}
\end{proposition}

\begin{proof}
The definition of $\tau$ and the assumption $\norm{\xi} \le \delta$ directly yield
\begin{equation*}
\norm{A(\widehat{f}_{\tau}-f)} = \norm{A\widehat{f}_{\tau}-Y+\xi} \le R_{\tau} + \norm{\xi} \le (c+1)\delta. \qedhere
\end{equation*}
\end{proof}


\section{Transfer to the reconstruction error} \label{SecTransferReconstructionError}

\subsection{Bound for the reconstruction error}

The main tool to control the reconstruction error is the following relatively rough bound in terms of an approximation error for $f^\dagger$ and the quantities already studied for the prediction error.

\begin{proposition} \label{PropRoughBoundStrongError}
For $t \in [0,d]$, we have
\begin{equation*}
\norm{\widehat{f}_t-f^\dagger}^2 \le 2 \norm{r_{t,<}(A^\top A)f^\dagger}^2 + 4\abs{r_t'(0)}S_{t,\lambda} + 2\abs{r_t'(0)}A_{t,\lambda}.
\end{equation*}
\end{proposition}

\begin{proof}
For $t=0$, the claim is clear. Let $t>0$.
Since $1-r_t$ is a polynomial vanishing in zero, we have
\[A^\dagger (1-r_t)(AA^\top)A=A^\top \Big(\tfrac{1-r_t(x)}{x}\Big)(AA^\top)A=(1-r_t)(A^\top A)\]
and thus
\begin{equation*}
\widehat f_t-f^\dagger =A^\dagger (1-r_t)(AA^\top)(Af^\dagger+\xi)-f^\dagger
= -r_t(A^\top A)f^\dagger+A^\dagger (1-r_t)(AA^\top)\xi.
\end{equation*}
We conclude by direct calculation
\begin{align}
&\norm{\widehat{f}_t-f^\dagger}^2 \notag \\*
&= \norm{\widehat{f}_t-f^\dagger}_<^2 + \norm{\widehat{f}_t-f^\dagger}_{\ge}^2 \notag \\
&\le \norm{r_t(A^\top A)f^\dagger - A^\dagger (\id_D -r_t(AA^\top))\xi}^2_< + x_{1,t}^{-1} \norm{A(\widehat{f}_t-f^\dagger)}^2 \notag \\
&\le 2\norm{r_{t,<}(A^\top A)f^\dagger}^2 + 2\norm{A^\dagger (\id_D -r_t(AA^\top))\xi}_<^2 + \abs{r_t'(0)} \norm{A(\widehat{f}_t-f)}^2, \label{EqRoughBoundStrongError}
\end{align}
where we used for the last inequality that $\abs{r_t'(0)} \ge x_{1,t}^{-1}$. Since additionally $x^{-1}(1-r_t(x))_< \le \abs{r_t'(0)}$ by \cref{LemrtBound}, the second term satisfies
\begin{equation}
\norm{A^\dagger (\id_D -r_t(AA^\top))\xi}_<^2 = \norm{\mathsf{x^{-1/2}(1-r_t)_<}\xi}^2
\le \abs{r_t'(0)} \norm{\mathsf{(1-r_t)_<^{1/2}}\xi}^2 \le \abs{r_t'(0)} S_{t,\lambda}.
\end{equation}
\cref{PropWeakLoss} yields for the last term
\begin{equation*}
\norm{A(\widehat{f}_t-f)}^2\le 2(A_{t,\lambda}+ S_{t,\lambda}).
\end{equation*}
It remains to combine these inequalities.
\end{proof}

\begin{remark}
The approach to bound the reconstruction error as in \eqref{EqRoughBoundStrongError} is similar to that in the proof of \citet[Lemma~7.11]{EngHanNeu1996}. Yet, instead of splitting the norm at $x_{1,t}$, they choose a positive $\varepsilon \le \abs{r_t'(0)}^{-1}$ such that their upper bound is minimal. This choice is necessary for their approach since they use, under the source condition $\sourcecondf$, the rough bound
\begin{equation*}
\norm{(\mathsf{r_t(x)}\mathbf{1}(\mathsf{x} \le \varepsilon))(A^\top A) f} \le \sup_{x \in [0,\varepsilon]} \abs{r_t(x)x^\mu} R \le \varepsilon^{\mu} R.
\end{equation*}
Instead, we will bound the term in \eqref{EqInterpolIneqWeakErrors} below with respect to the error terms. For the second summand in \eqref{EqRoughBoundStrongError}, they consider the worst-case bound
\begin{equation*}
\norm{\mathbf{1}(\mathsf{x \le \varepsilon}) \mathsf{x^{-1/2}(1-r_t(x))}\xi}^2 \le \abs{r_t'(0)} \norm{\xi}^2,
\end{equation*}
which is suboptimal in the stochastic noise setting \eqref{AssGN}.
\end{remark}

Using an interpolation inequality similar to \cite[(2.49)]{EngHanNeu1996}, we find the following bound.

\begin{corollary} \label{CorRoughBoundStrongErrorSobolevCond}
Assume the source condition $\sourcecondf$. Then, for $t \in [0,d]$, we have
\begin{equation*}
\norm{\widehat{f}_{t}-f^\dagger}^2 \le 16R^{2/(2\mu+1)}(S_{t,\lambda}\vee A_{t,\lambda})^{2\mu/(2\mu+1)} + 6\abs{r_t'(0)}(S_{t,\lambda}\vee A_{t,\lambda}).
\end{equation*}
\end{corollary}

\begin{proof}
By \cref{PropRoughBoundStrongError}, it remains to bound $\norm{r_{t,<}(A^\top A)f^\dagger}^2$. Note that $\scapro{f^\dagger}{v_i}=f_i$ for $i=1,\dots,D$ and $\scapro{f^\dagger}{v_i}=0$ for $i=D+1,\dots,P$ in case of $P>D$. Using the source condition $\sourcecondf$ and H\"older's inequality with the conjugate exponents $\tilde p \coloneqq (2\mu+1)/(2\mu)$ and $\tilde q \coloneqq 2\mu+1$, we obtain
\begin{align}
\norm{r_{t,<}(A^\top A)f^\dagger}^2 &= \sum_{i=1}^D r_{t,<}^2(\lambda_i^2)\lambda_i^{4\mu/(2\mu+1)}\lambda_i^{-4\mu/(2\mu+1)}f_i^2 \notag\\
&\le \Big(\sum_{i=1}^D r_{t,<}^2(\lambda_i^2)\lambda_i^{2}f_i^2\Big)^{2\mu/(2\mu+1)}\Big(\sum_{i=1}^D r_{t,<}^2(\lambda_i^2)\lambda_i^{-4\mu}f_i^2\Big)^{1/(2\mu+1)} \notag\\
&\le\norm{\mathsf{r_{t,<}}g}^{4\mu/(2\mu+1)}\Big(\sum_{i=1}^D \lambda_i^{-4\mu}f_i^2\Big)^{1/(2\mu+1)} \notag\\
&\le \norm{\mathsf{r_{t,<}}g}^{4\mu/(2\mu+1)}R^{2/(2\mu+1)}, \label{EqInterpolIneq}
\end{align}
where in the second last step $\abs{r_{t,<}(x)}\le 1$ was used. Invoking \cref{Lemrkg}, we thus have
\begin{equation}
2\norm{r_{t,<}(A^\top A)f^\dagger}^2\le 16R^{2/(2\mu+1)}(S_{t,\lambda}\vee A_{t,\lambda})^{2\mu/(2\mu+1)} \label{EqInterpolIneqWeakErrors}
\end{equation}
and conclude the claim.
\end{proof}

\begin{corollary} \label{CorBoundStrongErrorSourceCondDN}
Under \eqref{AssDN} and the source condition $\sourcecondf$, we have for $t \in (0,d]$
\begin{equation}
\norm{\widehat{f}_t-f^\dagger}^2 \le C_{\mu}\big( \abs{r_t'(0)}^{-2\mu}R^2 + \abs{r_t'(0)}\delta^2 \big) \label{EqPropBoundStrongErrorSourceCondDN}
\end{equation}
with a constant $C_{\mu}>0$ depending only on $\mu$.
\end{corollary}

\begin{proof}
Combining \cref{PropRoughBoundStrongError}, \eqref{EqInterpolIneq} and $S_{t,\lambda} \le \delta^2$ under \eqref{AssDN}, we obtain
\begin{equation*}
\norm{\widehat{f}_t-f^\dagger}^2 \lesssim \norm{\mathsf{r_{t,<}}g}^{4\mu/(2\mu+1)}R^{2/(2\mu+1)} + \abs{r_t'(0)} \delta^2 + \abs{r_t'(0)}A_{t,\lambda}.
\end{equation*}
Noting  $\norm{\mathsf{r_{t,<}}g} \le \norm{\mathsf{r_{t,<}^{1/2}}g}$ due to $\abs{r_{t,<}(x)} \le 1$, we conclude by \cref{CorUpperBoundApproxErrorSourceCond} and \eqref{EqUpperBoundApproxErrorSource}
\begin{align*}
\norm{\widehat{f}_t-f^\dagger}^2
&\lesssim_{\mu} \big(\abs{r_t'(0)}^{-2\mu-1}R^2\big)^{2\mu/(2\mu+1)}R^{2/(2\mu+1)} + \abs{r_t'(0)} \delta^2 + \abs{r_t'(0)}^{-2\mu}R^2 \\
&\lesssim_{\mu} \abs{r_t'(0)}^{-2\mu}R^2 + \abs{r_t'(0)}\delta^2. \qedhere
\end{align*}
\end{proof}

\begin{remark}
For small $t$, the first summand in the bound \eqref{EqPropBoundStrongErrorSourceCondDN} dominates and the iterates seem to converge until both summands are of the same order. From then on, the estimator diverges. This phenomenon is named \emph{semiconvergence} \cite[p.~45]{Han1995}. It corresponds to the tradeoff of approximation and stochastic error for the prediction error, see \cref{RemDerivativeRespolZero:tradeoff}.
\end{remark}

We are in a position to prove convergence rates for the reconstruction error under Gaussian noise.

\begin{theorem} \label{ThmStrongMinimax}
Assume \eqrefAssPSDupper{} with $p>1/2$ and \eqref{AssGN}. Under the source condition $\sourcecondf$, the reconstruction error at the balanced oracle satisfies
\begin{equation*}
\E\big[\norm{\widehat{f}_{\baloracle}-f^\dagger}^2\big] \le C_{\mu,p,C_A} \recrate
\end{equation*}
with $\recrate$ given in \eqref{EqDefRecRate} and a constant $C_{\mu,p,C_A}>0$ depending only on $\mu$, $p$ and $C_A$. In particular, under \eqref{AssPSD}, the balanced oracle is simultaneously minimax adaptive for prediction and reconstruction error for all $\mu>0$ with $(R^2 \delta^{-2})^{1/(4\mu p +2p+1)} \lesssim_{\mu,p,c_A} D$.
\end{theorem}

\begin{proof}
In the case $Af=0$, we have $\baloracle = 0$ such that $\widehat{f}_{\baloracle} = f^{\dagger} = 0$ and the bound is clear. Suppose $Af \neq 0$. Then $\baloracle > 0$ and $\abs{r_{\baloracle}'(0)}>0$.
\cref{CorRoughBoundStrongErrorSobolevCond} and the identity $S_{\baloracle,\lambda}=A_{\baloracle,\lambda}$ yield
\begin{equation*}
\E\big[\norm{\widehat{f}_{\baloracle}-f^\dagger}^2\big] \le 16 R^{2/(2\mu + 1)} \E[S_{\baloracle,\lambda}]^{2\mu/(2\mu+1)} + 6 \E[\abs{r_{\baloracle}'(0)} S_{\baloracle,\lambda}],
\end{equation*}
where we have applied Jensen's inequality. The first summand has the desired rate by \cref{ThmWeakMinimax}. Using \cref{LemPropErrorTerms:BoundStochError}, \eqref{EqUpperBoundApproxErrorSource} and \eqref{EqBoundExpWeakStochError}, we deduce similarly to the proof of \cref{ThmWeakMinimax} \pagebreak
\begin{align*}
&\E[\abs{r_{\baloracle}'(0)} S_{\baloracle,\lambda}] \\
&= \E[\abs{r_{\baloracle}'(0)} A_{\baloracle,\lambda} \wedge \abs{r_{\baloracle}'(0)} S_{\baloracle,\lambda}] \\
&\le \E\big[R^2(\mu+1/2)^{2\mu+1} \abs{r_{\baloracle}'(0)}^{-2\mu} \wedge \abs{r_{\baloracle}'(0)} \norm{(\mathsf{(\abs{r_{\baloracle}'(0)} x)^{1/2}\wedge 1})\xi}^2\big]\\
&\le \E\Big[\inf_{\rho>0}\big( R^2(\mu+1/2)^{2\mu+1} \rho^{-2\mu} \vee \rho \norm{(\mathsf{(\rho \mathsf x)^{1/2}\wedge \mathsf 1})\xi}^2\big)\Big]\\
&\le \inf_{\rho>0}\big( R^2(\mu+1/2)^{2\mu+1} \rho^{-2\mu} + (C_A \vee 1)^2\delta^2 \tfrac{2p}{2p-1} \rho^{(2p+1)/(2p)} \big) \\
&\lesssim_{\mu,p,C_A} R^{(4p+2)/(4 \mu p + 2p + 1)} \delta^{8\mu p/(4\mu p +2p+1)},
\end{align*}
which gives the right rate for the second summand. Under \eqref{AssPSD}, this rate is minimax optimal for the given range of regularity parameters $\mu$, see \cref{RemMinimaxReconstructionError}.
The minimax optimality for the prediction error was shown in \cref{ThmWeakMinimax,PropWeakErrorLowerBound}.
\end{proof}

\begin{remark} \label{RemMinimaxReconstructionError}
The error rate $\recrate$ coincides with the minimax rate in the infinite-dimensional Gaussian sequence space model, and the corresponding lower bound can be proven in the same way \cite[see][Proposition~4.23]{Joh2017}, very similar to the proof of \cref{PropWeakErrorLowerBound}:
under \eqrefAssPSDlower{} with $p>0$ and \eqref{AssGN}, the rate $\recrate$ is minimax optimal over signals $f$ satisfying $\sourcecondf$, provided $(R^2\delta^{-2})^{1/(4\mu p+2p+1)} \lesssim_{\mu,p,c_A} D$. If $D \ll (R^2\delta^{-2})^{1/(4\mu p+2p+1)}$, then $\delta^2 D^{2p+1} \ll \recrate$, and we obtain exactly as in \cref{RemMinimaxTruncationTimeWeak}:
\begin{equation*}
\E\big[ \norm{\widehat{f}_{\baloracle} - f^\dagger}^2 \big] \le \lambda_D^{-2} \E\big[ \norm{A(\widehat{f}_{\baloracle}-f^\dagger)}^2 \big] \lesssim_{c_A} \delta^2 D^{2p+1} \ll \recrate.
\end{equation*}
Also \cref{RemNoise2Moments} applies again for the upper bound: only conditions on the second moments of the $\xi_i$ are used and the bound is dimension-free.
\end{remark}

\subsection{Early stopping and its reconstruction error}

We analyse the reconstruction error of the early stopping rule $\tau$ from \cref{DefEarlyStopRule}. Unfortunately, there is no reconstruction error analogue of the oracle inequality in \cref{ThmWeakOracleIneqEarlyStop}. Instead, we need to directly bound $\abs{r_{\tau}'(0)}$ and the errors at $\tau$ in high probability to apply \cref{CorRoughBoundStrongErrorSobolevCond}. For this purpose \cref{LemRhoBound,ThmCrossTermBoundHighProb} are essential, requiring $\abs{r_{\tau}'(0)} < \lambda_D^{-2}$.
The case $\abs{r_{\tau}'(0)} \ge \lambda_D^{-2}$ leads to high stochastic error and is dealt with separately,
assuming the same lower bound on the dimension $D$ as for the lower bound on the reconstruction error in \cref{RemMinimaxReconstructionError}. The proofs for this subsection are delegated to \cref{SecProofsTransferReconstructionError}.

\begin{proposition} \label{PropBoundsStrongErrorEarlyStopped}
Assume \eqref{AssGN}, \eqref{AssPSD} and put $M_{\tau,\lambda}\coloneqq A_{\tau,\lambda}\vee S_{\tau,\lambda}$. Under the source condition $\sourcecondf$, the following holds:
\begin{enumproposition}
\item We have for $p>0$ and all $z > 0$ with probability at least $1-e^{-z}$ \label{PropBoundStrongErrorSource}
\begin{equation*}
\norm{\widehat{f}_\tau-f^\dagger}^2 \lesssim_{p,c_A} R^{2/(2\mu+1)}M_{\tau,\lambda}^{2\mu/(2\mu+1)}+\delta^{-4p}S_{\tau,\lambda}^{2p}M_{\tau,\lambda}
+(\log d)^{2p}M_{\tau,\lambda} + z^{2p}M_{\tau,\lambda}.
\end{equation*}
\item We have for $p>1/2$ and all $z > 0$ with probability at least $1-e^{-z}$ \label{PropUpperBoundEarlyStopStrongError}
\begin{equation*}
M_{\tau,\lambda} \lesssim_{\mu,p,C_A} R^{2/(4\mu p +2p +1)}\delta^{(8\mu p + 4p)/(4\mu p +2p+1)} + \abs{A_{\tau,\lambda}-S_{\tau,\lambda}} + \delta^2 z.
\end{equation*}
\item We have for $p>1/2$, $\abs{\kappa-\delta^2 D} \le C_{\kappa} \delta^2 \sqrt{D}$ for some $C_{\kappa} \ge 0$ and all $z \ge 1$ with probability at least $1-5e^{-z}$ \label{PropBoundKappa}
\begin{equation*}
M_{\tau,\lambda}
\lesssim_{\mu,p,c_A,C_A,C_{\kappa}} R^{2/(4\mu p+2p+1)}\delta^{(8\mu p +4p)/(4\mu p +2p+1)}+ \delta^2\big( \sqrt{D} \sqrt{z} + z^{(4p+1)/(4p)}\big),
\end{equation*}
provided $D\gtrsim (R^2 \delta^{-2})^{1/(4 \mu p +2p+1)}
$.
\end{enumproposition}
\end{proposition}

\begin{remark}\label{RemMDeltaBound}
The argument for \cref{PropBoundKappa} hinges on the fact that the bound on $\abs{A_{\tau,\lambda}-S_{\tau,\lambda}}$ derived in \eqref{EqBoundKappaStochErrorPolDecay} is of smaller order than the bound on $M_{\tau,\lambda}$ in \cref{PropUpperBoundEarlyStopStrongError}. This is provided by the bound \eqref{EqCrossTermBoundHighProbPSD}, which is valid under \eqref{AssPSD}, but not under exponential singular value decay.
\end{remark}

Inserting the high probability bound of \cref{PropBoundKappa}  into \cref{PropBoundStrongErrorSource} yields a high probability bound for the reconstruction error that in turn implies the main result on the reconstruction error under early stopping.

\begin{theorem} \label{ThmEarlyStopStrongMinimax}
Assume \eqref{AssPSD} with $p>1/2$, \eqref{AssGN} and $\abs{\kappa-\delta^2 D} \le C_{\kappa} \delta^2 \sqrt{D}$ for some $C_{\kappa} \ge 0$. Under the source condition $\sourcecondf$, we have for all $z \ge 1$ with probability at least $1-6e^{-z}$
\begin{equation*}
\norm{\widehat{f}_{\tau}-f^\dagger}^2 \le C_{\mu,p,c_A,C_A,C_{\kappa}}\Big(\recrate + \delta^2 \big(\sqrt{D} \sqrt{z} + z^{(4p+1)/(4p)}\big)^{2p+1}\Big)
\end{equation*}
and the bound in expectation
\begin{equation*}
\E\big[\norm{\widehat{f}_{\tau}-f^\dagger}^2\big] \le \widetilde{C}_{\mu,p,c_A,C_A,C_{\kappa}} \big(\recrate + \delta^2 D^{p+1/2}\big),
\end{equation*}
provided $(R^2 \delta^{-2})^{1/(4 \mu p +2p+1)} \lesssim D$, where $\recrate$ is given in \eqref{EqDefRecRate}. Here $C_{\mu,p, c_A, C_A,C_{\kappa}}$, $\widetilde{C}_{\mu,p, c_A, C_A,C_{\kappa}}$ are positive constants depending only on $\mu$, $p$, $c_A$, $C_A$ and $C_{\kappa}$. This gives the minimax rate for all regularities $\mu>0$ such that $\sqrt{D} \lesssim (R^2 \delta^{-2})^{1/(4\mu p +2p+1)} \lesssim_{\mu,p,c_A} D$.
\end{theorem}

\begin{remark} \label{RemEarlyStopStrongMinimax}
Comparing \cref{ThmEarlyStopStrongMinimax} and \cref{CorEarlyStopWeakMinimax}, we note that the early stopping rule $\tau$ is minimax adaptive both for prediction and reconstruction error over the same range of regularity parameters $\mu$.

The error term $\delta^2 D^{p+1/2}$ in \cref{ThmEarlyStopStrongMinimax} is due to the difference $\abs{\norm{\xi}^2-\delta^2D}$ and also appears in other early stopping regularisation methods \cite{BlaHofRei2018EJS,Sta2020}.
\end{remark}

\begin{proposition} \label{PropDiscrepancyPrincipleOrderOpt}
Under \eqref{AssDN} and the source condition $\sourcecondf$, the discrepancy principle $\tau$ with $\kappa \coloneqq c^2 \delta^2$, $c>1$, satisfies
\begin{equation*}
\norm{\widehat{f}_{\tau}-f^\dagger}^2 \le C_{\mu,c} R^{2/(2\mu+1)} \delta^{4\mu/(2\mu+1)},
\end{equation*}
where $C_{\mu,c}>0$ is a constant depending only on $\mu$ and $c$.
\end{proposition}

\begin{remark} \label{RemOrderOptRegMethod}
\citet[Theorem~7.12]{EngHanNeu1996} obtain the same result for noninterpolated CG, yet arguing sometimes differently. In addition, they have to control the difference between two iterates using the update polynomials, which we circumvent by interpolation.
\end{remark}


\section{Numerical examples} \label{SecNumericalIllustration}

We illustrate the performance of early stopping for the conjugate gradient algorithm by Monte-Carlo simulations.
First, we consider the same simulation setting as \citet{BlaHofRei2018EJS,BlaHofRei2018SIAM}: we assume the mildly ill-posed case \eqref{AssPSD} with $p=1/2$ and $c_A = C_A = 1$, i.e.\ $\lambda_i=i^{-1/2}$, $i=1,\dots,D$, with dimension  $D=10\, 000$ and Gaussian noise \eqref{AssGN} of level $\delta = 0.01$.

The SVD representations of the three signals from \cite{BlaHofRei2018EJS,BlaHofRei2018SIAM,Sta2020} are given by
\begin{align}
f_i^{(1)} &= 5 \exp(-0.1 i), \tag{supersmooth} \\
f_i^{(2)} &= 5000 \abs{\sin(0.01i)}i^{-1.6}, \tag{smooth} \\
f_i^{(3)} &= 250 \abs{\sin(0.002 i)} i^{-0.8}, \quad i=1,\dots,D. \tag{rough}
\end{align}
The names allude to the decay of the coefficients, interpreted as Fourier coefficients. For every signal we run $1000$ Monte-Carlo simulations.

In addition, we consider the test problem \emph{gravity} from the popular Matlab package \emph{Regularization Tools} \cite{Han2007,Han2010}. The one-dimensional gravity surveying problem leads to the Fredholm integral equation \cite[(2.1)]{Han2010} on $L^2([0,1])$
\begin{equation*}
\int_0^1 K(s,t) f^{(4)}(t) dt = g(s) \quad \text{with} \quad K(s,t) = \mathtt{d}(\mathtt{d}^2 + (s-t)^2)^{-3/2}, \quad s,t \in [0,1],
\end{equation*}
which is discretised by the midpoint quadrature rule. The ill-posedness is controlled by the parameter $\mathtt{d}$. We choose the default value $\mathtt{d} = 0.25$. The true signal (\emph{mass distribution}) is given by
\begin{equation*}
f^{(4)}(t) = \sin(\pi t) + 0.5 \sin(2\pi t), \quad t \in [0,1].
\end{equation*}
We obtain discretisations $A \in \R^{D \times D}$ and $f^{(4)} \in \R^D$ of the forward operator and the signal, respectively, for dimension $D = 4096$. Under \eqref{AssGN} with noise level $\delta = 0.01$, we run $1000$ Monte-Carlo simulations. \citet{Jah2022} considers the same example for the early stopped Landweber (classical gradient descent) algorithm.

\begin{figure}
\centering
\includegraphics{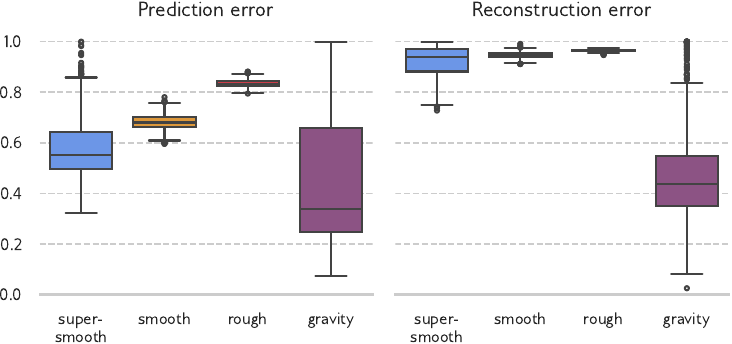}
\caption{Relative efficiencies of early stopping for prediction and reconstruction error (higher is better)}\label{FigRelativeEfficiencies}
\end{figure}

In both settings, we calculate  the conjugate gradient estimator $\widehat{f}_{\tau}$, where the threshold for the early stopping rule is chosen as $\kappa \coloneqq \delta^2 D$ for the first three examples and $\kappa \coloneqq \delta^2 D + \delta^2 \sqrt{D}$ for the gravity problem. At the latest, the algorithm is terminated at step $T \in \N$ if $T$ reaches the maximal iteration index $D$ or if the squared norm $\norm{A^\top (Y-A\widehat{f}_{T})}^2$ of the transformed residual vector is smaller than $10^{-8}$ (\emph{emergency stop}). Otherwise, it would break down with division by zero \cite[p.~179]{EngHanNeu1996}. Note that for
ill-conditioned matrices, $R_t^2$ is not necessarily close to zero for $\norm{A^\top (Y-A\widehat{f}_{t})}^2$ to be small.

We measure how close the estimator $\widehat{f}_\tau$ is to the optimal iterate along the (continuous) interpolation path $(\widehat f_t,t\in[0,T])$ by calculating  the \emph{relative efficiencies}
\begin{equation*}
\frac{\min_{t \in [0,T]}\norm{A(\widehat{f}_t-f)}}{\norm{A(\widehat{f}_{\tau}-f)}} \quad \text{and} \quad \frac{\min_{t \in [0,T]}\norm{\widehat{f}_t-f}}{\norm{\widehat{f}_{\tau}-f}}
\end{equation*}
for the prediction and reconstruction error, respectively. The values lie in $[0,1]$ and ideally should be close to one.

First, consider the performance for the signals $f^{(1)},f^{(2)},f^{(3)}$ in \cref{FigRelativeEfficiencies}. The relative efficiencies are surprisingly high, especially for the reconstruction error. Note that a relative efficiency larger than $0.5$ means that the error is less than twice the optimal error. For both criteria the supersmooth signal suffers from higher variability than the other two due to its very small optimal errors, compare  \cref{TabSimulations}. \citet{BlaHofRei2018SIAM} analyse the same stopping rule for the Landweber algorithm. Their results for the three signals are very similar, conjugate gradients even perform slightly better.

The relative efficiencies for the gravity example show more variability. Note that this inverse problem is severely ill-posed with exponentially decaying singular values violating our assumptions for the adaptivity in reconstruction error of early stopping. Additionally, the optimal oracle errors are even smaller than for the supersmooth signal $f^{(1)}$, see \cref{TabSimulations}. So, the gravity test problem serves as an example where  the theory lets us expect bad performance. Further problems arise from machine precision: although in theory the residuals reach zero at iteration $D$ at the latest, in our calculations the residuals tend to stagnate at some nonzero value, which might be larger than the usual threshold $\delta^2 D$. In these cases, the algorithm is terminated via the aforementioned emergency stop. In unreported simulations with $\kappa \coloneqq \delta^2 D$, this happened in $44$ percent of the runs.

\cref{RemMDeltaBound} explains that we cannot expect $\abs{S_{\tau,\lambda}-A_{\tau,\lambda}}$ to be small for severely ill-posed problems so that we loose control of $\abs{r_{\tau}'(0)}$. Therefore, we recommend to stop slightly earlier and we choose here $\kappa \coloneqq \delta^2 D + \delta^2 \sqrt{D}$, still satisfying our condition on $\kappa$ in \cref{CorEarlyStopWeakMinimax} with $C_{\kappa} = 1$, leading to emergency stops in $21$ percent of the runs. If we ignore these runs, the minimal relative efficiency for the reconstruction error is about $0.13$ such that in the worst case the reconstruction error is eight times the optimal error.
The interesting conclusion is that the early stopping criterion provides reasonable results even in a setting not covered by the theory.

\begin{table}%
\centering%
{\small%
\begin{tabular}{l@{\hskip 5pt}c@{\hskip 3.5pt}c@{\hskip 3.5pt}c@{\hskip 3.5pt}c@{\hskip 3.5pt}c@{\hskip 3.5pt}c@{\hskip 3.5pt}c}
\toprule%
& \multicolumn{3}{@{}c@{}}{Stopping indices} & \multicolumn{4}{@{}c@{}}{Corresponding errors} \\\cmidrule{2-4}\cmidrule{5-8}%
& $\predoracle$ & $\recoracle$ & $\tau$ & $\norm{A(\widehat{f}_{\predoracle}-f)}$ & $\norm{A(\widehat{f}_{\tau}-f)}$ & $\norm{\widehat{f}_{\recoracle}-f}$ & $\norm{\widehat{f}_{\tau}-f}$ \\
\midrule
$f^{(1)}$ & $6.48$\,($0.06$) & $5.56$\,($0.04$) & $5.07$\,($0.24$) & $0.1$\,($0.01$) & $0.18$\,($0.02$) & $0.81$\,($0.02$) & $0.87$\,($0.05$) \\
$f^{(2)}$ & $15.42$\,($0.61$) & $12.57$\,($0.04$) & $11.45$\,($0.07$) & $0.28$\,($0.01$) & $0.41$\,($0.01$) & $5.68$\,($0.06$) & $6.01$\,($0.08$) \\
$f^{(3)}$ & $19.5$\,($0.19$) & $17.34$\,($0.05$) & $14.15$\,($0.05$) & $0.56$\,($0.01$) & $0.67$\,($0.01$) & $21.86$\,($0.11$) & $22.69$\,($0.12$) \\
$f^{(4)}$ & $12.45$\,($1.05$) & $12.43$\,($0.81$) & $6.97$\,($2.54$) & $0.03$\,($0.01$) & $0.09$\,($0.03$) & $0.5$\,($0.1$) & $1.18$\,($0.3$) \\
\bottomrule
\end{tabular}}
\caption{Medians (mean absolute deviations) of stopping indices and errors for different signals}
\label{TabSimulations}
\end{table}

In \cref{TabSimulations}, we state the medians (and mean absolute deviations around the median) of the optimal (noise-dependent) stopping indices
\begin{equation*}
\predoracle = \argmin_{t \in [0,T]} \norm{A(\widehat{f}_t-f)} \quad \text{and} \quad \recoracle = \argmin_{t \in [0,T]} \norm{\widehat{f}_t-f}
\end{equation*}
and of the stopping rules $\tau$ as well as of the corresponding root mean squared errors for the Monte-Carlo iterations.
For all signals we tend to stop rather too early than too late compared to $\predoracle$ and $\recoracle$. \citet{BlaHofRei2018SIAM} argue that in case of the Landweber method for the first three signals a slightly smaller value of $\kappa$ is appropriate.

Finally, we analyse the behaviour of the errors for $f^{(1)},f^{(2)},f^{(3)}$ in the asymptotic regime where the dimension $D(\delta)$ grows as the noise level $\delta$ vanishes. Following the asymptotical setting of \citet{Sta2020}, we choose the dimensions $D = D_m = 100 \cdot 2^m$ for $m=0,\dots,10$ with noise levels
\begin{equation*}
\delta_m = R D_m^{-2 \mu p -p - 1/2} \quad \text{such that} \quad D_m=(R^2 \delta_m^{-2})^{1/(4\mu p + 2p +1)}, \quad m=0,\dots,10,
\end{equation*}
where $R=1000$, $\mu = 1/4$ and $p=1/2$. For this choice of $\mu$, the rough signal $f^{(3)}$ fulfils the source condition \eqref{EqDefSourceCond}. The dimensions $D_m$ satisfy the condition for the lower bounds in \cref{PropWeakErrorLowerBound,RemMinimaxReconstructionError} with an equality. Setting $D=10 \, 000$ would yield $\delta=0.01$, as in  the previous setting. For each noise level, we calculate the Monte-Carlo means (in $1000$ runs) of the errors $\norm{A(\widehat{f}_{\tau}-f)}^2$ and $\norm{\widehat{f}_{\tau}-f}^2$ and compare them to the Monte-Carlo means of the minimal errors $\norm{A(\widehat{f}_{\predoracle}-f)}^2$ and $\norm{\widehat{f}_{\recoracle}-f}^2$, respectively. The rates of the former are plotted in \cref{FigConvergenceRates} as solid lines, those of the latter as dashed lines. Additionally, we depict the slopes of the theoretical minimax rates $\predrate = R^2 D^{-4 \mu p -2p}$ and $\recrate = R^2 D^{-4 \mu p}$ for prediction and reconstruction error, respectively. Note that the log-log plots amplify differences between rates for small values.

\begin{figure}
\centering
\includegraphics{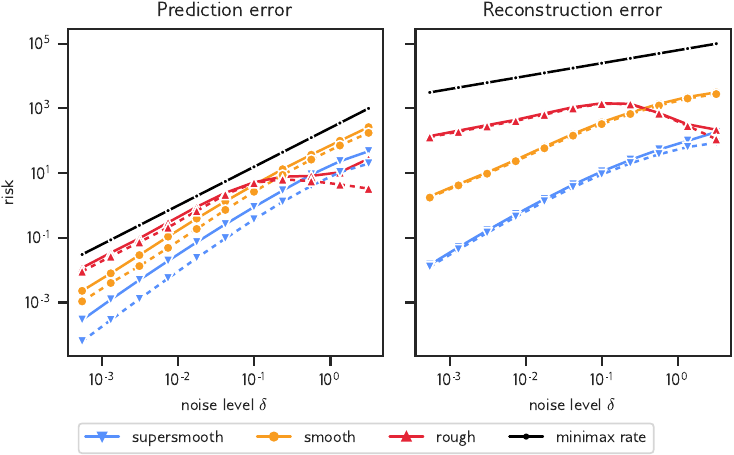}
\caption{Log-log plots of the prediction and reconstruction error rates for $f^{(1)}$, $f^{(2)}$ and $f^{(3)}$ (solid lines: risk estimates at $\tau$, dashed lines: risk estimates at $\predoracle$ and $\recoracle$, respectively)} \label{FigConvergenceRates}
\end{figure}

For all three signals and both error criteria, the early stopped estimators closely match the optimal ones. For the rough signal, the asymptotic regime applies only for small $\delta$, that is large $D$. This is explained by the fact that the signals still change with $D$, having more nonzero coefficients and thus inducing larger approximation errors.  Again, the similarity is even higher for the reconstruction error. Note that the choice $\mu=1/4$ corresponds to the rough signal, for which the slopes of the minimax rates are well-matched. For the smoother signals, the convergence is faster, as expected for an adaptive method.


\section{Remaining results and proofs}\label{SecAppendix}

For the convenience of the reader, the next two subsections collect and sometimes prove classical results for the CG algorithm.

\subsection{Results for zeros of orthogonal polynomials} \label{SecZerosOrthoPolynomials}

We consider the Wronskian
\begin{equation}
W(x; p,q) \coloneqq p(x)q'(x) - p'(x)q(x), \quad x \in \R, \label{DefWronskian}
\end{equation}
of two polynomials $p$ and $q$ on $\R$. Note that the Wronskian $W(x;p,q)$ is continuous in~$x$.
The following result is due to \citet[Section~4]{BeaDri2005}.

\begin{lemma}[Properties of the Wronskian] \label{LemPropWronskian}
If $W(x;p,q)$ is nonzero in an open interval $J \subset \R$, then
\begin{enumlemma}
\item $p$ and $q$ have no common zeros in $J$,
\item any zero of $p$ and $q$ in $J$ is a simple zero,
\item any two consecutive zeros of $p$ in $J$ are separated by a zero of $q$,
\item any two consecutive zeros of $q$ in $J$ are separated by a zero of $p$.
\end{enumlemma}
\end{lemma}

Let $\mu$ denote a discrete measure on the interval $[a,b] \subset \R$ with a support consisting of $N < \infty$ points. Denote by $\norm{\cdot}_{L^2(\mu)}$ the $L^2$-norm with respect to $\mu$.

\begin{lemma} \label{LemWronskianOrthPolynomials}
Let $(p_n)_{n=0}^N$ be a sequence of orthogonal polynomials in $L^2([a,b],\mu)$ with $\norm{p_n}_{L^2(\mu)} > 0$ for $n=0,\dots,N-1$. Then the Wronskian $W(x;p_n,p_{n+1})$, $n=0,\dots,N-1$, of two consecutive orthogonal polynomials does not vanish on $(a,b)$.
\end{lemma}

\begin{proof}
Denote by $k_m$, $m=0,\dots,N$, the highest coefficient of $p_m$. Since $\norm{p_n}_{L^2(\mu)} > 0$, $n=0,\dots,N-1$, we obtain by the Christoffel--Darboux formula \cite[Theorem~3.2.2]{Sze1939} that
\begin{align*}
W(x;p_n,p_{n+1}) &= p_n(x) p_{n+1}'(x) - p_n'(x) p_{n+1}(x) \\
&= \frac{\norm{p_n}_{L^2(\mu)}^2 k_{n+1}}{k_n} \bigg(\sum_{i=1}^n \frac{p_i^2(x)}{\norm{p_i}_{L^2(\mu)}^2} + \frac{k_0^2}{\norm{p_0}_{L^2(\mu)}^2}\bigg) \neq 0. \qedhere
\end{align*}
\end{proof}

\subsection{Results for  classical conjugate gradients} \label{SecClassicalConjugateGradientMethod}

Recall \cref{DefEstimators} of the noninterpolated estimators.

\begin{lemma} \label{LemRkWellDef}
The residual polynomial $r_k$ is well-defined for $k=0,\dots,d$. Moreover, $r_d$ is uniquely determined by its zeros $0 < \tilde\lambda_d^2 < \dots < \tilde\lambda_1^2 = \norm{AA^\top}_{\mathrm{spec}}$.
\end{lemma}

\begin{proof}
For $k < d$, the orthogonal projection with respect to $\scapro{\cdot}{\cdot}_Y$ of the origin on the affine subspace $\Pol_{k,1}$ of $\Pol_k$ uniquely solves the minimisation problem $\min_{p_k \in \Pol_{k,1}} \norm{\mathsf{p_k}}_Y^2$. Therefore, $r_k$ is well-defined for $k<d$. In the case $k=d$, note that $\min_{p_d \in \Pol_{d,1}} \norm{\mathsf{p_d}}_Y^2 = 0$, since $A$ has $d$ distinct nonzero singular values. By Assumption~\eqref{AssY}, the solution is uniquely defined, namely as the polynomial in $\Pol_{d,1}$, which has the squared singular values $\tilde{\lambda}_i^2$, $i=1,\dots,d$, as zeros. This proves the well-definedness and claimed characterisation of $r_d$.
\end{proof}

\begin{lemma} \label{LemOrthogonalityY}
The residual polynomial $r_k$ is orthogonal to $\Pol_{k,0}$ with respect to $\scapro{\cdot}{\cdot}_Y$ for $k=0,\dots,d$. We write $r_k\perp_Y \Pol_{k,0}.$
\end{lemma}

\begin{proof}
First, consider the case $k<d$. By the definition of $r_k$, the polynomial $r_k-1$ is the orthogonal projection with respect to $\scapro{\cdot}{\cdot}_Y$ of $-1$ onto the linear subspace $\Pol_{k,0}$. Therefore, $r_k = r_k-1 - (-1)$ is orthogonal to $\Pol_{k,0}$ with respect to $\scapro{\cdot}{\cdot}_Y$. For $k=d$, the claim follows from \cref{LemRkWellDef}.
\end{proof}

\cref{LemOrthogonalityY} implies the orthogonality of the residual polynomials. The approach of analysing the CG iterate by using the properties of zeros of orthogonal polynomials dates back to the seminal work of \citet{Nem1986} for deterministic inverse problems, which was further developed by \citet{Han1995} and \citet{EngHanNeu1996}.

\begin{corollary} \label{CorResidualPolynomialsOrth}
The sequence $(r_k)_{k=0,\dots,d}$ of residual polynomials is conjugate with respect to $AA^\top$ and $\scapro{\cdot}{\cdot}_Y$, i.e.
\begin{equation*}
\scapro{\mathsf{r_k}}{\mathsf{xr_l}}_Y = \norm{\mathsf{x^{1/2} r_l}}_Y^2 \delta_{l k}, \quad 0 \le k \le d,\ 0 \le l < d,
\end{equation*}
where $\norm{\mathsf{x^{1/2} r_l}}_Y > 0$ for all $l=0,\dots,d-1$. This implies that $r_k$ has degree $k$ for $k=0,\dots,d$.
\end{corollary}

\begin{proof}
Without loss of generality, assume $l < k$. Since $x \mapsto xr_l(x) \in \Pol_{l+1,0} \subset \Pol_{k,0}$, \cref{LemOrthogonalityY} yields $\scapro{\mathsf{r_k}}{\mathsf{x r_l}}_Y=0$.
For $l<d$, there is $i^* \in \{1,\dots,d\}$ such that $r_l(\tilde{\lambda}_{i^*}^2) \neq 0$, since a polynomial of degree at most $l$ can have at most $l$ zeros. Hence, by Assumption~\eqref{AssY}, $\norm{\mathsf{x^{1/2} r_l}}_Y^2 \ge \tilde{\lambda}_{i^*}^2 r_l(\tilde{\lambda}_{i^*}^2)^2 \sum_{j: \, \lambda_j = \tilde{\lambda}_{i^*}}  Y_{j}^2 > 0$. For the second part, assume $r_k$ has degree $l < k$. Then, by  \cref{LemRkWellDef}, we have $r_k = r_l$ such that $\scapro{\mathsf{r_k}}{\mathsf{xr_l}}_Y = \norm{\mathsf{x^{1/2} r_l}}_Y^2 > 0$. On the other hand, we have shown $\scapro{\mathsf{r_k}}{\mathsf{xr_l}}_Y = 0$, which yields a contradiction. This concludes the proof.
\end{proof}

Note that $\norm{\mathsf{x^{1/2} r_d}}_Y=0$. The following properties of the noninterpolated residual polynomials are well-known \cite[Section~7.2 and Appendix~A.2]{EngHanNeu1996}.

\begin{lemma}[Properties of the residual polynomial $r_k$] \leavevmode \label{LemPropResidualPolynomials}
\begin{enumlemma}
\item $r_k$ has $k$ real simple zeros $0<x_{1,k}< \dots < x_{k,k} < \norm{AA^\top}_{\mathrm{spec}}$, $k=1,\dots,d-1$. For $k=d$, the zeros are the $d$ distinct squared nonzero singular values $0 < x_{1,d} = \tilde\lambda_d^2 < \dots < x_{d,d} = \tilde\lambda_1^2 = \norm{AA^\top}_{\mathrm{spec}}$. \label{LemPropResidualPolynomials:simpleZeros}
\item $r_k$, $k=1,\dots,d$, can be written as \label{LemPropResidualPolynomials:Formula}
\begin{equation*}
r_k(x) = \prod_{i=1}^k \Big( 1-\frac{x}{x_{i,k}} \Big), \ x \in [0,\norm{AA^\top}_{\mathrm{spec}}], \quad \text{with} \quad \abs{r_k'(0)}=\sum_{i=1}^k x_{i,k}^{-1}.
\end{equation*}
\item $r_k$ is nonnegative, decreasing and convex on $[0,x_{1,k}]$ and log-concave on $[0,x_{1,k})$ for $k=1,\dots,d$.\label{LemPropResidualPolynomials:Properties}
\item For $k=1, \dots, d-1$, we have \label{LemPropResidualPolynomials:interlacedZeros}
\begin{equation*}
x_{1,k+1} < x_{1,k} < x_{2,k+1} < \dots < x_{k,k+1} < x_{k,k} < x_{k+1,k+1}.
\end{equation*}
\item $r_k$ is lower bounded by $r_{k+1}$ on $[0, x_{1,k+1}]$. In particular, $r_{k}(x) > r_{k+1}(x)$ for $0 < x \le x_{1,k+1}$ and $r_k(0)=r_{k+1}(0)=1$, $k=1,\dots,d-1$. \label{LemPropResidualPolynomials:lowerBounded}
\item For $k=1,\dots,d$, $i=1,\dots,k$, we have \label{LemPropResidualPolynomials:ZerosEigenvalues}
\begin{equation*}
x_{k+1-i,k} \le \lambda_i^2.
\end{equation*}
\end{enumlemma}
\end{lemma}

\begin{proof} \leavevmode
The properties \subcref{LemPropResidualPolynomials:simpleZeros}--\subcref{LemPropResidualPolynomials:interlacedZeros} besides the log-concavity in \subcref{LemPropResidualPolynomials:Properties} are well-known, see \linebreak \citet[Section~3.2]{Nem1986}, \citet[Section~2.4]{Han1995} and \citet[Appendix~A.2 for \subcref{LemPropResidualPolynomials:simpleZeros} and \subcref{LemPropResidualPolynomials:interlacedZeros}, (7.5) and (7.6) for \subcref{LemPropResidualPolynomials:Formula}]{EngHanNeu1996}.
\begin{enumerate}
\item \cref{CorResidualPolynomialsOrth} shows that $(r_k)_{k=0,\dots,d}$ is orthogonal in $L^2([0,\norm{AA^\top}_{\mathrm{spec}}], \hat{\mu})$ with $\hat{\mu} \coloneqq \sum_{i=1}^D \lambda_i^2 Y_i^2 \delta_{\lambda_i^2}$  and $\norm{r_k}_{L^2(\hat{\mu})}>0$ for $k=0,\dots,d-1$. The first part for $k<d$ is a well-known fact for orthogonal polynomials, see \citet[Theorem~3.3.1]{Sze1939}. We conclude the second part by \cref{LemRkWellDef}.
\item Using that all zeros of $r_k$ are real by \subcref{LemPropResidualPolynomials:simpleZeros} and $r_k(0)=1$, we obtain the given formula. Taking the derivative yields
\begin{equation*}
r_k'(x) = \sum_{i=1}^k \bigg(  \Big(-\frac{1}{x_{i,k}}\Big) \prod_{\substack{1 \le j \le k \\ j \neq i}} \Big( 1-\frac{x}{x_{j,k}} \Big) \bigg).
\end{equation*}
This gives the claim by plugging in $x=0$.
\item The first two properties follow immediately from \subcref{LemPropResidualPolynomials:Formula}. For the convexity note that $r_k'$ is nondecreasing on $[0,x_{1,k}]$. For the log-concavity we observe that
\begin{equation*}
\log(r_k(x))=\sum_{i=1}^{k}\log\Big(1-\frac{x}{x_{i,k}}\Big)
\end{equation*}
is concave on $[0,x_{1,k})$ as a sum of concave functions.
\item Recall the definition of the Wronskian given in \eqref{DefWronskian} and that $(r_k)_{k=0,\dots,d}$ is orthogonal in $L^2([0,\norm{AA^\top}_{\mathrm{spec}}], \hat{\mu})$ with $\norm{r_k}_{L^2(\hat{\mu})}>0$ for $k=0,\dots,d-1$.
For $k=1,\dots,d-1$, \cref{LemWronskianOrthPolynomials} yields
\begin{equation}
W(x; r_k, r_{k+1}) \neq 0, \quad x \in (0,\norm{AA^\top}_{\mathrm{spec}}), \label{EqWronskianUneqZero}
\end{equation}
such that by \cref{LemPropWronskian} the zeros of $r_k$ and $r_{k+1}$ in $(0, \norm{AA}^\top)$ are distinct and interlaced. For $k \le d-2$, all zeros of $r_k$ and $r_{k+1}$ are contained in $(0, \norm{AA}^\top)$ by \subcref{LemPropResidualPolynomials:simpleZeros}, which gives the claim. For $k=d-1$ there are a-priori two possibilities, namely
\begin{equation*}
x_{1,d} < x_{1,d-1} < x_{2,d} < x_{2,d-1} < \dots < x_{d-1,d} < x_{d-1,d-1} < x_{d,d},
\end{equation*}
and
\begin{equation*}
x_{1,d-1} < x_{1,d} < x_{2,d-1} < x_{2,d} < \dots < x_{d-1,d-1} < x_{d-1,d} < x_{d,d}.
\end{equation*}
Assume the second case, where, in particular, $x_{d-1,d-1} \in (x_{d-2,d}, x_{d-1,d})$. By considering the two cases whether $d$ is even or odd separately, one can show that $W(x_{d-1,d-1}; r_{d-1}, r_d) > 0$ and $W(x_{d,d}; r_{d-1}, r_d) < 0$. Since the Wronskian is continuous, it has a zero in $(x_{d-1,d-1}, x_{d,d}) \subset (0,\norm{AA^\top}_{\mathrm{spec}})$. This contradicts \eqref{EqWronskianUneqZero}, which shows the claim.
\item The interlacing property \subcref{LemPropResidualPolynomials:interlacedZeros} of the zeros of $r_k$ implies that
\begin{equation*}
r_{k+1}(x)=\prod_{i=1}^{k+1}\Big(1-\frac{x}{x_{i,k+1}}\Big) < \prod_{i=1}^{k}\Big(1-\frac{x}{x_{i,k}}\Big)=r_k(x)
\end{equation*}
for $0 < x \le x_{1,k+1}$, which shows the claim.
\item We prove the assertion by decreasing induction over $k$, starting in $k=d$.
By definition, we have $\tilde{\lambda}_i \le \lambda_i$, $i=1,\dots,d$, and by \subcref{LemPropResidualPolynomials:simpleZeros}, the zeros of $r_d$ are
\begin{equation*}
x_{d+1-i,d} = \tilde{\lambda}_i^2 \le \lambda_i^2, \quad i=1,\dots,d.
\end{equation*}
This gives the claim for $k=d$. Let $k \in \{2,\dots,d\}$ and assume as induction hypothesis that
\begin{equation*}
x_{k+1-i,k} \le \lambda_i^2, \quad i=1,\dots,k.
\end{equation*}
For the induction step from $k$ to $k-1$, we obtain by \subcref{LemPropResidualPolynomials:interlacedZeros}
\begin{equation*}
x_{(k-1)+1-i,k-1} < x_{k+1-i,k}, \quad i=1,\dots,k-1,
\end{equation*}
such that the claim for $k-1$ follows from the induction hypothesis. \qedhere
\qedhere
\end{enumerate}
\end{proof}

The following result is originally due to \citet[(3.14)]{Nem1986}, who introduced the function $\varphi_k^2$, named $r_k$ therein. The proof of the first inequality can also be found in \citet[(3.8)]{Han1995} and \citet[(7.7)]{EngHanNeu1996}. For the sake of completeness, we state it in our notation below. The second inequality is new, but follows easily from \cref{LemPropResidualPolynomials:Formula} and the definition of $\varphi_k$.

\begin{lemma} \label{LemRkbound}
The residual polynomial satisfies
\begin{equation*}
\norm{\mathsf{r_k}}_Y^2 < \norm{\mathsf{\varphi_k}}_Y^2 \le \norm{\mathsf{r_{k,<}^{1/2}}}_Y^2
\end{equation*}
for $k=1,\dots,d$, where $\varphi_k^2(x) \coloneqq \frac{x_{1,k}}{x_{1,k}-x}r_{k,<}^2(x)$, $x \in [0,\norm{AA^\top}_{\mathrm{spec}}]$. For $k=0$, $\norm{\mathsf{r_k}}_Y^2 \le \norm{\mathsf{r_{k,<}^{1/2}}}_Y^2$ continues to hold.
\end{lemma}

\begin{proof}
For $k=0$, the claim is clear. Let $k=1,\dots,d$. We have
\begin{equation*}
\frac{x}{x-x_{1,k}} r_k(x) = -\frac{x}{x_{1,k}} \prod_{i=2}^k \Big( 1-\frac{x}{x_{i,k}} \Big) \in \Pol_{k,0}.
\end{equation*}
\cref{LemOrthogonalityY} yields $r_k \perp_{Y} \frac{x}{x-x_{1,k}} r_k$ such that
\begin{equation*}
\bscapro{\mathsf{r_k}}{\mathsf{\frac{x}{x-x_{1,k}}r_k}}_{Y,>}=\bscapro{\mathsf{r_k}}{\mathsf{\frac{x}{x_{1,k}-x}r_k}}_{Y,<}.
\end{equation*}
Using
\begin{equation*}
\bscapro{\mathsf{r_k}}{\mathsf{\frac{x}{x-x_{1,k}}r_k}}_{Y,>} > \norm{\mathsf{r_k}}_{Y,>}^2,
\end{equation*}
this implies
\begin{align*}
\norm{\mathsf{r_k}}_Y^2 &= \norm{\mathsf{r_k}}_{Y,<}^2 + \norm{\mathsf{r_k}}_{Y,>}^2 \\
&< \scapro{\mathsf{r_k}}{\mathsf{r_k}}_{Y,<} + \bscapro{\mathsf{r_k}}{\mathsf{\frac{x}{x_{1,k}-x}r_k}}_{Y,<} \\
&= \norm{\mathsf{\varphi_k}}_{Y}^2.
\end{align*}
Note that $r_{k,<}(x) \le 1-\frac{x}{x_{1,k}} = \frac{x_{1,k}-x}{x_{1,k}}$. Thus,
\begin{equation*}
\norm{\mathsf{\varphi_k}}_{Y}^2 = \bnorm{ \mathsf{\Big( \frac{x_{1,k}}{x_{1,k}-x} \Big)^{1/2} r_{k,<}^{1/2} r_{k,<}^{1/2}} }_Y^2 \le \norm{\mathsf{r_{k,<}^{1/2}}}_Y^2. \qedhere
\end{equation*}
\end{proof}

\subsection{Proofs for Section~\ref{SecInterpolConjGradientMethod}} \label{SecProofsResPol}

\begin{proof}[Proof of \cref{LemPropSquaredResNorm}] \leavevmode
\begin{enumerate}
\item Note that $r_{k+1}-r_k \in \Pol_{k+1,0}$. \cref{LemOrthogonalityY} implies $\scapro{\mathsf{r_k}}{\mathsf{r_{k+1}}}_Y=R_{k+1}^2$. We conclude the claim by
\begin{equation*}
R_t^2 = (1-\alpha)^2 R_k^2 +2(1-\alpha)\alpha \scapro{\mathsf{r_k}}{\mathsf{r_{k+1}}}_Y + \alpha^2 R_{k+1}^2.
\end{equation*}
\item The continuity follows directly from \subcref{LemPropSquaredResNorm:NonlinInterpol}. Note that $R_k^2$ is strictly decreasing in $k \in \{ 0,\dots,d \}$, since $r_k$ is the unique solution of $\min_{p_k \in \Pol_{k,1}} \norm{\mathsf{p_k}}_Y^2$ and has degree $k$ by \cref{LemRkWellDef} and \cref{CorResidualPolynomialsOrth}. Together with \subcref{LemPropSquaredResNorm:NonlinInterpol}, this implies the monotonicity of $R_t^2$ in $t$.
\item For $t,s=0,\dots,d$, \cref{LemOrthogonalityY} yields
\begin{equation*}
R_{t \vee s}^2 - \scapro{\mathsf{r_t}}{\mathsf{r_s}}_Y = \scapro{\mathsf{r_{t \vee s}-r_{t \wedge s}}}{\mathsf{r_{t \vee s}}}_Y = 0,
\end{equation*}
since $r_{t \vee s}-r_{t \wedge s} \in \Pol_{t \vee s,0}$. We conclude
\begin{align}
\norm{A(\widehat f_t-\widehat f_s)}^2 &= \norm{\mathsf{r_{t} - r_{s}}}_Y^2 \notag \\
&= R_{t}^2 + R_{s}^2 - 2 \scapro{\mathsf{r_{t}}}{\mathsf{r_{s}}}_Y \notag \\
&= R_{t}^2 + R_{s}^2 - 2 R_{t \vee s}^2 \notag \\
&= \abs{R_t^2 - R_{s}^2}. \label{EqProofCorResNormMaxts}
\end{align}
This gives the claim for integer indices.

For $t=k+\alpha$ with $k=0,\dots,d-1$, $\alpha\in[0,1]$ and $0\le s\le k$, we have $r_s\in\Pol_{k,1}$,  $r_t-r_s\in \Pol_{k+1,0}$ and $r_k-r_s \in \Pol_{k,0}$. \cref{LemOrthogonalityY} implies $\scapro{\mathsf{r_{k+1}}}{\mathsf{r_t-r_s}}_Y=0$, $\scapro{\mathsf{r_k}}{\mathsf{r_k-r_s}}_Y=0$ and $\scapro{\mathsf{r_k}}{\mathsf{r_{k+1}}}_Y=R_{k+1}^2$. Thus, we obtain
\begin{align*}
\scapro{\mathsf{r_t}}{\mathsf{r_t-r_s}}_Y &= \scapro{(1-\alpha)\mathsf{r_k}+\alpha \mathsf{r_{k+1}}}{\mathsf{r_t-r_s}}_Y \\
&= (1-\alpha)\scapro{\mathsf{r_k}}{\mathsf{r_t-r_s}}_Y \\
&=(1-\alpha) \scapro{\mathsf{r_k}}{\mathsf{r_k-r_s}}_Y + (1-\alpha)\alpha\scapro{\mathsf{r_k}}{\mathsf{r_{k+1}-r_k}}_Y \\
&= (1-\alpha) \alpha (R_{k+1}^2 - R_k^2) \\
&\le 0.
\end{align*}
Similarly, for $t=k+\alpha$, $s=k+\beta$ with $0\le\beta\le\alpha\le 1$, we deduce
\begin{align*}
\scapro{\mathsf{r_t}}{\mathsf{r_t-r_s}}_Y &= (1-\alpha) \scapro{\mathsf{r_k}}{\mathsf{r_t-r_s}}_Y \\
&= (1-\alpha)\scapro{\mathsf{r_k}}{(1-\alpha)\mathsf{r_k}+\alpha \mathsf{r_{k+1}} - (1-\beta)\mathsf{r_k} - \beta \mathsf{r_{k+1}}}_Y \\
&= (1-\alpha)(\beta-\alpha) R_k^2 + (1-\alpha)(\alpha-\beta) \scapro{\mathsf{r_k}}{\mathsf{r_{k+1}}}_Y \\
&= (1-\alpha)(\alpha-\beta) \big(R_{k+1}^2-R_{k}^2\big) \\
&\le 0.
\end{align*}
We infer $\scapro{\mathsf{r_t}}{\mathsf{r_s}}_Y\ge R_{t\vee s}^2$ for any $t,s \in [0,d]$. The second inequality follows similarly to \eqref{EqProofCorResNormMaxts} by applying $\scapro{\mathsf{r_t}}{\mathsf{r_s}}_Y\ge R_{t\vee s}^2$. \qedhere
\end{enumerate}
\end{proof}

\begin{proof}[Proof of \cref{LemPropInterpolResidualPolynomials}] \leavevmode
\begin{enumerate}
\item See \cref{LemPropResidualPolynomials:simpleZeros}.
\item Let $k=0$. \cref{LemPropResidualPolynomials:Formula} yields
\begin{equation*}
r_t(x)= 1-\alpha + \alpha \Big( 1-\frac{x}{x_{1,1}} \Big) = 1-\alpha \frac{x}{x_{1,1}}, \quad x \in [0,\norm{AA^\top}_{\mathrm{spec}}],
\end{equation*}
such that
\begin{equation}
x_{1,t}=\alpha^{-1}x_{1,1}>x_{1,1}. \label{EqZeroExplicitForm}
\end{equation}
This gives the claim for $k=0$. Now, let $k=1,\dots,d-1$. By \cref{LemPropResidualPolynomials:simpleZeros} and \subcref{LemPropResidualPolynomials:interlacedZeros}, the zeros satisfy
\begin{equation*}
0 < x_{1,k+1} < x_{1,k} < x_{2,k+1} < \dots < x_{k,k+1} < x_{k,k} < x_{k+1,k+1} \le \norm{AA^\top}_{\mathrm{spec}}.
\end{equation*}
For $x \in (0,\norm{AA^\top}_{\mathrm{spec}})$, \eqref{EqWronskianUneqZero} implies
\begin{align*}
W(x; r_t, r_{k+1}) &= r_t(x)r_{k+1}'(x) - r_t'(x)r_{k+1}(x) \\
&= (1-\alpha)\big(r_k(x) r_{k+1}'(x) - r_k'(x)r_{k+1}(x)\big) \\
&= (1-\alpha) W(x;r_k,r_{k+1}) \\
&\neq 0,
\end{align*}
and analogously $W(x; r_t, r_{k}) \neq 0$. Thus, by \cref{LemPropWronskian}, any two consecutive zeros of $r_k$ or $r_{k+1}$ in $(0, \norm{AA^\top}_{\mathrm{spec}})$ are separated by a zero of $r_t$ and vice versa, and the zeros of $r_t$ in $(0,\norm{AA^\top}_{\mathrm{spec}})$ are simple. Now, $r_t(x) > 0$ on $[0,x_{1,k+1}]$, since $r_k(x) > 0$ and $r_{k+1}(x) \ge 0$ on $[0,x_{1,k+1}]$, and $r_t(x_{1,k})<0$, since $r_{k+1}(x) < 0$ on $(x_{1,k+1},x_{1,k}]$. Hence, $x_{1,k+1} < x_{1,t} < x_{1,k}$. Thus, by the interlacing of $r_t$ and $r_k$ in $(0, \norm{AA^\top}_{\mathrm{spec}})$, $r_t$ has at least $k$ simple zeros in $(0, \norm{AA^\top}_{\mathrm{spec}})$. The $k+1$-th zero is also real and simple. Otherwise, by the complex conjugate root theorem, its complex conjugate would also be a zero, but this is not possible, since $r_t$ has degree $k+1$. We obtain by the interlacing
\begin{equation*}
x_{1,t} < x_{1,k} < x_{2,t} < \dots < x_{k,t} < x_{k,k} < x_{k+1,t}
\end{equation*}
and
\begin{equation*}
x_{1,k+1} < x_{1,t} < x_{2,k+1} < \dots < x_{k,k+1} < x_{k,t}.
\end{equation*}
By considering the two cases whether $k$ is even or odd separately, one can show that $r_t(x) \neq 0$ on $(x_{k,k}, x_{k+1,k+1}]$ such that $x_{k+1,t} > x_{k+1,k+1}$. Combining this with the two chains of inequalities yields the claim.
\item Follows analogously to \cref{LemPropResidualPolynomials:Formula}.
\item Follows analogously to \cref{LemPropResidualPolynomials:Properties}.
\item Let $k=0$. By \eqref{EqZeroExplicitForm} and $\alpha_1 < \alpha_2$, we have
\begin{equation*}
x_{1,t_2} = \alpha_2^{-1}x_{1,1} < \alpha_1^{-1}x_{1,1} = x_{1,t_1}.
\end{equation*}
This gives the claim for $k=0$. Now, let $k=1,\dots,d-1$. The definitions of $r_{t_1}$ and $r_{t_2}$ and \eqref{EqWronskianUneqZero} yield
\begin{equation*}
W(x; r_{t_1}, r_{t_2}) = (\alpha_2-\alpha_1) W(x; r_k, r_{k+1}) \neq 0, \quad x \in (0,\norm{AA^\top}_{\mathrm{spec}}).
\end{equation*}
Hence, by \cref{LemPropWronskian}, any two consecutive zeros of $r_{t_1}$ in $(0, \norm{AA^\top}_{\mathrm{spec}})$ are separated by a zero of $r_{t_2}$ and vice versa. Since by \subcref{LemPropInterpolResidualPolynomials:simpleZeros} both polynomials have $k+1$ simple zeros, $x_{k,t_j} < x_{k,k} < \norm{AA^\top}_{\mathrm{spec}}$ and $x_{k,k}<x_{k+1,k+1}<x_{k+1,t_j}$ by \subcref{LemPropInterpolResidualPolynomials:simpleZeros} and \cref{LemPropResidualPolynomials:interlacedZeros}, $j=1,2$, it remains to show that $x_{1,t_2}<x_{1,t_1}$ and that $x_{k+1,t_2} < x_{k+1,t_1}$. For all $x \in (0,x_{1,k}]$, we have
\begin{equation*}
r_{t_1}(x) = r_{t_2}(x) + (\alpha_2-\alpha_1)(r_k(x)-r_{k+1}(x)) > r_{t_2}(x),
\end{equation*}
since $r_{k+1}(x) < r_{k}(x)$ for all $x \in (0,x_{1,k+1}]$ by \cref{LemPropResidualPolynomials:lowerBounded}, $r_k(x) \ge 0$ for $x \in [0,x_{1,k}]$ and $r_{k+1}(x) < 0$ for $x \in (x_{1,k+1},x_{1,k}]$. This implies $x_{1,t_2} < x_{1,t_1}$. Let $x > x_{k+1,k+1}$. For $k$ even we obtain $r_k(x)>0$ and $r_{k+1}(x)<0$ such that $r_{t_1}(x) > r_{t_2}(x)$. Since $r_{t_j}(x_{k+1,k+1}) > 0$, $j=1,2$, this yields $x_{k+1,t_2} < x_{k+1,t_1}$. In the same way, for $k$ odd we deduce $r_{t_1}(x) < r_{t_2}(x)$ and $r_{t_j}(x_{k+1,k+1}) < 0$, $j=1,2$, implying $x_{k+1,t_2} < x_{k+1,t_1}$.
\item The right-continuity in $t_0=0$ is immediate, since we have for $t=\alpha \in (0,1]$ by \eqref{EqZeroExplicitForm}
\begin{equation*}
r_{t,<}(x) = \Big( 1-\alpha \frac{x}{x_{1,1}} \Big) {\bf 1}(x < \alpha^{-1} x_{1,1}), \quad x \in [0,\norm{AA^\top}_{\mathrm{spec}}].
\end{equation*}
Now, let $t_0=k+\alpha_0$ with $k=0,\dots,d-1$ and $\alpha_0 \in (0,1]$. For $t=k+\alpha$ with $0 < \alpha < \alpha_0$, note that $x_{1,t_0} < x_{1,t} < x_{2,t_0}$ by \subcref{LemPropInterpolResidualPolynomials:simpleZeros} and \subcref{LemPropInterpolResidualPolynomials:interlacedZeros} and \cref{LemPropResidualPolynomials:interlacedZeros}, setting $x_{2,t_0} \coloneqq \infty$ for $t_0 \le 1$. This implies $r_{t_0,<}(x)=r_{t,<}(x)=0$ for $x > x_{1,t}$ and $r_{t_0}(x) < 0$ for $x \in (x_{1,t_0},x_{1,t}]$. Since in addition $r_t(x) \ge 0$ for $x \in [0,x_{1,t}]$ by \subcref{LemPropInterpolResidualPolynomials:Properties}, we obtain
\begin{align*}
\sup_{x \in [0,\norm{AA^{\top}}_{\mathrm{spec}}]} \abs{r_{t_0,<}(x)-r_{t,<}(x)} &\le \sup_{x \in [0,x_{1,t}]} \abs{r_{t_0}(x)-r_{t}(x)} \\
&= (\alpha_0 - \alpha) \sup_{x \in [0,x_{1,t}]} \abs{r_{k+1}(x)-r_{k}(x)} \\
&\le (t_0-t) \sup_{x \in [0,\norm{AA^{\top}}_{\mathrm{spec}}]} \abs{r_{k+1}(x)-r_{k}(x)},
\end{align*}
where the latter supremum is finite. Letting $t \uparrow t_0$ shows the left-continuity for all $t_0 \in (0,d]$, and right-continuity for $t_0 \in (0,d)$ can be proven in the same way. This gives the first claim.

The formula given in \subcref{LemPropInterpolResidualPolynomials:Formula} together with \subcref{LemPropInterpolResidualPolynomials:simpleZeros} and \subcref{LemPropInterpolResidualPolynomials:interlacedZeros} implies that $r_{k+1,<} \le r_{t_2,<} \le r_{t_1,<} \le r_{k,<}$, where $t_j$, $j=1,2$, are defined as in \subcref{LemPropInterpolResidualPolynomials:interlacedZeros} for $k=0,\dots,d-1$. This yields the second claim.
\item See \cref{LemPropResidualPolynomials:ZerosEigenvalues}. \qedhere
\end{enumerate}
\end{proof}

\begin{proof}[Proof of \cref{PropRtbound}]
For integer $t=0,\dots,d$, this was proved in \cref{LemRkbound}. Now, consider the case $t=k+\alpha$ with $k = 0,\dots,d-1$, $\alpha\in(0,1)$.
By \cref{LemOrthogonalityY}, we have $r_k \perp_Y \Pol_{k,0}$ and $r_{k+1} \perp_Y \Pol_{k+1,0} \supset \Pol_{k,0}$ such that $r_t\perp_Y \Pol_{k,0}$.
Using in addition that $1-r_k \in \Pol_{k,0}$, $1-r_{k+1} \in \Pol_{k+1,0}$ and $r_{k+1}-r_k \in \Pol_{k+1,0}$, we obtain
\begin{equation*}
\scapro{\mathsf{r_t}}{\mathsf{1}}_Y = (1-\alpha)\scapro{\mathsf{r_k}}{\mathsf{1}}_Y + \alpha \scapro{\mathsf{r_{k+1}}}{\mathsf{1}}_Y = (1-\alpha) R_k^2 + \alpha R_{k+1}^2
\end{equation*}
and
\begin{align*}
\scapro{\mathsf{r_t}}{\mathsf{r_{k+1}-1}}_Y &= (1-\alpha) \scapro{\mathsf{r_k}}{\mathsf{r_{k+1}-1}}_Y \\
&= (1-\alpha) \scapro{\mathsf{r_k-r_{k+1}}}{\mathsf{r_{k+1}-1}}_Y \\
&=-(1-\alpha) \scapro{\mathsf{r_k-r_{k+1}}}{\mathsf{1}}_Y \\
&= -(1-\alpha)(R_k^2 - R_{k+1}^2).
\end{align*}
We deduce the orthogonality relation
\begin{equation}
\scapro{\mathsf{r_t}}{\mathsf{r_{k+1}-1+\zeta}}_Y = -(1-\alpha)(R_k^2 - R_{k+1}^2) + \zeta \big((1-\alpha) R_k^2 + \alpha R_{k+1}^2\big)=0 \label{EqOrthogonalityRho}
\end{equation}
with
\begin{equation*}
\zeta\coloneqq \tfrac{(1-\alpha)(R_k^2-R_{k+1}^2)}{(1-\alpha)R_k^2+\alpha R_{k+1}^2} \le 1.
\end{equation*}
Put
\begin{equation*}
\tilde x_{1,t} \coloneqq \zeta x_{1,{k+1}}\prod_{i=2}^{k+1}\tfrac{x_{i,k+1}}{x_{i,t}},\quad \tilde r_t \coloneqq \tfrac{\tilde x_{1,t}-x}{x_{1,t}-x} r_t,
\end{equation*}
and note the order $0\le \tilde x_{1,t} < x_{1,k+1} < x_{1,t}$ due to $x_{i,t} > x_{i,k+1}$ by \cref{LemPropInterpolResidualPolynomials:simpleZeros}.
With some $q_k,\tilde q_k, \dot q_k \in\Pol_{k,0}$, $\tilde{p}_{k-1} \in \Pol_{k-1,1}$, we have
\begin{align*}
\tilde{r}_t(x) &= \frac{\tilde x_{1,t}-x}{x_{1,t}-x} \prod_{i=1}^{k+1} \Big( 1-\frac{x}{x_{i,t}} \Big) \\
&= \frac{\tilde{x}_{1,t}}{x_{1,t}} \prod_{i=2}^{k+1} \Big( 1-\frac{x}{x_{i,t}} \Big) - \frac{x}{x_{1,t}} \prod_{i=2}^{k+1} \Big( 1-\frac{x}{x_{i,t}} \Big) \\
&= \frac{\tilde{x}_{1,t}}{x_{1,t}}  + \dot{q}_k(x) -\frac{x}{x_{1,t}} \bigg(x^{k} \prod_{i=2}^{k+1} \Big( -\frac{1}{x_{i,t}} \Big) + \tilde{p}_{k-1}(x) \bigg) \\
&= \frac{\tilde{x}_{1,t}}{x_{1,t}} + \tilde{q}_k(x) - \frac{x^{k+1}}{x_{1,t}} \prod_{i=2}^{k+1} \Big( -\frac{1}{x_{i,t}} \Big) \\
&= \frac{\tilde{x}_{1,t}}{x_{1,t}} + \tilde{q}_k(x) + \frac{x^{k+1}}{x_{1,t}} \tilde{x}_{1,t} \zeta^{-1} \prod_{i=1}^{k+1}\Big(-\frac{1}{x_{i,k+1}}\Big) \\
&= \frac{\tilde{x}_{1,t}}{x_{1,t}} \zeta^{-1} \bigg( \zeta + \prod_{i=1}^{k+1}\Big(-\frac{x}{x_{i,k+1}}\Big) \bigg) + \tilde{q}_k(x) \\
&= \frac{\tilde{x}_{1,t}}{x_{1,t}} \zeta^{-1} \big( \zeta + r_{k+1}(x)-1-q_k(x) \big) + \tilde{q}_k(x).
\end{align*}

The orthogonality relation \eqref{EqOrthogonalityRho} together with $r_t\perp_Y\Pol_{k,0}$ therefore gives $\tilde r_t \perp_Y r_t$.
Consequently, $\scapro{\mathsf{\tilde r_t}}{\mathsf{r_t}}_{Y,>}=-\scapro{\mathsf{\tilde r_t}}{\mathsf{r_t}}_{Y,<}$ holds and thus
\begin{align*}
\norm{\mathsf{r_t}}_Y^2 &= \norm{\mathsf{r_{t,<}}}_Y^2 + \norm{\mathsf{r_{t,>}}}_Y^2
\le \norm{\mathsf{r_{t,<}}}_Y^2 + \scapro{\mathsf{r_t}}{\mathsf{(\tfrac{x-\tilde x_{1,t}}{x-x_{1,t}}) r_t}}_{Y,>}
= \norm{\mathsf{r_{t,<}}}_Y^2 + \scapro{\mathsf{r_t}}{\mathsf{\tilde{r}_t}}_{Y,>} \\
&= \norm{\mathsf{r_{t,<}}}_Y^2 - \scapro{\mathsf{r_t}}{\mathsf{\tilde{r}_t}}_{Y,<}
= \scapro{\mathsf{r_t(r_t-\tilde{r}_t)}}{\mathsf{1}}_{Y,<}
= \norm{\mathsf{(r_{t,<}^2-r_{t,<}\tilde r_{t,<})^{1/2}}}_Y^2.
\end{align*}
From  $r_{t,<}(x)\le (1-x_{1,t}^{-1}x)_+$ by \cref{LemPropInterpolResidualPolynomials:Formula}, we infer
\begin{align*}
(r_{t,<}^2-r_{t,<}\tilde r_{t,<})(x) &= r_{t,<}^2(x) \Big( 1-\frac{\tilde{x}_{1,t}-x}{x_{1,t}-x} \Big)
= r_{t,<}^2(x)\frac{x_{1,t}-\tilde{x}_{1,t}}{x_{1,t}-x} \\
&\le r_{t,<}(x) \frac{x_{1,t}-\tilde{x}_{1,t}}{x_{1,t}}
\le r_{t,<}(x).
\end{align*}
This gives the result.
\end{proof}

\begin{proof}[Proof of \cref{LemrtBound}]
For $t=0$, the claim is clear. Let $t \in (0,d]$. Recall that $r_t'(0)<0$ and that $r_t$ is convex and nonnegative on $[0,x_{1,t}]$ by \cref{LemPropInterpolResidualPolynomials:Properties}. This implies
\begin{equation*}
(1-\abs{r_t'(0)}x)_+\le r_{t}(x)=\prod_{i=1}^{\ceil{t}} \Big( 1-\frac{x}{x_{i,t}} \Big) \le 1-\frac{x}{x_{1,t}}, \quad x \in [0,x_{1,t}].
\end{equation*}
By the log-concavity in \cref{LemPropInterpolResidualPolynomials:Properties}, note for $t=k+\alpha$, $k=0,\dots,d-1$, $\alpha \in (0,1]$,
\begin{equation*}
\log(r_t(x)) \le \log(r_t(0)) + x\Big(\frac{d}{dy} \log(r_t(y)) \Big) \Big|_{y=0} = r_t'(0)x, \quad x \in [0,x_{1,t}).
\end{equation*}
We conclude
\begin{equation*}
(1-\abs{r_t'(0)}x)_+\le r_{t}(x) \le \exp(r_t'(0)x) = \exp(-\abs{r_t'(0)}x), \quad x \in [0,x_{1,t}]. \qedhere
\end{equation*}
\end{proof}

\begin{proof}[Proof of \cref{CorUpperBoundApproxErrorSourceCond}]
Note that $\abs{r_t'(0)}=\sum_{i=1}^{\lceil t \rceil} x_{i,t}^{-1}>0$ by \cref{LemPropInterpolResidualPolynomials:Formula}. We have
\begin{equation*}
\sup_{x\ge 0}e^{-\abs{r_{t}'(0)}x} x^{2\mu+1}= \Big(\frac{2\mu+1}{e\abs{r_t'(0)}}\Big)^{2\mu+1}.
\end{equation*}
Under the source condition $\sourcecondg$, we deduce by \cref{LemrtBound}
\begin{align*}
\norm{\mathsf{r_{t,<}^{1/2}}g}^2 &\le \norm{\mathsf{\exp(-\abs{r_t'(0)}x/2)}g}^2 \\
&= \sum_{i=1}^D \big(\exp(-\abs{r_t'(0)}\lambda_i^2)\lambda_i^{4\mu+2}\big)\lambda_i^{-(4\mu +2)}g_i^2 \\
&\le \Big(\frac{2\mu+1}{e\abs{r_t'(0)}}\Big)^{2\mu+1}\norm{g}_{\mu+1/2}^2 \\
&\le R^2 (\mu+1/2)^{2\mu+1} \abs{r_t'(0)}^{-2\mu-1}. \qedhere
\end{align*}
\end{proof}

\begin{proof}[Derivation of \eqref{Eqrkpen}]
Note that $p_{k+1}-r_k-\alpha(r_{k+1}-r_k)$ lies in $\Pol_{k,0}$ for $\alpha=\linebreak[4] a_{k+1}(p_{k+1})/a_{k+1}(r_{k+1})$ and every $p_{k+1} \in \Pol_{k+1,1}$, where $a_{k+1}(r_{k+1}) \neq 0$ by \cref{CorResidualPolynomialsOrth}, and $r_k,r_{k+1}\perp_Y \Pol_{k,0}$ by \cref{LemOrthogonalityY}. By rewriting $p_{k+1}=r_k+\alpha(r_{k+1}-r_k)+q_k$ with $q_k \in \Pol_{k,0}$ and applying the Pythagorean theorem, the minimisation \eqref{EqDefrkpen} yields a polynomial $r_{k+1,\lambda}= r_k+\alpha_\lambda(r_{k+1}-r_k)$ with
\begin{align*}
\alpha_\lambda \coloneqq{}& \argmin_{\alpha\in\R}\Big(\norm{\mathsf{r_k}+\alpha(\mathsf{r_{k+1}-r_k})}_Y^2+\lambda a_{k+1}(r_{k+1})^2\alpha^2\Big) \\
={}&\frac{\norm{\mathsf{r_k}}_Y^2-\norm{\mathsf{r_{k+1}}}_Y^2} {\norm{\mathsf{r_k}}_Y^2-\norm{\mathsf{r_{k+1}}}_Y^2+\lambda a_{k+1}(r_{k+1})^2},
\end{align*}
where we used that $\scapro{\mathsf{r_{k+1}-r_k}}{\mathsf{r_{k+1}}}_{Y}=0$ by \cref{LemOrthogonalityY}. The numerator is just $R_k^2-R_{k+1}^2$, which is strictly positive for $k+1\le d$ so that $\alpha_\lambda\in(0,1]$.
\end{proof}

\subsection{Proofs for Section~\ref{SecAnalysisPredError}}\label{SecProofsAnalysisPredError}

\begin{proof}[Proof of \cref{Lemrkg}]
Noting that $r_{t,<} \le 1$, we have
\begin{equation*}
\norm{\mathsf{r_{t,<}}Y}^2 - \norm{\mathsf{r_{t,<}^{1/2}}Y}^2 = -\scapro{\mathsf{r_{t,<}(1-r_{t,<})}Y}{Y} = - \norm{\mathsf{(r_{t,<}(1-r_{t,<}))^{1/2}}Y}^2,
\end{equation*}
and analogously
\begin{equation*}
\norm{\mathsf{r_{t,<}} g}^2 - \norm{\mathsf{r_{t,<}^{1/2}}g}^2 = - \norm{\mathsf{(r_{t,<}(1-r_{t,<}))^{1/2}}g}^2.
\end{equation*}
Hence, the approximation error can be expressed as
\begin{align*}
A_{t,\lambda} = &\norm{\mathsf{r_{t,<}}g}^2+\norm{\mathsf{r_{t,>}}Y}^2 \\
&-\big(\norm{\mathsf{(r_{t,<}(1-r_{t,<}))^{1/2}}Y}^2
-\norm{\mathsf{(r_{t,<}(1-r_{t,<}))^{1/2}}g}^2\big)
\end{align*}
with
\begin{align*}
&\norm{\mathsf{(r_{t,<}(1-r_{t,<}))^{1/2}}Y}^2 \\
&= \norm{\mathsf{(r_{t,<}(1-r_{t,<}))^{1/2}}g}^2 + \norm{\mathsf{(r_{t,<}(1-r_{t,<}))^{1/2}}\xi}^2 + 2 \scapro{\mathsf{r_{t,<}}g}{(\mathsf{1-r_{t,<}})\xi}.
\end{align*}
This implies
\begin{equation*}
\norm{\mathsf{r_{t,<}}g}^2 \le A_{t,\lambda}+\norm{\mathsf{(r_{t,<}(1-r_{t,<}))^{1/2}}\xi}^2+2\scapro{\mathsf{r_{t,<}}g}{(\mathsf{1-r_{t,<}})\xi}.
\end{equation*}
By the Cauchy--Schwarz inequality and the definition of $S_{t,\lambda}$, we deduce
\begin{align*}
\norm{\mathsf{r_{t,<}}g}^2 &\le A_{t,\lambda}+\norm{\mathsf{(r_{t,<}(1-r_{t,<}))^{1/2}}\xi}^2+2 \norm{\mathsf{r_{t,<}}g} \norm{(\mathsf{1-r_{t,<}})\xi} \\
&\le A_{t,\lambda} + S_{t,\lambda} + 2 S_{t,\lambda}^{1/2} \norm{\mathsf{r_{t,<}}g}.
\end{align*}
For $x,A,B \ge 0$, $x^2\le Ax+B$ implies $x^2\le A^2+2B$, which is immediate for $x\le A$ and follows for $x > A$ from $(x+A)(x-A)\le 2x(x-A)\le 2B$.
Noting that $B\coloneqq A_{t,\lambda} + S_{t,\lambda} \ge \tfrac{1}{2} \norm{\widehat{g}_t-g}^2 \ge 0$ by \cref{PropWeakLoss}, we obtain $\norm{\mathsf{r_{t,<}}g}^2  \le 6S_{t,\lambda}+2A_{t,\lambda}$.
\end{proof}

\begin{proof}[Proof of \cref{PropWeakErrorLowerBound}]
We prove the claim analogously to \citet[Proposition~4.23]{Joh2017}. Without loss of generality, assume that $v_1,\dots,v_D$ are the first $D$ canonical unit vectors. Let $\Theta^m(\eta)$, $m=1,\dots,D$, for $\eta \ge 0$ denote the set of $f \in \R^P$ that satisfy $\abs{f_i} \le \eta$ for $i=1,\dots,m$ and $f_i=0$ for $i=m+1,\dots,P$ in case of $P > m$. Define the loss
\begin{equation*}
L_A(\widehat{f},f) \coloneqq \norm{A(f-\widehat{f})}^2 = \sum_{i=1}^D \lambda_i^2 (f_i-\widehat{f}_i)^2 = \sum_{i=1}^D L_{A,i}(\widehat{f}_i, f_i),
\end{equation*}
where $L_{A,i}(\widehat{f}_i, f_i) \coloneqq \lambda_i^2 (f_i-\widehat{f}_i)^2$, $i=1,\dots,D$. The function $\widehat{f}_i \mapsto L_{A,i}(\widehat{f}_i, f_i)$ is convex and lower semicontinuous for any $f_i$. Thus, we can apply \citet[Proposition~4.16]{Joh2017} and obtain
\begin{align*}
\inf_{\widehat{f}} \sup_{f \in \Theta^m(\eta)} \E_f\big[\norm{A(f-\widehat{f})}^2\big] &= \sum_{i=1}^m \Big(\inf_{\widehat{f}_i} \sup_{f_i \in [-\eta,\eta]} \E_f\big[\lambda_i^2 (f_i-\widehat{f}_i)^2\big]\Big) \\
&= \sum_{i=1}^m \lambda_i^2 \Big(\inf_{\widehat{f}_i} \sup_{f_i \in [-\eta,\eta]} \E_f\big[(f_i-\widehat{f}_i)^2\big]\Big) \\
&= \gamma \sum_{i=1}^m \lambda_i^2 (\eta^2 \wedge \delta^2 \lambda_i^{-2})
\end{align*}
with $\gamma \in [2/5, 1]$, where we used \citet[Eq.~(4.40)]{Joh2017} for the last equality. Setting $\eta \coloneqq \delta m^p$, under \eqrefAssPSDlower{}, we deduce
\begin{equation*}
\inf_{\widehat{f}} \sup_{f \in \Theta^m(\delta m^p)} \E_f\big[\norm{A(f-\widehat{f})}^2\big] = \gamma \sum_{i=1}^m \lambda_i^2 \eta^2 \wedge \delta^2 \ge (c_A \wedge 1)^2 \gamma m \delta^2.
\end{equation*}
We have
\begin{equation*}
\sum_{i=1}^m \lambda_i^{-4\mu} (\delta m^p)^2 \le \delta^2 m^{2p} c_A^{-4\mu} \sum_{i=1}^m i^{4\mu p} \le c_A^{-4\mu} \delta^2 m^{4 \mu p + 2p + 1}.
\end{equation*}
Hence, for $\Theta^m(\delta m^p)$ to be contained in the class of signals $f$ satisfying $\sourcecondf$, it suffices that $c_A^{-4\mu} \delta^2 m^{4 \mu p + 2p + 1} \le R^2$. Let $m_1$ be the largest integer such that $c_A^{-4\mu} m_1^{4 \mu p + 2p + 1} \le R^2/\delta^2$ and $m_0 \in \R$ the solution to $c_A^{-4\mu} m_0^{4 \mu p + 2p + 1} = R^2/\delta^2$. If $m_0 \ge 4$, then $m_1/m_0 \ge 1/2$, which can be shown by contradiction. For $\delta/R \le \tilde{c}_{\mu,p,c_A}$, we conclude
\begin{align*}
&\inf_{\widehat{f}} \sup_{f: \, \sourcecondf} \E_f\big[\norm{A(f-\widehat{f})}^2\big] \\
&\ge \inf_{\widehat{f}} \sup_{f \in \Theta^m(\delta m_1^p)} \E_f\big[\norm{A(f-\widehat{f})}^2\big] \\
&\ge (c_A \wedge 1)^2 \gamma m_1 \delta^2 \ge (c_A \wedge 1)^2\gamma \frac{m_0}{2} \delta^2 \\
&= (c_A \wedge 1)^2 c_A^{4\mu/(4 \mu p +2p + 1)} \frac{\gamma}{2} R^{2/(4 \mu p+2p+1)} \delta^{(8 \mu p+4p)/(4 \mu p+2p+1)}.
\end{align*}
This gives the claim.
\end{proof}

\subsection{Proofs for Section~\ref{SecTransferReconstructionError}} \label{SecProofsTransferReconstructionError}

\begin{proof}[Proof of \cref{PropBoundsStrongErrorEarlyStopped}] \leavevmode
\begin{enumerate}
\item Consider the event of \cref{LemRhoBound}, which holds with probability at least $1-e^{-z}$. If $\abs{r_{\tau}'(0)} < \lambda_D^{-2}$, then by the first implication of  \cref{LemRhoBound}
\begin{equation*}
\abs{r_{\tau}'(0)} \lesssim_{p,c_A} (\delta^{-2}S_{\tau,\lambda})^{2p} + (\log d)^{2p} + z^{2p}
\end{equation*}
and the assertion follows from \cref{CorRoughBoundStrongErrorSobolevCond}.
If $\abs{r_{\tau}'(0)} \ge \lambda_D^{-2}$, then we use the rough bound
\begin{equation}\label{EqRoughRecError}
\norm{\widehat f_\tau-f^\dagger}^2\le \lambda_D^{-2}\norm{A(\widehat f_\tau-f^\dagger)}^2\le 4\lambda_D^{-2}M_{\tau,\lambda},
\end{equation}
which is due to the prediction error bound from \cref{PropWeakLoss}.
Moreover, we then have by \ref{AssPSD} and the second implication of \cref{LemRhoBound}
\begin{equation*}
\lambda_D^{-2}\lesssim_{c_A} D^{2p} \lesssim_{p,c_A} (\delta^{-2}S_{\tau,\lambda})^{2p} + (\log d)^{2p} + z^{2p},
\end{equation*}
which inserted into \eqref{EqRoughRecError} yields an even sharper bound than claimed.

\item Fix $z>0$ and note $M_{\tau,\lambda}=A_{\tau,\lambda} \wedge S_{\tau,\lambda}+\abs{A_{\tau,\lambda} - S_{\tau,\lambda}}$. In the case $\tau > 0$, we have by \cref{LemPropErrorTerms:BoundStochError}, the source condition on $f$ and \eqref{EqUpperBoundApproxErrorSource}
\begin{align}
A_{\tau,\lambda} \wedge S_{\tau,\lambda}
&\le \big( R^2(\mu+1/2)^{2\mu + 1} \abs{r_{\tau}'(0)}^{-2\mu-1} \wedge \norm{(\mathsf{(\abs{r_{\tau}'(0)}x)^{1/2} \wedge 1})\xi}^2 \big)  \notag \\
&\le \inf_{\rho > 0} \big( R^2(\mu+1/2)^{2\mu + 1} \rho^{-2\mu-1} + \norm{(\mathsf{(\rho x)^{1/2} \wedge 1})\xi}^2 \big) \notag \\
&\le \inf_{\rho>0}\Big(R^2 (\mu + 1/2)^{2\mu + 1} \rho^{-2\mu-1} + \delta^2 \sum_{i=1}^D (\rho \lambda_i^2 \wedge 1)Z_i^2\Big), \label{EqUpperBoundEarlyStopStrongError}
\end{align}
where $Z_i \coloneqq \scapro{Z}{u_i}$, $i=1,\dots,D$.
For $\tau=0$, the bound \eqref{EqUpperBoundEarlyStopStrongError} continues to hold due to $A_{\tau,\lambda} \wedge S_{\tau,\lambda}=0$.
Choosing $\rho \thicksim (R\delta^{-1})^{4p/(4\mu p +2p+1)}$ gives the right order of the first summand. For the second summand, we obtain similarly to~\eqref{EqBoundExpWeakStochError}
\begin{equation*}
\sum_{i=1}^D (\rho \lambda_i^2 \wedge 1)^2 \le (C_A \vee 1)^4 \sum_{i=1}^D (\rho^2 i^{-4p} \wedge 1) \le (C_A \vee 1)^4 \tfrac{4p}{4p-1} \rho^{1/(2p)}.
\end{equation*}
The $\chi^2$-concentration bound of \citet[Lemma 1]{LauMas2000} gives
\begin{equation*}
\PP \Big( \textstyle \sum_{i=1}^D (\rho \lambda_i^2 \wedge 1)(Z_i^2-1) \ge 2(C_A \vee 1)^2 \sqrt{\tfrac{4p}{4p-1}}\rho^{1/(4p)}\sqrt{z}+2z \Big) \le e^{-z}.
\end{equation*}
From $2(C_A \vee 1)^2\sqrt{\tfrac{4p}{4p-1}}\rho^{1/(4p)}\sqrt{z}+2z \le  (C_A \vee 1)^4\tfrac{4p}{4p-1} \rho^{1/(2p)} + 3z$
we deduce with \eqref{EqBoundExpWeakStochError}
\begin{align*}
\sum_{i=1}^D (\rho \lambda_i^2 \wedge 1)Z_i^2 &\le \sum_{i=1}^D (\rho \lambda_i^2 \wedge 1) + (C_A \vee 1)^4\tfrac{4p}{4p-1} \rho^{1/(2p)} + 3z \\ &\le (C_A \vee 1)^4\tfrac{4p}{2p-1} \rho^{1/(2p)} + 3z
\end{align*}
with probability at least $1-e^{-z}$. Plugging into \eqref{EqUpperBoundEarlyStopStrongError} for our choice of $\rho$ yields the claim.
\item Fix $z \ge 1$ and put $\Delta:=\abs{A_{\tau,\lambda} - S_{\tau,\lambda}}$ for short. In the case $\kappa \le \norm{Y}^2$, we have $R_{\tau}^2 = \kappa$ such that by \eqref{EqDiffWeakErrors}
\begin{equation*}
\Delta = \abs{\kappa - \norm{\xi}^2 - 2 \scapro{\xi}{\mathsf{r_{\tau,<}}g}}.
\end{equation*}
Otherwise, if $\kappa > \norm{Y}^2$, we stop at $\tau=0$ and
\begin{align*}
0 \le \Delta = A_{0,\lambda} = \norm{g}^2 &= \norm{Y}^2 + \norm{\xi}^2 -2 \scapro{\xi}{Y} \\
&= \norm{Y}^2 - \norm{\xi}^2 -2 \scapro{\xi}{\mathsf{r_{0,<}}g} < \kappa - \norm{\xi}^2 -2 \scapro{\xi}{\mathsf{r_{0,<}}g}.
\end{align*}
Hence, using the assumption on $\kappa$, we can in both cases bound
\begin{align}
\Delta
&\le \abs{\kappa - \norm{\xi}^2 - 2 \scapro{\xi}{\mathsf{r_{\tau,<}}g}} \notag \\
&\le \bigabs{\norm{\xi}^2 - \delta^2 D} + \abs{\kappa - \delta^2 D} + 2 \abs{\scapro{\xi}{\mathsf{r_{\tau,<}}g}} \notag\\
&\le \bigabs{\norm{\xi}^2 - \delta^2 D} + 2 \abs{\scapro{\xi}{\mathsf{r_{\tau,<}}g}} + C_{\kappa}\delta^2 \sqrt{D}. \label{EqFirstBoundKappa}
\end{align}
By the $\chi^2$-concentration bound of \citet[Lemma~1]{LauMas2000}, we have
\begin{equation}
\bigabs{\norm{\xi}^2 - \delta^2 D} \le 2 \delta^2\sqrt{D} \sqrt{z} +2\delta^2z \label{EqLaurentMassart}
\end{equation}
with probability at least $1-2e^{-z}$.

Let us first assume $\abs{r_{\tau}'(0)} < \lambda_D^{-2}$. \cref{Lemrkg} yields $\norm{\mathsf{r_{\tau,<}}g}^2 \le 8M_{\tau,\lambda}$.  We then obtain by \eqref{EqCrossTermBoundHighProbPSD}, \eqref{EqFirstBoundKappa} and \eqref{EqLaurentMassart}
\begin{align*}
&\delta^{-2}\Delta \\
&\lesssim_{p,c_A,C_A,C_{\kappa}} (\delta^{-2} M_{\tau,\lambda})^{1/2}\big((\delta^{-2}S_{\tau,\lambda})^{p/(2p+1)}+\sqrt{\log d}+\sqrt{z}\big) + \sqrt{D} \sqrt{z} + z \\
&\lesssim_{p,c_A,C_A,C_{\kappa}} (\delta^{-2}M_{\tau,\lambda})^{(4p+1)/(4p+2)} + (\delta^{-2} M_{\tau,\lambda})^{1/2} \big(\sqrt{\log d} +  \sqrt{z}\big) + \sqrt{D} \sqrt{z} + z
\end{align*}
with probability at least $1-4e^{-z}$. Choosing as conjugate exponents $\tilde{p} \coloneqq (4p+1)/(2p+1)$ and $\tilde{q} \coloneqq (4p+1)/(2p)$, Young's inequality implies
\begin{align*}
&(\delta^{-2} M_{\tau,\lambda})^{1/2} (\sqrt{\log d}+\sqrt{z}) \\
&\le 2(\delta^{-2}M_{\tau,\lambda})^{(4p+1)/(4p+2)}+(\log d)^{(4p+1)/(4p)}+z^{(4p+1)/(4p)}.
\end{align*}
We conclude
\begin{equation}
\delta^{-2}\Delta \lesssim_{p,c_A,C_A,C_{\kappa}} (\delta^{-2}M_{\tau,\lambda})^{(4p+1)/(4p+2)} + \sqrt{D}\sqrt{z} + z^{(4p+1)/(4p)} \label{EqBoundKappaStochErrorPolDecay}
\end{equation}
with probability at least $1-4e^{-z}$.
Hence, we have by \subcref{PropUpperBoundEarlyStopStrongError} and \eqref{EqBoundKappaStochErrorPolDecay}
\begin{align*}
\delta^{-2} M_{\tau,\lambda} \lesssim_{\mu,p,c_A,C_A,C_{\kappa}} &(R \delta^{-1})^{2/(4 \mu p + 2p + 1)} \\
&+ (\delta^{-2} M_{\tau,\lambda})^{(4p+1)/(4p+2)} + \sqrt{D} \sqrt{z} + z^{(4p+1)/(4p)}
\end{align*}
with probability at least $1-5e^{-z}$. Due to $R\delta^{-1}\ge 1$ and $(4p+1)/(4p+2)<1$, we arrive for $z\ge 1$ at
\begin{equation}\label{EqMtauBound}
\delta^{-2} M_{\tau,\lambda} \lesssim_{\mu,p,c_A,C_A,C_{\kappa}} (R \delta^{-1})^{2/(4 \mu p + 2p + 1)} + \sqrt{D} \sqrt{z} + z^{(4p+1)/(4p)}
\end{equation}
with probability at least $1-5e^{-z}$.

In the case $\abs{r_{\tau}'(0)} \ge \lambda_D^{-2}$, we have by \cref{LemrtBound}
\begin{equation*}
S_{\tau,\lambda}\ge \norm{\mathsf{(1-{\exp(-\lambda_D^{-2} x)})^{1/2}}\xi}^2\ge (1-e^{-1})\norm{\xi}^2
\end{equation*}
and bound directly
\begin{equation*}
2 \abs{\scapro{\xi}{\mathsf{r_{\tau,<}}g}}\le 2\norm{\xi}\norm{\mathsf{r_{\tau,<}}g} \le 2(1-e^{-1})^{-1/2} S_{\tau,\lambda}^{1/2}\norm{\mathsf{\exp(-\lambda_D^{-2}x)}g}.
\end{equation*}
In view of
\begin{equation*}
\sup_{x > 0} \exp(-2\lambda_D^{-2} x) x^{2\mu+1} = \big( (\mu +1/2)e^{-1} \big)^{2\mu + 1} \lambda_D^{4\mu + 2}
\end{equation*}
and the source condition $\sourcecondf$, we obtain
\begin{align*}
\norm{\mathsf{\exp(-\lambda_D^{-2}x)}g} &= \Big( \sum_{i=1}^D \big( \exp(-2 \lambda_D^{-2} \lambda_i^2) \lambda_i^{4\mu + 2} \big) \lambda_i^{-4\mu} f_i^2 \Big)^{1/2} \\
&\le \big( (\mu +1/2)e^{-1} \big)^{\mu + 1/2} \lambda_D^{2\mu + 1}R \lesssim_{\mu,C_A} R D^{-2\mu p - p}.
\end{align*}
Combining with \eqref{EqFirstBoundKappa} and \eqref{EqLaurentMassart} yields with probability at least $1-2e^{-z}$
\begin{equation*}
\delta^{-2}\Delta
\lesssim_{\mu,C_A,C_{\kappa}} (\delta^{-2} M_{\tau,\lambda})^{1/2}R\delta^{-1}D^{-2\mu p- p} + \sqrt{D} \sqrt{z} + z.
\end{equation*}
Using part \subcref{PropUpperBoundEarlyStopStrongError}, we deduce with probability at least $1-3e^{-z}$
\begin{equation*}
\delta^{-2} M_{\tau,\lambda} \lesssim_{\mu,p,C_A,C_{\kappa}}(R \delta^{-1})^{2/(4 \mu p + 2p + 1)} + (\delta^{-2} M_{\tau,\lambda})^{1/2}R\delta^{-1} D^{-2\mu p - p} + \sqrt{D} \sqrt{z} + z
\end{equation*}
and thus for $z\ge 1$
\begin{equation*}
\delta^{-2} M_{\tau,\lambda} \lesssim_{\mu,p,C_A,C_{\kappa}} (R \delta^{-1})^{2/(4 \mu p + 2p + 1)}+ R^2\delta^{-2} D^{-4\mu p - 2p} + \sqrt{D} \sqrt{z} + z.
\end{equation*}
From the assumption $D\gtrsim (R\delta^{-1})^{2/(4\mu p+2p+1)}$ we thus derive \eqref{EqMtauBound} again on the same event.
\qedhere
\end{enumerate}
\end{proof}

\begin{proof}[Proof of \cref{ThmEarlyStopStrongMinimax}]
Fix $z \ge 1$ and recall $M_{\tau,\lambda} \coloneqq A_{\tau,\lambda} \vee S_{\tau,\lambda}$. By \cref{PropBoundStrongErrorSource}, the inequality
\begin{alignat*}{2}
\norm{\widehat{f}_\tau-f^\dagger}^2 &\lesssim_{p,c_A} &&R^{2/(2\mu+1)}M_{\tau,\lambda}^{2\mu/(2\mu+1)}+\delta^{-4p}S_{\tau,\lambda}^{2p}M_{\tau,\lambda} \\
& &&+(\log d)^{2p}M_{\tau,\lambda} + z^{2p}M_{\tau,\lambda} \eqqcolon T_1+T_2+T_3+T_4
\end{alignat*}
holds true with probability at least $1-e^{-z}$. We control the summands using \cref{PropBoundKappa}. The following bounds hold simultaneously with probability at least $1-5e^{-z}$, where $\recrate$ is given in \eqref{EqDefRecRate}:

We have
\begin{align*}
T_1 \lesssim_{\mu,p,c_A,C_A,C_{\kappa}}
&\recrate \\
&+ R^{2/(2\mu+1)} \delta^{4\mu /(2\mu+1)} \big(\sqrt{D}\sqrt{z}+z^{(4p+1)/(4p)}\big)^{2\mu/(2\mu+1)}.
\end{align*}
Young's inequality with the conjugate exponents $\tilde{p}\coloneqq(2\mu + 1)(2 p +1)/(4 \mu p + 2p + 1)$ and $\tilde{q}\coloneqq(2\mu + 1)(2p+1)/(2\mu)$ yields
\begin{align*}
&\big(R^2 \delta^{-2}\big)^{1/(2\mu + 1)} \big(\sqrt{D}\sqrt{z}+z^{(4p+1)/(4p)}\big)^{2\mu/(2\mu+1)} \\
&\le \big(R^2 \delta^{-2}\big)^{(2p+1)/(4\mu p +2p+1)} + \big(\sqrt{D}\sqrt{z}+z^{(4p+1)/(4p)}\big)^{2p+1}
\end{align*}
such that
\begin{equation*}
T_1 \lesssim_{\mu,p,c_A,C_A,C_{\kappa}} \recrate + \delta^2 \big(\sqrt{D}\sqrt{z}+z^{(4p+1)/(4p)}\big)^{2p+1}.
\end{equation*}
Concerning $T_2$, we have
\begin{align*}
T_2 &\lesssim \delta^{-4p} M_{\tau,\lambda}^{2p+1} \\
&\lesssim_{\mu,p,c_A,C_A,C_{\kappa}} \recrate + \delta^2 \big(\sqrt{D}\sqrt{z}+z^{(4p+1)/(4p)}\big)^{2p+1}.
\end{align*}
Finally, by Young's inequality with $\tilde{p}\coloneqq (2p+1)/(2p)$ and $\tilde{q}\coloneqq 2p+1$, we deduce
\begin{equation*}
T_3 = \delta^{-4p}\big(\delta^2 \log d\big)^{2p} M_{\tau,\lambda}
\le \delta^{-4p}M_{\tau,\lambda}^{2p+1} + \delta^2 (\log d)^{2p+1}
\end{equation*}
and
\begin{equation*}
T_4 = \delta^{-4p} (\delta^2 z)^{2p} M_{\tau,\lambda} \le \delta^{-4p}M_{\tau,\lambda}^{2p+1} + \delta^2 z^{2p+1}.
\end{equation*}

We conclude the asserted high probability bound by combining all inequalities.
By \cref{LemHighProbExpBound}, the expectation bound follows directly from the high probability bound.

If $\sqrt{D} \lesssim (R^2 \delta^{-2})^{1/(4\mu p +2p+1)}$, then $\delta^2 D^{p+1/2} \lesssim \recrate$ holds and we obtain minimax adaptivity for the given range of regularity parameters $\mu$ by \cref{RemMinimaxReconstructionError}.
\end{proof}

\begin{proof}[Proof of \cref{PropDiscrepancyPrincipleOrderOpt}]
In the case $c\delta \le \norm{Y}$, we have $R_{\tau} = c\delta$ and \cref{PropWeakMinimaxDN} implies
\begin{equation*}
c \delta = R_{\tau} \le \delta + R(\mu+1/2)^{\mu+1/2} \abs{r_{\tau}'(0)}^{-\mu-1/2}.
\end{equation*}
Since $c>1$, this yields $\abs{r_{\tau}'(0)} \lesssim_{\mu,c} (R^2 \delta^{-2})^{1/(2\mu+1)}$. By combining \eqref{EqRoughBoundStrongError}--\eqref{EqInterpolIneq}, we have
\begin{equation*}
\norm{\widehat{f}_{\tau}-f^\dagger}^2 \lesssim \norm{\mathsf{r_{\tau,<}}g}^{4\mu/(2\mu+1)}R^{2/(2\mu+1)} + \abs{r_{\tau}'(0)}S_{\tau,\lambda} + \abs{r_{\tau}'(0)} \norm{A(\widehat{f}_{\tau}-f)}^2.
\end{equation*}
Noting that $\norm{\mathsf{r_{\tau,<}}g} \vee \norm{A(\widehat{f}_{\tau}-f)} \le R_{\tau} + \norm{\xi} \le (c+1)\delta$ and $S_{\tau,\lambda} \le \delta^2$ under \eqref{AssDN}, we obtain the claim by plugging in the bound on $\abs{r_{\tau}'(0)}$. Otherwise, if $c \delta > \norm{Y}$, we stop at $\tau = 0$ and obtain similar to \eqref{EqInterpolIneq}
\begin{equation*}
\norm{\widehat{f}_{\tau}-f^\dagger}^2 = \norm{f^\dagger}^2 \le \norm{g}^{4\mu / (2\mu + 1)} R^{2/(2\mu + 1)}.
\end{equation*}
Plugging in $\norm{g} \le \norm{Y}+\norm{\xi} < (c+1)\delta$ yields the claim.
\end{proof}

\subsection{Auxiliary result}

\begin{lemma} \label{LemHighProbExpBound}
Let $X$ be a real-valued random variable and $p\ge 1$ such that
\begin{equation*}
\PP(X \ge z^p) \le C e^{-z}
\end{equation*}
for all $z>c$ with  absolute constants $c,C \ge 0$. Then
\begin{equation*}
\E[X_+] \le c^p + Cp \E[Z^{p-1}],
\end{equation*}
where $Z\sim\Exp(1)$ is an exponentially distributed random variable with parameter $1$.
\end{lemma}

\begin{proof}
We have
\begin{align*}
\E[X_+] &\le c^p + \E \Big[ \int_{c^p}^\infty \mathbf{1}(u \le X) du \Big] = c^p + \E \Big[ \int_c^\infty \mathbf{1}(z^p \le X) p z^{p-1} dz \Big] \\
&= c^p + p \int_c^\infty z^{p-1} \PP(X \ge z^p) dz \le c^p + Cp \int_0^\infty z^{p-1} e^{-z} dz = c^p + Cp \E[Z^{p-1}]. \qedhere
\end{align*}
\end{proof}


\paragraph{Acknowledgements}

We are grateful for comments by Tim Jahn that helped to improve the paper. Financial support by the Deutsche Forschungsgemeinschaft (DFG) through the research unit FOR 5381 \emph{Mathematical Statistics in the Information Age -- Statistical Efficiency and Computational Tractability} is gratefully acknowledged.

\bibliography{bibliography}

\end{document}